\documentclass[12pt,a4paper,final]{article}

\usepackage[a4paper,margin=25mm]{geometry}

\usepackage{graphicx}

\usepackage[color,notcite,notref]{showkeys} 
\definecolor{refkey}{gray}{.5}   
\definecolor{labelkey}{rgb}{.7,.2,.7} 

\usepackage{pdfsync,paralist,tikz}
\usepackage{graphpap,bm,color,hyperref}
\usepackage{amsmath,amsthm,amssymb,mathrsfs,enumerate}
\usepackage[normalem]{ulem}
\usepackage[mathcal]{euscript}

\numberwithin{equation}{section}
\numberwithin{figure}{section}

\newtheorem{theorem}{Theorem}[section]
\newtheorem{lemma}[theorem]{Lemma}
\newtheorem{definition}[theorem]{Definition}
\newtheorem{proposition}[theorem]{Proposition}
\newtheorem{corollary}[theorem]{Corollary}
\newtheorem{remark}[theorem]{Remark}
\newtheorem{example}[theorem]{Example}

\newtheorem{notation}[theorem]{Notation}

\newcommand{\N}{\mathbb{N}}

\newcommand{\R}{\mathbb{R}}

\newcommand{\EE}{\mathscr{E}}

\newcommand{\cE}{\mathcal E}

\newcommand{\cM}{\mathcal M}

\def\calD{{\mathcal D}}  \def\calE{{\mathcal E}} \def\calF{{\mathcal F}}
\def\calG{{\mathcal G}}  
\def\calI{{\mathcal I}} \def\calJ{{\mathcal J}} 
\def\calM{{\mathcal M}} \def\calP{{\mathcal P}} 
\def\calR{{\mathcal R}} \def\calW{{\mathcal W}} 

\def\rmC{{\mathrm C}} \def\rmD{{\mathrm D}} 

\newcommand{\sft}{{\mathsf t}}

\newcommand{\rmd}{\mathrm{d}}
\newcommand{\dd}{\,\mathrm{d}}

\renewcommand{\div}{\mathop{\mathrm{div}}\nolimits}

\newcommand{\weakto}{\rightharpoonup}
\newcommand{\graph}{\mathop{\mathrm{graph}}}

\newcommand{\el}{\mathrm{el}}

\newcommand{\argmin}{\mathop\mathrm{Argmin}}
\newcommand{\pairing}[4]{\sideset{_{#1}}{_{#2}}{\mathop{ 
                          \langle #3 , #4  \rangle}}}
\newcommand{\down}{\downarrow}
\newcommand{\up}{\uparrow}
\newcommand{\foraa}{\text{for a.a.\ }}

\newcommand{\ene}[2]{\calE(#1,#2)}

\newcommand{\diff}[2]{\diffname(#1,#2)}
\newcommand{\diffe}[3]{\diffname_{#1}(#2,#3)}
\newcommand{\cg}[1]{\calG(#1)}
\newcommand{\diffname}{\mathfrak{F}}
\newcommand{\domain}{\mathrm{D}}
\newcommand{\dom}{\mathop{\mathrm{dom}}}
\newcommand{\BV}{\mathrm{ BV}}
\newcommand{\AC}{\mathrm{AC}}

\newcommand{\piecewiseConstant}[2]{\overline{#1}_{\kern-1pt#2}}
\newcommand{\underpiecewiseConstant}[2]{\underline{#1}_{\kern-1pt#2}}
\newcommand{\piecewiseLinear}[2]{#1_{\kern-1pt#2}}
\newcommand{\pwc}{\piecewiseConstant}
\newcommand{\upwc}{\underpiecewiseConstant}
\newcommand{\pwl}{\piecewiseLinear}
\newcommand{\pwm}[2]{\widetilde{#1}_{\kern-1pt#2}}

\newcommand{\SE}{S}
\newcommand{\Ptname}{\mathfrak{P}}
\newcommand{\Pt}[3]{\Ptname(#1,#2,#3)}
\newcommand{\Rtname}{\mathfrak{R}}
\newcommand{\Rt}[3]{\Rtname(#1,#2,#3)}
\newcommand{\enei}[3]{\cE_{#1}(#2, #3)}
\newcommand{\frsub}{\partial}
\newcommand{\phimin}{\varphi}
\newcommand{\phiminn}{\varphi_{n}}
\newcommand{\transp}[1]{{#1}^{\mathsf{T}}}
\newcommand{\transpi}[1]{{#1}^{\mathsf{-T}}}
\newcommand{\inve}[1]{{#1}^{-1}}
\newcommand{\qphi}{{q_{\Phi}}}
\newcommand{\qp}{{q_\mathsf P}}

\newcommand{\qg}{{q_{\mathsf G}}}
\newcommand{\qf}{{q_\mathsf F}}
\newcommand{\qga}{{q_{\gamma}}}
\newcommand{\en}{\calE}
\newcommand{\cof}{\mathrm{cof}}
\newcommand{\amb}{X}
\newcommand{\ambs}{X^*}
\newcommand{\normam}[1]{\|#1\|_{\amb}}
\newcommand{\normams}[1]{\|#1\|_{\ambs}}
\newcommand{\grasys}{(X,\en,\Psi,\diffname,\Ptname)}
\newcommand{\grasyse}[1]{(X,\en_{#1},\Psi,\diffname_{#1},\Ptname_{#1})}

\newcommand{\gateaux}[3]{\Manx{#1}{\nabla#2 \inve{#3}}\transpi{#3}}

\newcommand{\argut}[2]{(\nabla#1 \inve{#2})}
\newcommand{\GLD}{\mathrm{GL}^+(d)}
\newcommand{\rmetric}{R}
\newcommand{\rdiss}{\mathcal{R}}
\newcommand{\rpam}{\eta}
\newcommand{\param}{\eta}
\newcommand{\Man}[1]{M(#1)}
\newcommand{\Manx}[2]{M(#1,#2)}
\newcommand{\mins}[2]{\mathcal{M}(#1,#2)}
\newcommand{\minse}[3]{\mathcal{M}_{#1}(#2,#3)}
\newcommand{\rcalF}{\widetilde{\calF}}
\newcommand{\dell}{\dot{\ell}}
\newcommand{\gdir}{g_{\mathrm{Dir}}}
\newcommand{\dir}{\mathrm{Dir}}
\newcommand{\neu}{\mathrm{Neu}}
\newcommand{\calfid}{\mathcal Y}

\newcommand{\Kirchx}[2]{\mathbb K(#1,#2)}
\newcommand{\Kirchname}{S}
\newcommand{\Back}[2]{B(#1,#2)}
\newcommand{\Backx}[3]{B(#1,#2,#3)}
\newcommand{\inpow}[4]{\Kirchx{#1}{\nabla{#2} (t,#3)\nabla {#3}{#4}}}
\newcommand{\OOdd}{} 
\newcommand{\OOd}{} 

\newcommand{\sfti}[2]{\mathsf{t}_{#1}^{#2}}
\newcommand{\PPt}[2]{P_{#1}^{#2}}

\definecolor{ddmagenta}{rgb}{0.7,0,1.0}
\definecolor{ddcyan}{rgb}{0,0.1,1.0}
\definecolor{dred}{rgb}{.8,0,0}
\definecolor{ddgreen}{rgb}{0,0.4,0.4}
\definecolor{ALEXgreen}{rgb}{0.1,0.7,0.1}
\newcommand{\address}[1]{\thanks{#1}}
\newcommand{\email}[1]{\texttt{#1}}

\begin{document}

\title
{Global existence results for\\ viscoplasticity at finite strain%
 \thanks{A.M. has been partially
  supported by the ERC under AdG 267802 AnaMultiScale and by DFG
  under SFB\,1114, project C5. 
  R.R. and G.S. have been partially supported by a MIUR-PRIN'10-11 grant
  for the project ``Calculus of Variations''. R.R.\  also acknowledges
  support from the  Gruppo Nazionale per  l'Analisi Matematica, la
  Probabilit\`a  e le loro Applicazioni (GNAMPA) of the Istituto
  Nazionale di Alta Matematica (INdAM).}
} 

\date{September 18, 2016}

\author{
Alexander Mielke%
  \address{Weierstra\ss-Institut f\"ur Angewandte Analysis und Stochastik,
  Mohrenstra\ss{}e 39,  D--10117 Berlin, Germany.
  \email{alexander.mielke\,@\,wias-berlin.de}}
,
Riccarda Rossi%
\address{DIMI, Universit\`a di
  Brescia, via Valotti 9, I--25133 Brescia, Italy.
  \email{riccarda.rossi\,@\,unibs.it}}
,
Giuseppe Savar\'e%
  \address{Dipartimento di Matematica ``F.\ Casorati'', Universit\`a di
  Pavia.  Via Ferrata, I--27100 Pavia, Italy.
  \email{giuseppe.savare\,@\,unipv.it}}
}

\maketitle


\begin{abstract}
  We study a model for rate-dependent gradient plasticity at finite
  strain based on the multiplicative decomposition of the strain
  tensor, and investigate the existence of global-in-time
  solutions to the related PDE system.  We reveal its underlying
  structure as a \emph{generalized gradient system}, where the driving
  energy functional is highly nonconvex and features the geometric
  nonlinearities related to finite-strain elasticity as well as
  the multiplicative decomposition of finite-strain
  plasticity. Moreover, the dissipation potential depends on the
  left-invariant plastic rate and thus, depends on the plastic state
  variable.
 
 The existence theory is developed for a class of abstract, nonsmooth, 
 and nonconvex gradient systems, for which we introduce suitable
 notions of solutions, namely energy-dissipation-balance (EDB) and
 energy-dissipation-inequality (EDI) solutions. Hence, we resort
 to the toolbox of the direct method of the calculus of
 variations to check that the specific energy and 
 dissipation functionals for our viscoplastic models 
 comply with the conditions of the general theory.  
\end{abstract}

\section{Introduction}
\label{s:Intro}
This paper is focused on the analysis of a model for elastoplasticity
at finite strain in a bounded body $\Omega\subset \R^d$.  Its elastic
behavior is described by the \emph{deformation} $\varphi: \Omega \to
\R^d$. The ansatz at the core of finite-strain elastoplasticity, see
\cite{Lee69EPDF}, is the \emph{multiplicative} decomposition of the
deformation gradient into an \emph{elastic} and a \emph{plastic} part,
namely
\begin{equation}
 \label{decomposition-intro}
 \nabla \varphi = F_{\el } F_{\mathrm{pl}} \doteq F_{\el } P,
\end{equation}
which reflects the composition of elastic and plastic deformations.
While the elastic part contributes to energy storage and is governed
by an equilibrium equation, the plastic tensor $P$ evolves by a
plastic flow rule.

The multiplicative decomposition \eqref{decomposition-intro} leads to
significant geometric nonlinearities in the energy functional $\calI =
\calI(t,\varphi, P)$ driving the evolution of the elastoplastic
process. Nonetheless, such nonlinearities are compatible with
\emph{polyconvexity} of the energy density. In fact, the theory of
polyconvex materials, dating back to \cite{Ball76CCET}, has provided
the analytical toolbox to handle \emph{elastostatics} at finite
strain. In particular, for the analysis of static microstructures in
elastoplasticity we refer to, e.g., \cite{MulAif91VPGP, OrtRep99NEMD,
  OrReSt00TSDS, CaHaMi02NCPM, Mieh03CMMT, ConThe05SSEM, ConOrt05DMEB}.

A fundamental step towards the analysis of the \emph{evolution} of
finite-strain elastoplastic materials was made in \cite{OrtSta99VFVC}:
therein it was pointed out that evolutionary elastoplastic models can
be discretized by time-incremental problems that can be written as
minimization problems for the sum of  the  dissipated and of the stored
energy.  This observation was mathematically formalized in
\cite{Miel03EFME} and forms the basis of the analysis in
\cite{MieMul06LSEM,MaiMie09GERI,HaHeMi12MELF}, where the flow rule for
the plastic tensor was considered \emph{rate-independent}, i.e.\
driven by a dissipation potential that is positively
homogenous of degree $1$. In this context, the existence of
\emph{(global) energetic solutions} was proved by passing to
the limit in a time-discretization scheme in the spirit of
\cite{OrtSta99VFVC}.  In fact, the analysis developed in
\cite{FraMie06ERCR,MaiMie09GERI,MieRou15?RIEF} shows that this
solution concept for rate-independent processes is markedly suited to
the geometric nonlinearities in finite-strain
elastoplasticity. Essentially, its intrinsic \emph{variational}
character allows for the successful application of the direct
methods in the calculus of variations also in the evolutionary
case. Other mathematical results for energetic solutions of
rate-independent material models at finite strain include crack
propagation and brittle fracture, see
\cite{KnZaMi10CGPM,DalLaz10QCGF}. 
 
In this paper we consider finite-strain elastoplasticity as a
\emph{rate-dependent} process, i.e.\ involving a dissipation potential
with \emph{superlinear} growth at infinity. A natural choice is of
course given by a \emph{quadratic} dissipation potential, leading to a
classical gradient flow. In fact, our analysis builds on the
by now well-established \emph{variational theory} for gradient flows,
see \cite{Ambr95MM, AmGiSa05GFMS}, and \emph{generalized gradient
systems} \cite{RoMiSa08MACD, MiRoSa13NADN}, which are characterized by
nonquadratic dissipation potentials. In particular, along the
footsteps of \cite{MiRoSa13NADN} we will address two (intrinsically
variational) solution concepts for the elastoplastic system, based on
a suitable energy-dissipation inequality, which holds as an equality
in the case of the strongest notion. Accordingly, we will obtain two
distinct existence results.

Before illustrating our analysis more in detail, let us gain further
insight into the features of the elastoplastic model under
investigation, and in particular into the highly nonlinear driving
energy.

\subsection{Modeling of viscoplasticity}
\label{su:1.1}
The elastoplastic evolution of  the body $\Omega$ is described by two
variables. The deformation $\varphi:\Omega \to \R^d$ is a mapping such
that for almost all $x\in \Omega$ the gradient $\nabla \varphi(x)$
exists and belongs to the general linear group $\GLD$ of $(d{\times}
d)$-matrices with positive determinant. Following the theory
of \emph{generalized standard materials} \cite{HalNgu75MSG, Maug92TPF,
  Frem02NST}, we consider the plastic tensor $P:= F_{\mathrm{pl}} \in
\mathcal{P} \subset \GLD$, cf.\ \eqref{decomposition-intro}, as an
internal variable, modeling the internal state of the body. As such it
is a macroscopic variable (we do not resolve the atomistic length),  
which is  assumed to be generated by movements of dislocations, and
maps the material frame (i.e.\ the crystallographic lattice) onto
itself.

The evolution is governed by two principles: 
\begin{description}
\item[\emph{Energy storage}] via a time-dependent Gibbs' free energy $
\calI (t,\varphi, P)$ and 
\item[\emph{Energy dissipation}] via a dissipation potential
  $\widehat\Psi(\varphi,P,\dot\varphi,\dot P)$.
\end{description}
We assume that inertial effects can be ignored (quasistatic
approximation) such that the equations of interest take the
abstract variational form 
\begin{subequations}
 \label{eq:Abstr}
 \begin{align}
  &\label{eq:Abstr.a} 0 = \rmD_{\dot\varphi}
   \widehat\Psi(\varphi,P,\dot\varphi,\dot{P}) + \rmD_\varphi \calI(t,\varphi,P), 
  \\
  &0 \,\in\, \partial_{\dot P}\widehat\Psi(\varphi,P,\dot\varphi,\dot{P})
    + \rmD_P \calI(t,\varphi,P).  
  \label{eq:Abstr.b}
\end{align} 
\end{subequations}
Here all derivatives should be taken as variational derivatives, and
since $\calJ$ and $\widehat\Psi$ depend, except for potential external
loadings, on the gradients $\nabla \varphi$ and $\nabla\dot\varphi$,
the first equation \eqref{eq:Abstr.a} has the usual divergence form, i.e.\
$\rmD_{\dot\varphi}\widehat\Psi$ contains the divergence of the viscous
stresses while $\rmD_\varphi\calI$ contains the divergence of the elastic 
stress tensor, see e.g.\ \cite{MarHug94MFE,Antm95NPE,MiOrSe14ANVM}. Hence,
\eqref{eq:Abstr.a} provides the balance of linear momentum, whereas 
\eqref{eq:Abstr.b} contains the plastic flow rule. The term $\rmD_P
\calI$ contains the plastic 
backstress and the convex subdifferential 
$\partial_{\dot P}\widehat\Psi$ contains the viscoplastic
stresses; in particular for $\dot P=0$ it features the yield stress. We refer  
to \cite{Miel11FTDM} for general modeling background, even including
a thermodynamically consistent modeling of temperature effects.  

The theory of viscoelasticity at finite strain is notoriusly difficult
and it seems that the present mathematical tools are not sufficient to
provide sufficiently strong solutions in the multidimensional, truly
geometrically invariant case, see the discussion in
\cite{MiOrSe14ANVM}. Hence, we will neglect viscous effects
subsequently by using the dissipation potential
$\widehat\Psi(\varphi,P,\dot\varphi,\dot P)=\Psi_P(\dot P)$, which
leads to a static equation for the displacement $\varphi$. This gives
us the opportunity to replace the stationarity condition $\rmD_\varphi
\calI(t,\varphi,P)=0$ by the global minimality
\begin{equation}
 \label{eq:Abstr.c}
  \varphi(t) \in \mathrm{Argmin}\{ \, \calI(t,\widetilde \varphi, P(t))
  \: : \ \widetilde \varphi \in \calF\,\}.   
\end{equation}
Thus, the starting point of our analysis is the
system  \eqref{eq:Abstr.c} and \eqref{eq:Abstr.b} for the pair
$(\varphi,P):[0,T] \to \calF\times X$. However, the differential
inclusion \eqref{eq:Abstr.b} has be replaced by a slightly weaker
notion of solution, which we will call EDI solution or EDB solutions,
see  below.

The stored energy $\calI$ and the dissipation potential $\Psi$ take
the form
\begin{equation}
\label{gibbs-intro}
\begin{aligned}
&\calI (t,\varphi, P)  = \int_\Omega \calW (x,\nabla \varphi(x),
P(x),\nabla P(x)) \dd x - \pairing{}{}{\ell(t)}{\varphi}
\\
& \text{and } \Psi_P(\dot P)=\int_\Omega \calR(x,P(x),\dot
P(x))\dd x,   
\end{aligned}
\end{equation}
where $\ell$ is a sufficiently smooth time-dependent loading,  see
\eqref{loading}.  For a given Banach space $X$, we continue to use
$\pairing{}{X}{\cdot}{\cdot}$ for the dual pairing on $X^*\times X$.   

The energy density $\calW$ and the pointwise dissipation potential
$\calR$ feature geometric nonlinearities arising from frame
indifference,  non-self-interpenetration,  and the Lie group structure of
finite strains like the multiplicative decomposition
\eqref{decomposition-intro}.  More precisely, this is 
reflected in a series of \emph{invariance principles} for the energy
density $\calW $ and $\calR$. First of all, we assume that
$\calW $ is given by the sum of an elastic part and of a part
encompassing hardening and regularizing (through the gradient of the
plastic variable) terms, i.e.
\begin{equation}
\label{W-decomposition}
\calW(x, \nabla \varphi, P, \nabla P) = \calW_{\el } (x,\nabla
\varphi, P) +  H  (x, P, \nabla P)\,. 
\end{equation}

Following \cite{MarHug94MFE,Antm95NPE,Miel03EFME},  the elastic
part $W_{\el }$ has to satisfy \emph{spatial frame indifference}
(or \emph{objectivity}), namely
\begin{subequations}
\label{invariance-principles}
\begin{align}
&
\label{invar-1}
\calW_{\el } (x,Q\nabla \varphi,P) = \calW_{\el }
(x,\nabla \varphi,P) \quad \text{for all } Q \in \mathrm{SO}(d), 
\end{align}
i.e.\ invariance with respect to rotations acting from the left, which
is compatible with polyconvexity together with the condition that
$\calW_{\el } (\nabla \varphi) = \infty$ for
$\mathrm{det}(\nabla \varphi)\leq 0$, and $\calW_{\el } (\nabla
\varphi) \to \infty$ for $\mathrm{det}(\nabla \varphi)\down 0$.

In addition, we postulate for $\calW_{\el } $ and
$\mathcal{R}$ \emph{plastic indifference} (cf.\
\cite{Miel03EFME}), namely 
\begin{align}
\label{invar-2}
\calW_{\el } (x,\nabla \varphi \widetilde{P}, P
\widetilde{P} ) = \calW_{\el } (x,\nabla \varphi,P), \ \
\mathcal{R} (x,P \widetilde{P}, \dot{P} \widetilde{P} ) = \mathcal{R}
(x,P, \dot{P})\quad \text{for } \widetilde{P} \in
\mathcal{P}\,.
\end{align}
\end{subequations}
This axiom implies that both $\calW _{\el } $ and
$\mathcal{R}$ can be written in reduced form as
\begin{equation}
\label{reduced-densities}
\calW_{\el} (x, \nabla \varphi, P) = W(x,\nabla \varphi
P^{-1}), \qquad  \mathcal{R}  (x,P, \dot{P}) = R(x, \dot{P}
P^{-1})\,. 
\end{equation}

These multiplicative structures give rise to strong 
geometric nonlinearities.  This is easily seen when writing the
PDE system induced by \eqref{eq:Abstr.c} and \eqref{eq:Abstr.b}
explicitly, 
namely 
\begin{subequations}
\label{pde}
\begin{align}
&
\label{pde-1}
 \varphi(t) \in \argmin\left \{ \int_{\Omega}
 W(x,\nabla \widetilde{\varphi}(x) P^{-1}(t,x)) \dd x - \langle
 \ell(t), \widetilde\varphi \rangle 
 \, : \ \widetilde{\varphi}
\in \calF \right \}, 
\\
& 
\label{pde-2}
\begin{aligned} \partial \rmetric(x,\dot{P}   \inve{P})\transpi{P}  
+  \transp{(\nabla \varphi P^{-1})} \rmD_F W(x,\nabla \varphi P^{-1})
\transpi{P}\ &\\   
+\, \rmD_P H(P,\nabla P)- \div\big(\rmD_{\nabla P}
H(P,\nabla P)\big) 
&=0,
\end{aligned}
\end{align}
\end{subequations}
where $\calF$ denotes the set of admissible deformations. Clearly,
on the formal level \eqref{pde-1} yields solutions to
\eqref{eq:Abstr.a}, as the latter is the Euler-Lagrange equation for
the minimum problem in \eqref{pde-1}. We refer to
\cite[Sec.\,4]{ZRSRZ96TDIT} for an engineering application of such a
finite-strain viscoplastic model, where $R(V)=\sigma_\text{yield}|V|+
c|V|^{1.012}$, see also Example \ref{ex:R}.

\subsection[Variational approaches: rate-independent versus
\protect rate-dependent evolution]{Variational approaches:
  rate-independent versus \protect \\  rate-dependent evolution}\label{su:1.2}

Variational approaches and formulations are ideal for treating 
material models involving finite-strain elasticity and finite-strain
plasticity. The reason is that the direct methods from the calculus of
variations rely on the flexible concept of weak lower semicontinuity, which 
allows us to circumvent the much too strong convexity methods that are
available for small-strain theories. 

A first global existence result for finite-strain elastoplasticity
was obtained  in \cite{MaiMie09GERI}, where solvability of \eqref{pde} was
addressed in the case of rate-independent systems,  i.e.\ when
$\widetilde{R}$ (and thus $R$) fulfills $\widetilde{R}(x,P,\lambda
\dot{P}) =\lambda \widetilde{R}(x,P, \dot{P}) $ for every \hbox{$\lambda
\geq 0$}. Indeed, the authors proved the existence of \emph{(global)
  energetic solutions} to \eqref{pde} according to  the energetic concept 
(cf.\ \cite{MieRou15RIST} for a general exposition), where
$(\varphi,P):[0,T] \to \calF\times X$ has to satisfy 
the \emph{(global) stability condition} \eqref{stab-intro} and the
\emph{energy balance} \eqref{en-bal-intro} for all $t\in [0,T]$:  
\begin{align}
  \label{stab-intro}\tag{$\mathrm{S}$}
 &\calI(t,\varphi(t), P(t)) \leq \calI(t,\widetilde{\varphi},
 \widetilde{P}) + \mathcal{D}(P(t),\widetilde{P}) \quad \text{for all }
 (\widetilde{\varphi} ,\widetilde{P}) \in \calF  \times X, 
 \\
  \label{en-bal-intro}\tag{$\mathrm{E}$} 
 &\calI(t,\varphi(t), P(t)) +
 \mathrm{Diss}_{\mathcal{D}} (P; [0,t]) = \calI(0,\varphi(0), P(0)) +
 \int_0^t \partial_t \calI(s,\varphi(s), P(s)) \dd s.
\end{align}
Here $\calD$ is an extended distance suitably defined from the
$1$-homogeneous dissipation $R$, and $ \mathrm{Diss}_{\calD} $
the induced dissipation functional. The most convenient feature of the
energetic formulation via \eqref{stab-intro} and \eqref{en-bal-intro}
is that it neither involves the pointwise derivative $\dot P$ of the
plastic tensor, which is only $\mathrm{BV}$ as a function of time, nor
any differential ``$\rmD\calJ$'' of the energy $\calI$. This is
extremely advantageous in view of the highly nonlinear and nonsmooth
character of $\calI(t,\cdot,\cdot)$ and the fact that $(X,\calD)$
needs to be treated as a metric space not relying on a linear Banach
space structure. 

The present work is devoted to the \emph{rate-dependent case}, but we
still rely on the variational structure given  by  \eqref{eq:Abstr.b} and
\eqref{eq:Abstr.c}, which form the abstract version of \eqref{pde}. 
Our system is induced by a \emph{generalized gradient system}
$(X,\calE,\Psi)$ with the \emph{reduced energy} $\calE$ and the
state-dependent dissipation potential $\Psi$ given by
\begin{equation}
\label{reduced-energy-intro} \en (t,P) : = \inf\{ \calI(t,\varphi,P)\, : \
\varphi \in \calF\} \quad \text{and} \quad 
\Psi_{P}(\dot P) : =\int_\Omega R(\dot{P}P^{-1}) \dd x\,.
\end{equation} 
Then, \eqref{pde} can be rewritten as the abstract subdifferential
inclusion
\begin{equation}
\label{abstract-dne-intro}
0 \ \in\ \partial \Psi_{P(t)}(\dot P(t)) + \diff t{P(t)}  \quad
\text{in $\amb^*$} \ \ \foraa   t \in (0,T),
\end{equation}
where the state space $\amb$ is $L^p(\Omega;\R^{d\times d })$, and
the multivalued operator $\diffname: [0,T]\times \amb
\rightrightarrows \amb^*$ is the \emph{marginal subdifferential} of
$\calE$ (cf.\ \cite{MiRoSa13NADN}) defined via
\[
 \diff t{P}: = \{ \rmD_{P} \calI(t,\varphi,P)\, : \ \varphi \text{ is a
  minimizer for \eqref{reduced-energy-intro}}\}\,.
\]

Denoting by $\Psi^*_P$ the Fenchel-Moreau conjugate
$\Psi^*_P(\Xi) = \sup\{ \pairing{}{X}{\Xi}{V}  {-} \Psi_P(V)\,:\ V\in
X\,\}$ of $\Psi_P$, the celebrated Fenchel equivalence states that 
\begin{equation}
  \label{eq:Fenchel}
 -\Xi \in \partial \Psi_P(\dot P) \quad \Longleftrightarrow \quad 
\Psi_P(\dot P)+ \Psi^*_P(-\Xi)=
-\pairing{}{X}{ \Xi}{\dot P}. 
\end{equation}
Thus, together with the chain rule
\begin{equation}
\label{ch-rule-intro}
\frac{\dd }{\dd t } \en(t,P(t)) = \partial_t \en(t,P(t))  +\!
\pairing{}{\amb}{\Xi(t)}{\dot{P}(t)} \text{ for all } \Xi(t) \in
\diff t{P(t)} \  \foraa   t \in (0,T) 
\end{equation}
(cf.\ Section \ref{s:3} for details, in particular for the 
treatment of the nonsmoothness of $t\mapsto \calE(t,p)$ induced by
\eqref{reduced-energy-intro}) we see that the pointwise
subdifferential inclusion \eqref{abstract-dne-intro} is 
equivalent to the \emph{energy-dissipation balance} 
\begin{equation}
  \label{eq:Int.EDB}
 \text{(EDB)}\qquad \left\{\begin{aligned}
 &\en(T,P(T)) + \int_0^T\! \left( \Psi_{P(t)}(\dot P(t)) + \Psi_{P(t)}^*
   (-\Xi(t)) \right) \dd t  \\ &= 
\en(0,P(0)) + \int_0^T\!\! \partial_t \en(t,P(t)) \dd t\,, \ 
 \text{ where } \Xi(t) \in \diff{t}{P(t)}.
\end{aligned}\right.
\end{equation}
In \cite{Miel16EGCG} this
equivalence is called the \emph{energy-dissipation principle}, which has
its origin in De Giorgi's theory of curves of maximal slope, cf.\
\cite{DeMaTo80PEMS}. 

 Observe that both energy identities
\eqref{en-bal-intro} and \eqref{eq:Int.EDB} balance the
stored energy $\calE$ with the work of the external forces and the
dissipated energy. However, \eqref{eq:Int.EDB} features the
additional dissipative term $\int_0^T \Psi_{P(r)}^* (-\Xi(r)) \dd r $
(which is in fact null in the rate-independent case), involving the
force term $\Xi(t)  =\rmD_P \calI(t,\varphi(t), P(t))$ that needs to
be suitably handled. 

However, since the conjugate pair $(\Psi_P,\Psi^*_P) $ always
satisfies the Young-Fenchel inequality $\Psi_P(V)+\Psi^*_P( \Xi )\geq  \pairing{}{X}{ \Xi }{V} $  for all $V\in X$ and $ \Xi  \in X^*$, it is in fact
sufficient to ask for the estimate ``$\leq$'' in the right-hand side
of \eqref{eq:Fenchel}. Thus, it suffices to  replace the EDB by the 
weaker \emph{energy-dissipation inequality}:
\begin{equation}
  \label{eq:Int.EDI}
  \text{(EDI)}\qquad \left\{ \begin{aligned} 
 &\text{For }s=0 \text{ and a.a.\ }s\in (0,T] \text{ and all }t\in (s,T]
  \text{ we have }
 \\
 &\en(t,P(t)) + \int_s^t\! \left( \Psi_{P(r)}(\dot P(r)) + \Psi_{P(r)}^*
   (-\Xi(r)) \right) \dd r 
 \\ 
 & \leq \en(s,P(s)) + \int_s^t\!\! \partial_r \en(r,P(r)) \dd r\,, \ 
   \text{ where } \Xi(r) \in \diff{r}{P(r)}.
\end{aligned}\right.
\end{equation}
Assuming that the chain rule \eqref{ch-rule-intro} holds, an easy
application of the Young-Fenchel inequality gives the equivalence of
EDI and EDB, see Proposition \ref{prop:weak-to-en}.  However, in
situations where the chain rule does not hold, EDI is strictly weaker.
Thus, we can now define three different types of solutions for
$P:[0,T] \to X$.

\begin{description}\itemsep-0.2em
\item[DNE solution:] $P$ satisfies the doubly nonlinear eqn.\ 
  $0\in \partial\Psi_P(\dot P) + \diff tP$, see
  \eqref{abstract-dne-intro}; 
\item[EDI solution:] $P$ satisfies the EDI  \eqref{eq:Int.EDI};
\item[EDB solution:] $P$ satisfies the EDB \eqref{eq:Int.EDB} and 
the doubly nonlinear eqn.\ \eqref{abstract-dne-intro}.
\end{description}
The exact form of EDI and EDB solutions are given in Definitions
\ref{def:weak-sols} and \ref{def:energy-sols}, respectively; see
Figure \ref{fig:Sol} for the inclusion relations between the solution
types.  We will not address the notion of DNE solutions, which is
typically used for evolutionary system formulated with monotone
operators, since our approach involves 
the variational structure where the energy and dissipation are crucial
to obtain a priori bounds.  

\begin{figure}
{\centering\begin{tikzpicture}
\draw[fill = green!20,opacity=0.3] (0,0) ellipse (2 and 1);
  \node at (-1,0) {EDI}; 
\draw[opacity=0.3,fill=blue!20] (2,0) ellipse (2 and 1);
  \node at (3,0) {DNE};  
\draw[opacity=0.3,fill=red!40,opacity=0.5] (1,0) ellipse (0.6 and 0.6);
  \node at (1,0) {EDB};
\end{tikzpicture}\par}
\caption{Schematic graph of solution concepts}\label{fig:Sol}
\end{figure}
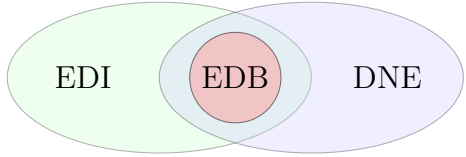

\subsection{The main  results}
\label{su:1.3}
For the analysis of the gradient system \eqref{abstract-dne-intro},
with $\calE$ and $\Psi$ given by \eqref{reduced-energy-intro}, we will
follow the abstract variational approach from \cite{MiRoSa13NADN},
which was developed exactly to treat nonsmooth and nonconvex energies
like $\calE$. There, following the spirit of the variational theory
for gradient systems in metric spaces, a set of abstract conditions on
the energy functional $\calE$ and on the dissipation potential $\Psi$
were derived that guarantee to the existence of an \emph{energy
  solution} (i.e.\ fulfilling the EDB \eqref{eq:Int.EDB}) to
\eqref{abstract-dne-intro}.
 
While referring for more details to Section \ref{s:3}, where the
results from \cite{MiRoSa13NADN} are recapitulated, let us highlight
here the role of the required conditions for $(\calE,\diffname,\Psi)$
on coercivity, (lower semi-) continuity, compactness, and
closedness. First we need compactness of sublevels of $\calE$ and
continuity of $\Psi_P(\cdot)$ to ensure that the time-incremental
minimization problems
\begin{equation}
 \label{time-incremental-intro}
  \PPt \tau n \in \argmin_{P \in X} \left( \tau \Psi_{\PPt \tau {n-1}}
  \left(\frac{P -\PPt\tau{n-1}}{\tau}  \right)  +  \ene{\sfti \tau
    n}{P}\right)
\end{equation}
with the state space $X=L^{p}(\Omega;\R^{d \times d})$, have
solutions for every partition $\{ 0=\sfti \tau 0< \sfti\tau 1< \dots <
\sfti\tau{N-1}<\sfti\tau N=T\} $ of $[0,T]$. Classical arguments from
the theory of \emph{minimizing movements}, cf.\
\cite{Ambr95MM,AmGiSa05GFMS} show that the approximate solutions
constructed from $(\PPt \tau n )_{\tau=0}^{N}$ satisfy a discrete
version of the EDI \eqref{eq:Int.EDI}. The limit passage $\tau\down 0$
in this discrete EDI is ensured by the following
\emph{closedness/continuity} property for $(\calE, \diffname)$:
\begin{equation}
\label{closedness-intro}
\left.
\begin{array}{ll}
P_n\to P & \text{in } X,
\\
\Xi_n \weakto \Xi &  \text{in } \amb^*,
\\
\Xi_n \in \diff {t}{P_n} & \text{for all } n \in \N
\end{array}
\right\} \ \Longrightarrow \ \Xi \in \diff t{P} \text{ and }
\en(t,P_n)\to\en(t,P)\,. 
\end{equation}
Hence, the limit $P$ for $\overline P_\tau(t)\to P(t)$ satisfies  the 
time-continuous EDI \eqref{eq:Int.EDI} and the existence of EDI
solutions follows, see Theorem \ref{thm:abstract-2}.  Moreover, if a
suitable version of the chain rule is available we even have an EDB
solution, see Theorem \ref{thm:abstract-2}. 

In Section \ref{s:4} the viscoplastic model is described in full
detail and and the main existence result for EDI solutions for the
viscoplastic model \eqref{pde} is stated, see Theorem
\ref{th:EDI.Viscopl}. The proof of this result is the content of
Section \ref{s:5}, where we check the assumptions of the abstract
theory. Compactness of sublevels for the energy functional $\calE$ is
provided by hardening and a gradient of plasticity term in
$\int_\Omega H(P,\nabla P) \dd x$ with $H(P,A)\geq
C\,\big(|P|^\qp{+}1/(\det P)^\qga + |A|^\qg\big)$.  As in
\cite{MaiMie09GERI,MieRou15?RIEF}, we require $\qg>d$ to ensure $P \in
W^{1,\qg}(\Omega;\R^{d\times d }) \subset \rmC^0 (\overline\Omega;
\R^{d\times d })$
as well as $P^{-1} \in \rmC^0 (\overline\Omega; \R^{d\times d })$,
which will play a crucial role for handling the multiplicative
nonlinearities $\nabla \varphi P^{-1}$ and $\dot{P} P^{-1}$.

The closedness/continuity condition \eqref{closedness-intro} is
the most difficult part in the application to finite-strain
plasticity, since the marginal subdifferential 
\[ 
\diff tP= \Big\{\: {-} \div\big( \rmD_{\nabla P}H(P,\nabla P)\big) {+}
\rmD_PH(P,\nabla P) {+} B(\nabla \varphi,P) \: : \: \varphi \text{
  minimizes } \calI(t,\cdot,P)\:\Big\} 
\]
contains the plastic backstress $B(F,P)=(F\inve P)^\top 
\rmD_{F_\el} W(F \inve P)P^{-\top}$. For $\diff{t}{P_n}$ this
backstress depends nonlinearly on the
deformation gradients $\nabla \varphi_n$, for which there is only
little control, namely that it is a minimizer. First one
can use a multiplicative stress control  for the  \emph{Mandel stress tensor},
namely 
\begin{equation}
  \label{eq:Int.Mandel}
|\transp{F_\el}\rmD_{F_\el} W(x,F_\el) | \leq C W(x,F_\el)
+C^2, 
\end{equation}
which provides an $L^1$ bound for $B(\nabla \varphi_n,P_n)$. Second,
the weak convergence of this nonlinear quantity is obtained by a
technique developed in \cite[e.g.\ Lem.\,4.11]{DaFrTo05QCGN} that relies
only on the minimization property of $\varphi_n$ and the lower
semicontinuity of $\calI(t,\cdot,\cdot)$.

In fact, the multiplicative stress control  condition  for the Mandel tensor in 
\eqref{eq:Int.Mandel} and the related control of the Kirchhoff tensor,
namely $|\partial_{F_\el} W(x,F_\el) \transp{F_\el}| \leq C W(x,F_\el)
+C^2 $ where introduced in \cite{Ball84MELE,BaOwPh91MPAP} to derive
suitable variants of the Euler-Lagrange equations in finite-strain
elasticity with polyconvex energy densities.  Applications in
rate-independent processes were developed in \cite{FraMie06ERCR,
  MaiMie09GERI, KnZaMi10CGPM, MieRou15?RIEF}, even to control the
power term $\partial_t\calI(t,\phi,P)$ in the case of time-dependent
Dirichlet boundary conditions.  

Based on these preparations we are then able to derive the existence
of EDI solutions for the viscoplastic problem \eqref{pde}, see Theorem 
\ref{th:EDI.Viscopl}.\medskip 

To show that these solutions are even EDB solutions, we need to verify
the chain rule \eqref{ch-rule-intro} for our application in
viscoplasticity. One major point is to gain enough control over $t\in
\Xi(t)$ such that the duality pairing $\pairing{}{X}{\Xi}{\dot P }$ can be manipulated. However, because of the highly nonconvex character
of the energy $\calI$, we are  lacking  individual  control of the 
backstress contribution
$\Xi_B(t):=B(\nabla\varphi(t),P(t))$, which is only in $L^1(\Omega;\R^{d \times
  d}$, and the hardening-regularizing contribution
$\Xi_H(t):=\rmD_PH(P,\nabla P) {-}\div\big( \rmD_{\nabla P}H(P,\nabla P)\big)$
of $\Xi=\Xi_B+\Xi_H$.  Ultimately, we are able to prove the chain
rule \eqref{ch-rule-intro} only in the case in which the duality
pairings between $\dot P$ and both contributions $\Xi_B$ and $\Xi_H$
are  well defined individually, cf.\ also
Remark \ref{rmk:failure}.  So far  the
validity of \eqref{ch-rule-intro} is seemingly an open problem for the
energy functional \eqref{reduced-energy-intro}.

In Section \ref{s:6}  we consider a regularized energy
$\calE_\eta(t,P)= \min\{\,\calI_\eta(t,\varphi,P)\,:\; \varphi\in
\calF\,\}$ with 
\[
\widetilde\calI(t,\varphi,P) = \calI(t,\varphi,P)+ \eta \int_\Omega
\widetilde W(\nabla \varphi)\dd x, \text{ where  } |W_\el(F_\el)|^{p'} \leq
C \widetilde W(F)+ C^2. 
\]
The new density $\widetilde W$  provides  a purely elastic part that gives
higher integrability to $\nabla \varphi$ such that the backstress
$\Xi_B$, which is given in terms of the Mandel tensor, now lies in
$X^*=L^{p'}(\Omega;\R^{d \times d}$). This construction allows us to
establish the chain rule \eqref{ch-rule-intro} for $\calE_\eta$ and
all EDI solutions are indeed EDB solutions, see Theorem
\ref{th:EDB.Viscopl}.  

Finally, in Section \ref{s:7} we discuss some extensions of our
existence results, in particular to the case of time-dependent
Dirichlet loadings.

\section{Solution concepts and  existence results for generalized
  gradient systems} 
\label{s:3}
As mentioned in the introduction, our approach to the existence theory
for the initial boundary value problem for \eqref{pde} is based
on the study of the abstract generalized gradient system
\begin{equation}
\label{abstract-dne}
0 \ \in \ \partial \Psi_{P(t)}(\dot P(t)) + \diff t{P(t)}  \ \ 
\subset \amb^* \quad \foraa   t \in (0,T),
\end{equation}
supplemented with the initial condition
\begin{equation}
\label{initial-P}
P(0)=P_0 \ \in  \amb,
\end{equation}
where we assume throughout that $\amb$ is a separable
reflexive Banach space. Moreover, for $\calE$, $\Psi$, and $\diffname$ we 
use the following notions and notations: 
\begin{enumerate}
\item The energy functional $\en : [0,T] \times \amb \to (-\infty,
  \infty]$ has proper domain $[0,T] \times \domain$, with $\domain
  \subset \amb$, it is bounded from below and lower semicontinuous.
\item The dissipation potential $(\Psi_P)_{P \in \domain}$
  is a Finsler family of non-negative, convex functionals  on
  $\amb$, such that for every $P \in \domain$ both $\Psi_P$ and its
  conjugate $\Psi_P^*$ have superlinear growth at infinity; then,
  $\partial \Psi_{P} : X \rightrightarrows X^*$ denotes the
  subdifferential of $\Psi_{P}$ in the sense of convex analysis.
\item The multivalued mapping $\diffname: [0,T] \times
  \domain \rightrightarrows \ambs$ is such that for every $t\in
  [0,T]$ the mapping $\diff t{\cdot} : \domain \rightrightarrows \ambs$ is a
  suitable notion of a subdifferential for  $\en
  (t,\cdot)$.
\end{enumerate}
In fact, we are going to apply this approach to the reduced energy
$\en:(t,P)\mapsto \min\calI(t,\cdot,P)$ defined by minimizing out the
deformation variable $\varphi$ from the original energy $\calI$. Even
if $\calI$ is differentiable with respect to $t\in [0,T]$, the
minimization with respect to $\varphi$ leads to nonsmoothness in $t$,
because the functional $\calI(t,\cdot,P)$ is nonconvex. Hence, the
subdifferential $\diffname$ of $\en$ as well as the time derivative of
$t\mapsto \ene tP$ need to be handled carefully by resorting to
suitable surrogates of the differentials of $\en$. In particular,
following \cite{MiRoSa13NADN} we need a coupling of the choice $\Xi\in
\diff tP$ and the time derivative. Thus as a surrogate of the power
$\partial_t\en$ we define the \emph{power functional}
\begin{equation}
\label{power}
\Ptname: \graph(\diffname) \to \R \quad  \text{ Borel-measurable}
\end{equation}
satisfying suitable properties given below. The typical choices
for the marginal subdifferential $\diffname$ and the associated power functional
$\Ptname$ in case of $\ene tP=\min \calI(t,\cdot,P)$ are 
\begin{align}
\label{diff:concrete}
&\diff tP= \big\{\, \rmD_P\calI(t,\varphi,P)\::\: \varphi \in
\calM(t,P)  \,\big\}, \quad \text{where }\calM(t,P):=\argmin
\calI(t,\cdot,P), \\
\label{eq:Pt.marg}
&\Pt tP\Xi = \sup\big\{\,\partial_t\calI(t,\varphi,P)  \::\:
\varphi\in\calM(t,P),\ \Xi = \rmD_P\calI(t,\varphi,P)  \,\big\}. 
\end{align}
In the abstract setting of Section \ref{s:3} we will not use this form, but in
the application to viscoplasticity we will show that these definitions
satisfy the assumptions \eqref{basic-1}--\eqref{eq:468} and
\eqref{eq:psi-sum-1}--\eqref{eq:psi-uni} given below.  

In what follows, we shall refer to the quintuple $\grasys$ as a
\emph{generalized gradient system}, which gives rise to the
generalized gradient-flow equation 
\eqref{abstract-dne}. We continue to use the short-hand $\Psi$ for the 
dissipation potential 
 $(\Psi_P)_{P \in
  \domain}$. 

Next we will discuss three notions of solutions for
$\grasys$ and a series of abstract conditions 
under which it is possible to obtain existence results for such
concepts.

\subsection{Solution concepts}
\label{su:3.1}

Our first notion of solution consists of the \emph{energy-dissipation
  inequality} (EDI) \eqref{e:weak-enineq} below. As it will be clear
from the the proof of Theorem \ref{thm:abstract-2} ahead, this is the
weakest notion of solution arising in the limit of the
\emph{Minimizing Movement scheme} under the typical assumptions of
lower semicontinuity, coercivity, and closedness on the pair
$(\cE,\Psi)$ in the \emph{variational approach to gradient
flows}, cf.\ \cite{DeMaTo80PEMS, Ambr95MM, AmGiSa05GFMS, RosSav06GFNC,
  MiRoSa13NADN}. For our time-dependent situation we will need
some extra condition on the power $\Ptname$.  

\begin{definition}[EDI solution]
\label{def:weak-sols}
We say that a function $P \in \AC ([0,T]; X)$ is a \emph{EDI
  solution} (i.e.\ it satisfies the Energy-Dissipation Inequality)
to the generalized gradient system $\grasys$ if there exist
\begin{enumerate}
\item a function  $\mathscr{E} \in \BV ([0,T])$ such that
\begin{equation}
\label{pointwise-energies}
 \begin{aligned} 
  &\mathscr{E}(t) \geq \en (t,P(t)) \text{ for all }t\in [0,T], \text{ and} 
 \\
  &\mathscr{E}(s) =  \en(s,P(s)) \text{ for a.a. }s \in (0,T) 
  \text{ and for }s=0;
 \end{aligned} 
\end{equation}
\item a function 
$\Xi \in L^1 (0,T;X^*)$ with
\begin{equation}
\label{xi-in-diff}
\Xi (t) \in
\diff t{P(t)} \quad \foraa   t \in (0,T);
\end{equation}
\end{enumerate}
such that
the  \emph{energy-dissipation inequality (EDI)} holds for all $0 \leq s \leq
t \leq T$ 
\begin{equation}
\label{e:weak-enineq}
\text{(EDI)} \quad 
 \mathscr{E}(t)  + \int_s^t\!\! \left( \Psi_{P(r)}(\dot P(r)) {+}
   \Psi_{P(r)}^* (-\Xi(r)) \right) \!\dd r  \leq 
 \mathscr{E}(s)  + \int_s^t \!\Pt{r}{P(r)}{\Xi(r)} \dd r\,.
\end{equation}
\end{definition}

\begin{remark}\slshape\label{rem:scrE} 
Clearly, from 
\eqref{pointwise-energies} and \eqref{e:weak-enineq} we deduce
\begin{equation}
\label{true-EDI} 
\begin{aligned}
 \nonumber
\ene t{P(t)} + \int_s^t\!\! \Big( \Psi_{P(r)}(\dot P(r))& {+}
  \Psi_{P(r)}^* (-\Xi(r)) \Big) \dd r  \\
  &  \leq 
\ene s{P(s)}   {+} \int_s^t \!\Pt{r}{P(r)}{\Xi(r)} \dd r\\
&  \quad \text{for all }  t\in [0,T], \text{ for } s=0,  \text{ and }
\foraa   s \in (0,t).
\end{aligned}
\end{equation}
Furthermore, 
since the triple $(P,\Xi,\mathscr{E})$ also fulfills the  energy-dissipation inequality in the differential form
\begin{equation}
\label{Eweak-ineq}
\dot{\mathscr{E}}(t) - \Pt{t}{P(t)}{\Xi(t)} \leq -  \Psi_{P(t)}(\dot P(t)) -
\Psi_{P(t)}^* (-\Xi(t))  \quad \foraa   t \in (0,T), 
\end{equation}
taking into account \eqref{pointwise-energies}, we infer the
distributional inequality
\begin{equation}
\label{Edistrib}
\frac{\dd }{\dd t }\ene t{P(t)}  +\Psi_{P(t)}(\dot P(t)) +
\Psi_{P(t)}^* (-\Xi(t)) \ \leq  \  \Pt{t}{P(t)}{\Xi(t)} 
  \quad \text{in }
\mathscr{D}'(0,T). 
\end{equation}
\end{remark}

\begin{remark}[Brenier's dissipative solutions]\slshape   Our notion of EDI
  solutions is closely related to the notion of \emph{dissipative
    solutions} introduced in \cite[Sec.\,5.1]{Bren15COTC}. There the
  dual dissipation potential $\Psi_P^*({-}\Xi)$ is estimated from
  below which leads to the
  estimate 
\begin{align*}
\ene t{P(t)} + \int_0^t \Big(\Psi_{P(r)}(\dot P(r))
- \pairing{}{X}{\Xi(r)}{z(r)} - \Psi_{P(r)}(z(r)) \Big)\dd r\ \\
 \leq \ 
\ene 0{P(0)} + \int_0^t \Pt r{P(r)}{\Xi(r)}\dd r,
\end{align*}
which is supposed to hold for a suitable dense set of test functions
$z:[0,T]\to X$. The advantage in \cite{Bren15COTC} is that the 
difficult pairing $ \pairing{}{X}{\Xi(r)}{z(r)}  $ can be treated more
efficiently for good test functions. 
\end{remark}

All functions $P:[0,T]\to
X$ satisfying 
 the differential inclusion
\eqref{abstract-dne}  are called \emph{DNE solutions}, since
they solve the doubly nonlinear equation. Note that for EDI solutions,  in general  we are  not able to establish  \eqref{abstract-dne},  while 
 DNE solutions 
need not
fulfill  EDI.  
We now present the much stronger concept of \emph{EDB solutions},
which are contained in the intersection of EDI and DNE solutions, see
Figure \ref{fig:Sol}. Indeed we ask that EDI holds as an
\emph{equality}, namely  the abstract Energy-Dissipation Balance (EDB)
\eqref{abstract-enid}. In our abstract nonconvex setting, the validity
of the doubly nonlinear differential inclusion \eqref{abstract-dne}
and of the EDB are strongly linked, which justifies our the choice of
the following definition.

\begin{definition}[EDB solution]
\label{def:energy-sols}
We say that a function $P \in \AC ([0,T]; X)$ is an \emph{EDB
  solution} to the generalized gradient system $\grasys$ if there
exists $\Xi \in L^1 (0,T;X^*)$ such that the pair $(P,\Xi)$ solves the
generalized gradient-flow equation \eqref{abstract-dne}, i.e.
  \begin{equation}
  \label{pointwise-sense}
  \Xi(t) \in \diff{t}{P(t)} \quad \text{and}   
  \quad 0 \in \partial \Psi_{P(t)}(\dot P(t))
  +\Xi(t) \quad \text{ for almost all } t \in (0,T), 
  \end{equation}
and fulfills, for all $0 \leq s \leq t \leq T$,
  the \emph{energy-dissipation balance}
\begin{equation}
\label{abstract-enid}
\text{(EDB)}\qquad \left\{
 \begin{aligned} 
  \en(t,P(t)) + \int_s^t\!\! \left( \Psi_{P(r)}(\dot P(r)) + \Psi_{P(r)}^*
   (-\Xi(r)) \right) \!\dd r\ \\
   = \en(s,P(s)) + \int_s^t \Pt{r}{P(r)}{\Xi(r)} \dd r\,.
 \end{aligned}\qquad \right.
\end{equation}
\end{definition}

Observe that the distributional inequality \eqref{Edistrib}
reformulates  as
\begin{equation}
\label{almost-solution}
\begin{aligned}
\Psi_{P(t)}(Q) - \Psi_{P(t)}(\dot P(t))   \geq \int_{\Omega}
\pairing{}{\amb}{-\Xi(t)}{Q} \dd x +  \frac{\dd }{\dd t }\ene t{P(t)}
- \Pt{t}{P(t)}{\Xi(t)} 
 \end{aligned}
\end{equation}
for all $Q \in X$. This highlights the fact that the
chain-rule inequality
\[
 \frac{\dd }{\dd t }\ene t{P(t)}- \Pt{t}{P(t)}{\Xi(t)}  \geq
 \int_{\Omega}\pairing{}{\amb}{\Xi(t)} { \dot P(t)} \dd x, 
 \]
  which would lead to the inequality
 \[
 \Psi_{P(t)}(Q) - \Psi_{P(t)}(\dot P(t)) \geq \int_{\Omega}
 \pairing{}{\amb}{{-}\Xi(t))}{Q{-}\dot P(t)} \dd x \text{ for all } Q
 \in X,
 \]
 is in fact the missing ingredient to conclude from
 \eqref{almost-solution} that the pair $(P,\Xi)$ is a pointwise
 solution to \eqref{abstract-dne} in the sense of
 \eqref{pointwise-sense}.  Indeed, our next result states that every
  EDI solution turns out to be an  EDB  solution if the
 functional $\en$ satisfies a suitable chain-rule inequality with
 respect to the triple $(\Psi,\diffname, \Ptname)$, namely

\begin{definition}[Chain-rule inequality (CRI)]\label{def:ChainRule}\slshape 
We say that the gradient system $\grasys$ satisfies the 
\emph{chain-rule inequality}, if for every $P \in
\AC([0,T];\amb)$ and $\Xi\in L^1(0,T;\ambs)$ 
with 
\begin{equation}
  \label{conditions-1}
  \begin{aligned}
  &
\sup_{t \in (0,T)} |\en(t,P(t))
|<\infty, \qquad  \Xi(t)\in \diff t{P(t)}\ \ \text{for a.a.\ }t\in
(0,T), 
\\
&
  \int_0^T\Psi_{P(t)}(\dot P(t)) \dd t<\infty,\quad \text{and}\quad 
    \int_0^T\Psi_{P(t)}^*(-\Xi(t))\dd t<\infty,
  \end{aligned}
  \end{equation}
we have that 
\begin{equation}
    \label{eq:48strong}
    \text{(CRI)} \quad 
    \begin{aligned}
        &\text{the map }  t\mapsto \en( t,P(t)) 
           \text{ is absolutely continuous and}
   \\
      &    \frac \rmd{\rmd t} \en(t,P(t)) \geq 
       \pairing{}{\amb}{\Xi(t)}{\dot P(t)}+\, \Pt{t}{P(t)}{\Xi(t)}
        \quad   \text{for a.a.\ }t\in (0,T).
    \end{aligned}
  \end{equation}
\end{definition}

\begin{proposition}
\label{prop:weak-to-en}
Assume that $\grasys$ satisfies the CRI \eqref{eq:48strong}.
Then every EDI solution $P \in  \AC ([0,T]; X) $
 is an EDB solution.
\end{proposition}
\begin{proof}
  From \eqref{true-EDI} written for $s=0$ we obtain
 \begin{equation}
 \label{chain-eqs}
 \begin{aligned}
&\en(t,P(t)) + \int_0^t \left( \Psi_{P(r)}(\dot P(r)) + \Psi_{P(r)}^* (-\Xi(r)) \right) \dd r  \\  &\leq
\en(0,P(0)) + \int_0^t \!\Pt{r}{P(r)}{\Xi(r)} \dd r
\ \leq \ \en(t,P(t)) + \int_0^t \!\!\pairing{}{\amb}{{-}\Xi(r)}{\dot P(r)} \dd r,
\end{aligned}
 \end{equation}
where the latter estimate is due to the CRI
\eqref{eq:48strong}. Thus,  we get 
\[
 \int_0^t \left( \Psi_{P(r)}(\dot P(r)) + \Psi_{P(r)}^* (-\Xi(r))  - \pairing{}{\amb}{{-}\Xi(r)}{\dot P(r)}  \right) \dd r \leq 0 \quad \text{for every } t \in [0,T].
\]
Since the integrand is non-negative by an elementary inequality from
convex analysis, we conclude that $\Psi_{P(t)}(\dot P(t)) +
\Psi_{P(t)}^* (-\Xi(t)) - \pairing{}{\amb}{{-}\Xi(t)}{\dot P(t)} = 0 $
for a.a.\ $t\in (0,T)$, whence $-\Xi(t) \in \partial \Psi_{P(t)}(\dot
P(t)) $ for a.a.\ $t\in (0,T)$. Taking into account
\eqref{xi-in-diff}, we conclude \eqref{pointwise-sense}. The above
arguments also yield that all inequalities in \eqref{chain-eqs} hold
as equalities for every $t\in (0,T]$, whence the EDB \eqref{abstract-enid}.
\end{proof}

\subsection{Assumptions for the abstract theory}
\label{su:3.2}
We now detail the exact conditions on the generalized gradient
system $\grasys$, under which
the existence of EDI solutions to the Cauchy problem for can be
obtained via the techniques developed in \cite{MiRoSa13NADN}.

We define $\domain(t):=\dom \ene t{\cdot} =\{\,P\in
X\::\: \ene tP < \infty\,\}$ and assume 
\[
\domain:=\domain(0)=\domain(t) \quad \text{for all }
t \in [0,T]. 
\]
Hereafter, we shall use the notation
\begin{equation}
\label{not-g} \cg P :  = \inf_{t \in [0,T]} \en(t,P)  \quad \text{for
  all }P \in \domain 
\end{equation}
For the energy functional $\en$
we require the following conditions, where we will  
use $C$, $C'$, $C_0$, $K_1$ etc.\ for various
positive constants depending only on known quantities.  

\begin{description}
\item[Lower semicontinuity:]
\begin{equation}
\label{basic-1}
\tag{2.$\mathrm{E}_0$}
\begin{aligned}
 &\text{the map } P\mapsto \en (t,{P}) \  \text{is
lower semicontinuous for all $t \in [0,T]$, }
\\
&\exists\, C_0>0\, \ \forall\, (t,u) \in [0,T] \times \domain
 : \ \en (t,P) \geq \cg P\geq  C_0, \text{ and }
\\
&\mathrm{graph}(\diffname) \text{ is a Borel set in } [0,T] \times \amb
\times \ambs. 
 \end{aligned}
\end{equation}
(Note that, if $\en$ is bounded from below, then we can suppose
without loss of generality that it is bounded by a strictly positive
constant.) 

\item[Coercivity:] For all $t \in [0,T]$
\begin{equation}
  \label{eq:17-bis}
  \tag{2.$\mathrm{E}_1$}
 \text{the map}\quad
   P\mapsto \en(t, P) \quad\text{has compact sublevels in $\amb$.}
\end{equation}

\item[Variational sum rule:]
  If for some $P_*\in \amb$ and $\tau>0$ the point $\bar P$ is a minimizer of
  $P\mapsto \en (t,P)+\tau\Psi_{P_*}((P-P_*)/\tau)$, then  $\bar P$
  fulfills the Euler-Lagrange equation 
  \begin{equation}
    \label{eq:42-bis}
   \tag{2.$\mathrm{E}_2$}
    \exists \,  \Xi\in \diff t{\bar P} \,:\quad -\Xi\in \partial 
    \Psi_{P_*}((\bar P- P_*)/\tau).
  \end{equation}

\item[Lipschitz  continuity with respect to $t \in {[0,T]}$:]
\begin{equation}
\label{eq:diffclass_a}
\tag{2.$\mathrm{E}_3$}
 \begin{aligned}
  &\exists\, K_1>0 \ \forall\, P \in  \domain  \
   \forall\, t,s \in [0,T]: 
   \quad |\en(t,P) -\en(s,P)| \leq
  K_1 \cg P \,|t{-}s|.
\end{aligned}
  \end{equation}

\item[Conditioned differentiability with respect to
  $t\in{[0,T]}$:] There exists a power functional $\Ptname:
  \mathrm{graph}(\diffname) \to \R$, which is a Borel map and
  which satisfies 
\begin{align}
\label{e:ass-p-a}
\tag{2.$\mathrm{E}_4$}
&\exists\, K_2>0 \  \  \forall\, (t,P,\Xi) \in
\mathrm{graph}(\diffname)\,: \\  
&\nonumber
\liminf_{h \down 0}
\frac1{h}\big(\en(t{+}h,P) {-} \en(t,P)\big) \leq \Pt{t}{P}{\Xi} 
\leq \lim_{h\down 0} \frac1h\big( \ene tP{-}\ene{t{-}h}P\big) \leq
K_2\cg{P}\,. 
\end{align}

\item[Weak closedness of $(\en,\diffname,\Ptname)$:] For all $t \in
  [0,T]$ and for all sequences $ (P_n)_n$ in $\amb $, $ (\Xi_n)_n$ in 
  $\ambs$ satisfying $\ \Xi_n\in \diff{t}{P_n}$, $\ \EE_n=\en(t,P_n)$, 
  $\ \mathscr{P}_n =\Pt{t}{P_n}{\Xi_n}$ as well as 
\[
   P_n\to P \ \ \text{in $\amb$,}  \quad  \Xi_n\weakto \Xi  \ \ 
   \text{weakly in  $\ambs$,} \quad
   \mathscr{P}_n \to \mathscr{P}  \ \ \text{and}  \ \  \EE_n\to \EE 
      \ \ \text{in }\R, 
\]
we have 
 \begin{equation}
 \label{eq:468}
 \tag{2.$\mathrm{E}_5$}
     \Xi\in \diff tP,\quad  \mathscr{P}\leq \Pt{t}{P}{\Xi},\quad
      \text{and} \quad \EE=\en(t,P).
 \end{equation}
\end{description}

Gronwall's lemma, \eqref{basic-1}, and  \eqref{eq:diffclass_a}
provide also an 
upper bound for $\ene tP$, namely
\begin{equation}
\label{energy-used-later}
\forall\, P \in \domain \, : \quad \cg P \leq
\inf_{t\in[0,T]} \ene tP \leq \sup_{t\in[0,T]}\en(t,P)  \leq 
\exp(K_1T)\,\cg P. 
\end{equation}
Condition \eqref{e:ass-p-a} is fitted to the power
functional $\Ptname$ that is given in \eqref{eq:Pt.marg}, where only left
and right derivatives are defined, see
\cite{KnZaMi10CGPM,MiRoSa13NADN}. In simpler  situations  where 
$t\mapsto \ene tP$ is differentiable, we just have to use $\Pt
tP{\Xi}=\partial_t \ene tP$. 

For the dissipation potentials $(\Psi_P)_{P \in
  \domain}$ we require the following conditions: 
\begin{description}
\item[Convexity and lower semicontinuity:] For every $P \in \domain$
  the function 
\begin{align}
 \label{eq:psi-sum-1}
 \tag{2.$\Psi_1$}
 &\Psi_P: \amb \to
 [0,\infty) \ \ 
\text{is l.s.c., convex, and satisfies }\Psi_P(0)=0.
\end{align}

\item[Superlinearity:] The potentials $(\Psi_P)_{P \in \domain}$ and
  $(\Psi_P^*)_{P \in \domain}$ 
have uniform superlinear growth on sublevels of $\en$, viz. 
\begin{align}
 \label{basic-psi-2}
\tag{2.$\Psi_2$}
\begin{aligned}
 \forall\, \SE>0\,: \ \  \left\{\begin{array}{ll}
    \displaystyle \lim_{\normam{V}\to \infty} \Big(\frac{1}{\normam{V}}
    \inf_{\cg P \leq \SE} \Psi_P(V) \Big) =\infty, \\[0.8em]
\displaystyle \lim_{\normams{\Xi}\to \infty}\Big(
\frac{1}{\normams{\Xi}} \inf_{\cg P \leq \SE} \Psi_P^*(\Xi)
 \Big) =\infty.
\end{array}\right.
\end{aligned}\qquad
\end{align}

\item[Mosco continuity:] The map $P \mapsto \Psi_P$ is continuous on
  sublevels of $\en$ in the sense of Mosco convergence (cf.\
  \cite{Atto84VCFO}), i.e.\
  for all $\SE>0$ we have  
\begin{align}
 \label{basic-psi-3}
\tag{2.$\Psi_3$}
\left. \begin{array}{c}
  P_n\to P \text{ in }\amb, \ \  \cg {P_n}\leq \SE,   \\[0.5em]
  V_n\weakto V \text{ in }\amb,\ \ \Xi_n\weakto \Xi \text{ in }\ambs
  \end{array} \right\}\ \ 
 \Longrightarrow \ \ 
\left\{ \begin{array}{l}
  \liminf\limits_{n\to\infty}\Psi_{P_n}(V_n)\ge \Psi_P(V),
\\[0.7em]
  \liminf\limits_{n\to\infty}\Psi_{P_n}^*(\Xi_n)\ge \Psi_P^*(\Xi).
  \end{array}\right.
\end{align}

\item[Univaluedness of $\Psi^*_P$:]  For all $P\in \domain$ we have 
 \begin{align}
  \label{eq:psi-uni}
   \tag{2.$\Psi_4$} 
   \forall\, V \in \amb \ \forall\, \Xi_1,\, \Xi_2 \in
  \partial \Psi_P(V)\,: \quad \Psi_P^*(\Xi_1) =
  \Psi_P^*(\Xi_2).
  \end{align}
\end{description}

As pointed out in \cite[Rem.\,2.1]{MiRoSa13NADN}, condition
\eqref{eq:psi-uni}  is satisfied if $\Psi_P$ is given by the sum of 
positively homogeneous or differentiable, convex potentials. 
 
\subsection{Existence of EDI and EDB solutions}
\label{su:ExistSol}

In \cite[Thm.\ 4.4]{MiRoSa13NADN} it is proved that
\eqref{basic-1}--\eqref{eq:468},
\eqref{eq:psi-sum-1}--\eqref{eq:psi-uni} and the chain-rule inequality  CRI
\eqref{eq:48strong} imply the existence of EDB solutions for the
Cauchy problem \eqref{abstract-dne}--\eqref{initial-P}. However, for
the application to the viscoplasticity system \eqref{pde}, it will be
crucial to establish an independent result on the existence of EDI
solutions, cf.\ Theorem \ref{thm:abstract-2}, because we are not able
to establish the  CRI  in the general case. Although
its proof can be  inferred from the argument for \cite[Thm.\
4.4]{MiRoSa13NADN}, we will briefly outline it for the sake of
readability.

\begin{theorem}[EDI solutions]
\label{thm:abstract-2} Assume that $\grasys$  complies with
\eqref{basic-1}--\eqref{eq:468} and
\eqref{eq:psi-sum-1}--\eqref{eq:psi-uni}.  Then, for every $P_0 \in
\domain$ there exists an EDI solution to the Cauchy problem
\eqref{abstract-dne}--\eqref{initial-P}.
\end{theorem}

As a straightforward consequence of Proposition \ref{prop:weak-to-en}
we have the following result. 

\begin{corollary}[EDB solutions, cf.\ {\cite[Thm.\,4.4]{MiRoSa13NADN}}]
\label{cor:abstract-1}
If additionally to the assumption of Theorem \ref{thm:abstract-2}
the CRI \eqref{eq:48strong} holds, then for every $P_0 \in
\domain$ there exists an  EDB  solution to the Cauchy problem
\eqref{abstract-dne}--\eqref{initial-P}.
\end{corollary}
\begin{proof}[Sketch of the proof of Thm.\ \ref{thm:abstract-2}]
  First of all, we set up the time-discretization scheme
  for \eqref{abstract-dne}: Given a partition $\{ 0=\sfti \tau 0<\sfti
  \tau 1<\ldots <\sfti \tau N =T\}$ of the interval $[0,T]$, with
  constant time-step $\tau = T/N$, we construct discrete solutions
  $(\PPt \tau n)_{n=0}^N\subset \amb$ starting from $\PPt \tau 0:=
  P_0$ and finding
\begin{equation}
\label{time-incremental}
\PPt \tau n \in \argmin_{P \in X} \left( \tau \Psi_{\PPt \tau {n-1}}
  \left(\frac{P -\PPt\tau{n-1}}{\tau}  \right)  +  \ene{\sfti \tau
    n}{P}\right) . 
\end{equation}
We thus define the approximate solutions to \eqref{abstract-dne}
via interpolation of the values  $(\PPt \tau n)_{n=0}^N,$ with
\[
\left\{
\begin{array}{lll}
\pwc P\tau (t) &  := \PPt \tau n  & \text{ for $t\in (\sfti \tau {n-1},\sfti\tau n]$},
\\
 \upwc P\tau (t)  & := \PPt \tau {n-1}  &  \text{ for $t\in [\sfti \tau {n-1},\sfti\tau n)$},
 \\
   \pwl P\tau(t)  & := \frac{t-\sfti \tau {n-1}}{\tau} \PPt \tau n +
   \frac{\sfti \tau {n} - t}{\tau} \PPt \tau {n-1}  & \text{ for $t
     \in [\sfti \tau{n-1},\sfti \tau n]$.} 
   \end{array}
   \right.
\]
Using the variational sum rule \eqref{eq:42-bis} and 
the Euler-Lagrange equation for \eqref{time-incremental} yields
\begin{equation}
\label{approx-eq-interp}
0\ \in \ \partial\Psi_{\upwc P\tau(t)} (\pwl {\dot P}{\tau}(t)) +
\diff {\pwc{\sft}{\tau}(t)} {\pwc P\tau (t)} \ \  \subset  \ambs
\quad \foraa   t \in (0,T), 
\end{equation}
where $\pwc {\mathsf{t}}\tau$ denotes the piecewise constant
interpolant associated with the partition.

Due to the lack of convexity of $\en(t,\cdot)$, it turns out that the
approximate version \eqref{approx-eq-interp} of \eqref{abstract-dne}
is not sufficient for the limit passage to the time-continuous
level. One needs the finer information provided by the
approximate EDI
\begin{equation}
\label{approx-edi}
\begin{aligned}
 &\en(\pwc {\mathsf t}\tau (t), \pwc P\tau (t))  + \int_{\pwc {\mathsf
     t}\tau (s)}^{\pwc {\mathsf t}\tau (t)} \left( \Psi_{\upwc P\tau
     (r)}(\pwl{\dot P}\tau(r)) + \Psi_{\upwc{P}\tau(r)}^* ({-}\pwm
   \Xi\tau(r)) \right) \dd r \\
& \leq \en(\pwc {\mathsf t}\tau (s), \pwc P\tau (s))  + 
 \int_{\pwc {\mathsf t}\tau (s)}^{\pwc {\mathsf t}\tau (t)} \Pt{r}{\pwm P\tau(r)}{\pwm
  \Xi\tau(r)} \dd r\,  
\end{aligned}
\end{equation}
for every $0 \leq s \leq t \leq T$.  Its proof (cf.\ \cite[Lemma 6.1]{MiRoSa13NADN})   relies on condition
\eqref{eq:psi-uni} and on De Giorgi's \emph{variational interpolant} for the values
$(P_\tau^n)_{n=1}^N$ (cf.\ \cite{Ambr95MM,
  AmGiSa05GFMS}), which is defined via $\pwm P \tau (0):= P_0$ and
\begin{equation}
\label{def-var-interp}
\pwm P\tau (t)\in  \argmin_{P \in X} \left( (t-\sfti\tau{n-1})
  \Psi_{\PPt \tau {n-1}} \left(\frac{P
      -\PPt\tau{n-1}}{t-\sfti\tau{n-1}}  \right)  +
  \ene{t}{P}\right) \ \ \text{ for } t \in (t_\tau^{n-1}, t_\tau^n). 
\end{equation}
Then, $\pwm \Xi\tau : (0,T) \to \ambs$ in \eqref{approx-edi}  is a
measurable selection with 
$\pwm \Xi\tau (t) \in \diff {t}{ \pwm P\tau(t)}$ for a.a.\ $t \in
(0,T)$  and satisfying the Euler-Lagrange equation for the minimum problem
\eqref{def-var-interp}.

Continuing from \eqref{approx-edi} and exploiting the positivity of
$\Psi$ and $\Psi^*$, one derives a series of a priori
  estimates for the families $(\pwc P\tau)_\tau$, $(\pwl P\tau)_\tau$,
$(\pwm P\tau)_\tau$ and $(\pwm \Xi\tau)_\tau$, namely 
\begin{equation}
\label{estimates-sec-3}
\sup_{t\in [0,T]} \big(\cg{\pwc P\tau(t)}{+}\cg{\pwm
  P\tau(t)}\big)\leq C, \quad  \int_0^T \!\!\left( \Psi_{\upwc P\tau
    (r)}(\pwl{\dot P}\tau(r)) {+} \Psi_{\upwc{P}\tau(r)}^* ({-}\pwm
  \Xi\tau(r)) \right) \dd r \leq C 
\end{equation}
for a positive constant $C$, uniformly with respect to
$\tau>0$. Hence, there exist $P \in \AC ([0,T];\amb)$ and
$\widetilde\Xi \in L^1(0,T;\ambs)$ such that along a sequence $\tau_k
\down 0$ we have 
\begin{subequations}
\label{convergences-3}
\begin{align}
&
\label{convergences-3-1}
\pwc P{\tau_k}, \, \pwl P{\tau_k}, \, \pwm P{\tau_k} \to P  && \text{ in } L^\infty (0,T; X),
\\
\label{convergences-3-2}
& \pwl {\dot P}{\tau_k}\weakto \dot P && \text{ in } L^1 (0,T; X),
\\
&
\label{convergences-3-3}
\pwm {\Xi}{\tau_k} \weakto \widetilde{\Xi} && \text{ in } L^1(0,T; X^*)
\end{align}
where the strong convergence with values in $X$ derives from the
energy bound in \eqref{estimates-sec-3} and the compactness of energy
sublevels (cf.\ \eqref{eq:17-bis}), while
\eqref{convergences-3-2}--\eqref{convergences-3-3} are a consequence
of the integral bound of the dissipative terms $\Psi_{\upwc P\tau
  (r)}(\pwl{\dot P}\tau(r)) $ and $ \Psi_{\upwc{P}\tau(r)}^* ({-}\pwm
\Xi\tau(r)) $, and of the superlinear growth \eqref{basic-psi-2} of
$\Psi$ and $\Psi^*$.  From \eqref{estimates-sec-3} it also follows
that $\sup_{t\in (0,T)} | \Pt{t}{\pwm P\tau(t)}{\pwm \Xi\tau(t)} |\leq
C$ via \eqref{e:ass-p-a}.  Moreover, Helly's principle also yields
that there exists a function $\EE \in \mathrm{BV}([0,T])$ such that
\begin{equation}
\label{convergences-3-4}
\lim_{\tau_k \down 0}  \en(\pwc {\mathsf t}{\tau_k} (t), \pwc
P{\tau_k} (t)) = \EE(t) \quad \text{for every } t \in [0,T]. 
\end{equation}
\end{subequations}
Taking into account that $\lim_{\tau_k \down 0} \en(\pwc {\mathsf
  t}{\tau_k} (t), \pwc P{\tau_k} (t)) \geq \ene t{P(t)}$ by the lower
semicontinuity of the map $t \mapsto \ene t{P}$, combined with the
fact that $| \en(\pwc {\mathsf t}{\tau_k} (t), \pwc P{\tau_k} (t)) -
\en(t, \pwc P{\tau_k} (t)) | \to 0 $ thanks to \eqref{eq:diffclass_a},
we deduce the first inequality in \eqref{pointwise-energies}.

To pass to the time-continuous limit it is necessary to gain further
insight into the limiting properties of the sequences $ (\pwm
\Xi{\tau_k} )_{\tau_k}$ and $( \Pt{t}{\pwm P{\tau_k}}{\pwm
  \Xi{\tau_k}})_{\tau_k}$ by exploiting a Young-measure argument, see
\cite[App.\,A]{MiRoSa13NADN}.  A (not relabeled) subsequence $ (\pwm \Xi{\tau_k}, \Pt{t}{\pwm
  P{\tau_k}}{\pwm \Xi{\tau_k}})_{\tau_k}$ possesses  a limiting Young 
measure $\bm\mu = (\mu_t)_{t\in (0,T)}$, such that for almost all $t\in
(0,T)$ the probability measure $\mu_t$ on $\ambs \times \R$ is 
concentrated on the set of the weak limit points of the sequence $
(\pwm \Xi{\tau_k}(t), \Pt{t}{\pwm P{\tau_k}(t)}{\pwm
  \Xi{\tau_k}(t)})_{\tau_k} $.  Taking into account the weak
closedness property \eqref{eq:468} it follows that for almost all
$t\in (0,T)$ we have $\mu_t ((\ambs {\times} \R) {\setminus}
\mathcal{S}(t,P(t))) =0$, where 
\[
\mathcal{S}(t,P(t)) = \{ (\Xi, \mathfrak{p} ) \in \ambs \times \R\, : \
\Xi \in \diff{t}{P(t)}, \ \mathfrak{p} \leq \Pt t{P(t)}{\Xi} \}.
\]

Moreover, for every subinterval $[s,t]\subset [0,T]$ the Mosco
convergence  \eqref{basic-psi-3} implies 
 \begin{equation}
 \label{YM-lsc}
\begin{aligned}
&\liminf_{k\to \infty} \int_s^t \! \left(  \Psi_{\upwc{P}{\tau_k}(r)}^*
  ({-}\pwm \Xi{\tau_k}(r)) {-} \Pt r{\pwm
    P{\tau_k}(r)}{\pwm\Xi{\tau_k}(r)} \right) \!\dd r 
\\
&\qquad \geq \int_s^t \int_{\ambs \times \R}\! \left( \Psi_{P(r)}^*
  ({-}\Xi)-\mathfrak{p} \right)\!  \dd \mu_r (\Xi,\mathfrak{p}) \dd r,  
\end{aligned}
\end{equation}
cf.\ \cite[Thm.\ A.3]{MiRoSa13NADN}. Finally, by \eqref{eq:468} and
classical selection techniques there exists a measurable selection $t
\mapsto (\Xi(t), \mathfrak{p}(t))$ with
\begin{equation}
\label{meas-sel-argument}
(\Xi(t), \mathfrak{p}(t)) \in \mathcal{S}(t,P(t)) \text{ and }
\Psi_{P(t)}^*(-\Xi(t)) - \mathfrak{p}(t) =\! \min_{(\Xi,p) \in \mathcal{S}(t,P(t))}
\!\!\left(  \Psi_{P(t)}^* ({-}\Xi)  {-} \mathfrak{p}\right)
\end{equation}
for all $t\in [0,T]$.  From this it is not difficult to deduce that
$\Xi \in L^1 (0,T;\ambs)$.  Moreover, for every $[s,t]\subset [0,T]$
we obtain
\[
\begin{aligned}
 & \EE(t) + \int_s^t \Psi_{P(r)}(\dot P(r)) \dd r + \int_s^t \left(   \Psi_{P(r)}^*({-}\Xi(r)) {-}  \Pt {r}{P(r)}{\Xi(r)} \right) \dd r
\\
& 
 \stackrel{(1)}{\leq}
\EE(t)   + \int_s^t \Psi_{P(r)}(\dot P(r)) \dd r + 
 \int_s^t \left(   \Psi_{P(r)}^*({-}\Xi(r)) {-} \mathfrak{p}(r) \right) \dd r 
\\
 & \stackrel{(2)}{\leq}
\EE(t)   + \int_s^t \Psi_{P(r)}(\dot P(r)) \dd r + \int_s^t   \int_{\ambs \times \R}\left(  \Psi_{P(r)}^* ({-}\Xi) - \mathfrak{p}  \right) \dd \mu_r (\Xi,\mathfrak{p}) \dd r
\\
&
\stackrel{(3)}{\leq}
 \lim_{k \to \infty}  \en(\pwc {\mathsf t}{\tau_k} (t), \pwc P{\tau_k} (t))
   + \liminf_{k \to \infty}  \int_{\pwc {\mathsf t}{\tau_k} (s)}^{\pwc {\mathsf t}{\tau_k} (t)}  \Psi_{\upwc P{\tau_k} (r)}(\pwl{\dot P}{\tau_k}(r)) \dd r
 \\
 & \qquad 
  +  \liminf_{k\to \infty} \int_{\pwc {\mathsf t}{\tau_k}
  (s)}^{\pwc {\mathsf t}{\tau_k} (t)}  \left(
  \Psi_{\upwc{P}{\tau_k}(r)}^* ({-}\pwm \Xi{\tau_k}(r)) {-} \Pt r{\pwm
    P{\tau_k}(r)}{\pwm\Xi{\tau_k}(r)} \right) \dd r 
 \\
 &
\stackrel{(4)}{\leq}   \liminf_{k \to \infty}  \en(\pwc {\mathsf
  t}{\tau_k} (s), \pwc P{\tau_k} (s))\  \stackrel{(5)}{=} \ \EE(s),
\end{aligned}
\]
where {\footnotesize (1)} 
 follows from the definition of the set $\mathcal{S}(\cdot,P(\cdot))$,  
{\footnotesize (2)} from the minimality
\eqref{meas-sel-argument} of the selection and {\footnotesize (3)}
from the convergences 
\eqref{convergences-3-1}--\eqref{convergences-3-2} combined with
\eqref{basic-psi-3}, the energy convergence \eqref{convergences-3-4},
and the lower semicontinuity estimate \eqref{YM-lsc}. Finally,
{\footnotesize (4)} follows from passing to the limit $\tau_k\down 0$
in the approximate energy-dissipation inequality
\eqref{approx-edi}, and {\footnotesize  (5)} again from \eqref{convergences-3-4}.

It remains to show the second assertion in \eqref{pointwise-energies}, namely
that 
\begin{equation} 
\label{cv3-5}
\lim_{\tau_k \down 0}  \en(\pwc {\mathsf t}{\tau_k} (t), \pwc P{\tau_k}(t))  
= \ene t{P(t)} \quad \text{ for almost all $t\in (0,T)$.}
\end{equation}
To this aim, we observe that, thanks to estimate
\eqref{estimates-sec-3} and the uniform superlinear growth
\eqref{basic-psi-2} of the potentials $\Psi_P$, we have
$\liminf_{k\to\infty} \| \pwl{\dot P}{\tau_k}(t) \|_X <\infty$ for
almost all $t \in (0,T)$. Using now the superlinearity of $\Psi_P^*$
(cf.\ \eqref{basic-psi-2} again) we see that the multivalued maps
$\partial \Psi_P : X \rightrightarrows X^*$ are uniformly bounded on
energy sublevels. Therefore $\liminf_{k\to\infty} \| \pwc
{\Xi}{\tau_k}(t)\|_{X^*}<\infty$, where $\pwc \Xi{\tau}$ fulfills
$\pwc \Xi{\tau}(t) \in \diff {\pwc{\sft}{\tau}(t)} {\pwc P\tau (t)}
\cap (- \partial\Psi_{\upwc P\tau(t)} (\pwl {\dot P}{\tau}(t)) )$ for
almost all $t\in (0,T)$, cf.\ the Euler-Lagrange equation
\eqref{approx-eq-interp}. Then we can apply the closedness condition
\eqref{eq:468} and conclude that $\liminf_{k\to\infty} \en(\pwc
{\mathsf t}{\tau_k} (t), \pwc P{\tau_k} (t)) = \ene t{P(t)}$ for a.a.\
$t \in (0,T)$. In view of \eqref{convergences-3-4}, we infer
\eqref{cv3-5}.

All in all, we have shown that the pair $(P,\Xi)$ satisfies
\eqref{pointwise-energies} and \eqref{xi-in-diff}, and the above chain
of inequalities yields the EDI \eqref{e:weak-enineq}. This concludes
the proof of Theorem \ref{thm:abstract-2}.
\end{proof}

\section{Assumptions and main result}
\label{s:4}
Prior to listing all our assumptions on the energy functionals and on
the dissipation potentials involved in the viscoplastic system, we fix
some notation and recall some useful identities and inequalities that
shall be used throughout.

\begin{notation}[Matrices]\slshape\label{not:4.1}
  We recall the notation $\GLD:= \{ A \in \R^{d\times d} \, : \
  \det(A)>0\}$.  By $\mathrm{tr}(A)$ we shall denote the trace of a
  matrix $A \in \R^{d\times d}$, by $\cof(A)$ its cofactor matrix, by
  $\mathbb M_s(A) $ the matrix in $\R^{\binom{d}{s}\times
    \binom{d}{s}}$ consisting of all minors (subdeterminants) of $A$
  of order $s$, and by $\mathbb M : \R^{d\times d}\to R^{\mu_d}$ the
  function which maps a matrix to all its minors, with $\mu_d :=
  \sum_{s=1}^d \binom{d}{s}^2 $.  We endow $\R^{d\times d}$ with the
  inner product
\begin{equation}
\label{inner}{ A  : B} \, : = \mathrm{tr}(A \transp{B}) = \sum_{i,j
  =1}^d  a_{ij} b_{ij}, \ \ 
\text{  giving  the \emph{Frobenius norm }} \ \  
|A| : = \sqrt{A : A}.
\end{equation}
Given $r>0$, we shall use the notation
\begin{equation}
\label{neighborhood}
 \mathcal{N}_r: = \left\{N \in \R^{d \times d}\, : \
|N-\mathbf{1}| <r \right\}.
\end{equation}
\end{notation}
\medskip

\noindent {\bf Preliminaries:} We point out for later use that
\begin{equation}
\label{useful-inner} A {:} (BC) = \mathrm{tr}(A \transp{C}
\transp{B}) = (A \transp{C}) : B \quad \text{for all $A,\,B,\,C
\in \R^{d\times d}$.}
\end{equation}
Furthermore, it follows from definition \eqref{inner} that for every $A,\,B \in
\R^{d\times d}$ we have $|A B| \leq |A| |B|$. Combining this with
Young's inequality, we deduce for all $A,\,B \in \R^{d\times d}$ the
following estimate
\begin{equation}
\label{useful-later} |A B^{-1}| \geq \frac{|A|}{|B|} \geq  C r
|A|^{1/r} - (r{-}1) C^{r/{(r-1)}} |B|^{1/{(r-1)}} \quad \text{for all
$r >1$ and all $C>0$.}
\end{equation}
We will also use that
\begin{equation}
\label{cofactor-repre} \inve{A} = \frac1{\det(A)} \transp{\cof(A)}
\quad \text{for all $A \in \GLD$.}
\end{equation}
and that the cofactor matrix  is the G\^ateau derivative of the determinant, i.e.
\begin{equation}
\label{cof-der}
\rmD_A (\det(A)) = \cof(A) \quad \text{for all } A \in \R^{d\times d}.
\end{equation}
\medskip

\noindent {\bf Setup:} In what follows, we consider
  a  bounded domain $\Omega\subset \R^d$,   with
Lipschitz boundary $\Gamma$, and denote by $\Gamma_{\mathrm{Dir}}$ the Dirichlet part of the boundary,
 assumed to have  positive surface measure.

In order to highlight the main features of our analysis and avoid overburdening it with technicalities,
throughout most of the paper we will impose on 
$\Gamma_{\mathrm{Dir}}$ 
  time-independent   boundary conditions  $\varphi_{\mathrm{Dir}}$ for  the deformation field; 
  in Section
\ref{s:7} we will outline how our results   generalize to  time-dependent Dirichlet conditions. 
Thus,
the space for the kinematically
admissible deformations is
\begin{equation}
\label{admissible-phi} \mathcal{F}= \left\{ \varphi \in
W^{1,\qphi}(\Omega;\R^d)\, : \ \varphi = \varphi_{\mathrm{Dir}}  \ \text{on
$\Gamma_{\mathrm{Dir}}$} \right\} \ \ \text{for  $\qphi>1$ to be
specified later.}
\end{equation}
The space for the internal variable $P$ is
\begin{equation}
\label{admissible-P}
\begin{aligned}
\mathcal{P}= \{ P \in W^{1,\qg}(\Omega;\R^{d\times d})
\, : \   P(x)\in \GLD \ \foraa   x
\in \Omega \} 
\end{aligned}
\end{equation}
with  $\qg>1$  specified
later.
\medskip

\noindent
\textbf{Assumptions on the stored energy:} The functional $\calI:
[0,T]\times W^{1,\qphi}(\Omega;\R^d) \times L^p(\Omega;\R^{d\times d})
\to (-\infty,\infty]$ has the form
\begin{equation}
  \label{stored-energy} \calI(t,\varphi,P) := \calE_1 (P) + \calI_2
  (t,\varphi,P) \text{ for }
  (t,\varphi,P) \in [0,T]\times W^{1,\qphi}(\Omega;\R^d) \times 
   L^p(\Omega;\R^{d\times d}). 
\end{equation}
We now list the conditions on $\calE_1$ and $\calI_2$ and postpone to
Example \ref{ex-energy} the discussion of a concrete energy $\calI$
complying with all of them.

First of all, $\calE_1 : 
 L^p(\Omega;\R^{d\times d}) \to (-\infty,\infty]$ is defined by
\begin{equation}
\label{ene1} \calE_1 (P) = \left\{
\begin{array}{lll}
\int_{\Omega} H(P(x),\nabla P(x)) \dd x & \text{if $P \in \calP$,}
\\
\infty  & \text{otherwise,}
\end{array}
\right.
\end{equation}
where the function $H: \R^{d\times d} \times \R^{d\times d \times
d} \to (-\infty,\infty]$  is given by
\begin{equation}
\label{form-of-h} H(P,A) = K(P) +  \frac1{\qg} |A|^{\qg} \quad
\text{with } \qg>d.
\end{equation}
We impose that  the hardening function   $K : \R^{d\times d} \to (-\infty,\infty]$
complies with the following conditions: 
\begin{align}
&\tag{3.$\mathrm{K}_1$} \label{kappa2}
\text{$K$ is of class   $\rmC^2$   in $\GLD$, and }
\\
&\label{kappa1}  \tag{3.$\mathrm{K}_2$} \exists\, C_1,\,C_2>0 \
\exists\,  \qp,\, \qga >d  \ \forall\, P \in \GLD\, : \ K(P) \geq
C_1 \left( |P|^{\qp} {+}\det(P)^{-\qga}\right) {-}C_2\,.
\end{align}

The functional $\calI_2: [0,T]\times W^{1,\qphi}(\Omega;\R^d) \times
 X
 \to (-\infty,\infty]$ is defined by
\begin{equation}
\label{I2} \calI_2 (t,\varphi,P): = \left\{
\begin{array}{lll}
\int_{\Omega} W(x,\nabla \varphi(x) \inve{P(x)})  \dd x \,  - \, 
\pairing{}{W^{1,\qphi}\OOd}{\ell(t)}{\varphi}
 & \text{if
$(\varphi,P) \in \calF \times \calP$,}
\\
\infty  & \text{otherwise,}
\end{array}
\right.
\end{equation}
where the external loading $\ell : [0,T] \to
W^{1,\qphi}(\Omega;\R^d)^*$ fulfills
\begin{equation}
\label{loading}
\tag{3.$\mathrm{L}$}
\ell \in \rmC^1 ([0,T];W^{1,\qphi}(\Omega;\R^d)^*).
\end{equation}
In most applications the time-dependent loading  consists  of   the work of
volume and surface forces, i.e.\ 
$  \pairing{}{W^{1,\qphi}\OOd}{\ell(t)}{\varphi}  =
\int_\Omega f(t,x) \cdot \varphi(x) \dd x + \int_{\Gamma_\neu} h(t,x)\cdot
\varphi(x) \dd x $ where  $\Gamma_\neu$  is  the Neumann part of the boundary
$\Gamma = \partial\Omega$.  

We require the following conditions on the elastic energy $W:
\Omega \times \R^{d \times d} \to [0,\infty]$:
\begin{align}
  \label{w0} \tag{3.$\mathrm{W}_0$} & \dom(W) = \Omega \times \GLD,
  \text{ i.e. }  W(x,F) = \infty \text{ for } \mathrm{det}F \leq 0
  \text{ for all $x \in \Omega$},
  \\
\label{w1} \tag{3.$\mathrm{W}_1$} &
\begin{aligned}
&
 \exists\, j \in L^1 (\Omega) \,
\exists\, \qf> d  \, \exists \, C_3>0 \,  \forall (x,F) \in \dom(W):
\, W(x,F) \geq j(x) + C_3 |F|^{\qf};
\end{aligned}
 \intertext{for all $x \in
\Omega$ the functional $W(x,\cdot): \R^{d\times d} \to
(-\infty,\infty]$ is \emph{polyconvex}, i.e.\ it is a convex function of its minors (cf.\ Notation
\ref{not:4.1}). Namely,
} & \label{w2}
\tag{3.$\mathrm{W}_2$}
\begin{aligned}
&
 \exists\, \mathbb W: \Omega \times \R^{\mu_d} \to
(-\infty,\infty] \ \text{such that}
\\ &
\text{(i)} \ \mathbb W \ \text{is  a normal integrand,}
\\
& \text{(ii)} \  \forall\, (x,F) \in \Omega \times \R^{d \times d}\,
:
 \ \ W(x,F)= \mathbb W(x,\mathbb M (F)),
\\
& \text{(iii)} \ \forall x\in \Omega\,:  \  \ \mathbb W(x,\cdot):
\R^{\mu_d} \to (-\infty,\infty] \text{ is convex,}
\end{aligned}
\intertext{ and $W$ also satisfies} & \label{w3}
\tag{3.$\mathrm{W}_3$}
\begin{aligned}
& \exists\, \delta>0 \    \ \exists\, C_4,\, C_5>0    \ \  \forall\, x\in \Omega \ \forall\, F \in \GLD \ \forall\, N \in \mathcal{N}_\delta\,:
\\
&  \text{(i)} \  \text{$W(x,\cdot) : \mathrm{GL}^+(d) \to \R$ is differentiable,} 
\\
& \text{(ii)} \ |\transp{F}\rmD_F W(x,F)| \leq C_4 (W(x,F) +1), \text{ and }
\\
& \text{(iii)} \ |\transp{F} \rmD_F W(x,F) -  \transp{(FN)} \rmD_F
W(x,FN) | \leq C_5 |N {-}\mathbf{1}| (W(x,F) +1).
\end{aligned}
\end{align}

Along the footsteps of \cite{MaiMie09GERI}, we will refer to
condition  \eqref{w3}(ii)   as a \emph{multiplicative stress control}, for
it involves the ``multiplicative''
 \begin{equation}
 \label{mandel}
 \text{Mandel stress tensor }\quad \Man{F} := \transp{F}\rmD_F W(F),
 \end{equation}
 which is estimated, together with its variation (cf.\
 \eqref{w3}(iii)),   in terms of the energy $W$. This condition goes back
 to   \cite{Ball84MELE,BaOwPh91MPAP}   and has been  recently  publicized in 
 \cite{Ball02SOPE}. In Section \ref{s:7}, we will see that the
 treatment of  time-dependent Dirichlet boundary conditions  requires a
 multiplicative stress control condition on the Kirchhoff
 stress tensor $S(F):= \rmD_F W(F) \transp{F}$. 
   Let us also mention that both the stress control
 conditions \eqref{w3} and their analogs for the Kirchhoff stress
 tensor are compatible with polyconvexity and with the physical
 feasibility requirements that $W(x,F) = \infty $ for $\det(F) \leq
 0$, and $W(x,F) \to \infty $ for $\det(F) \down 0$, cf.\
 Example \ref{ex-energy} ahead.

\begin{remark}
\label{rmk:ext-Ball}
\slshape In fact, in \cite{MaiMie09GERI} estimate  \eqref{w3}(iii)  was
required in a much more general form, involving on the right-hand side
the term $\omega(|N {-}\mathbf{1}| )$, with the modulus of continuity
$\omega : [0,\infty) \to [0,\infty)$ nondecreasing, and such that
$\lim_{\rho\down 0} \omega(\rho) =0$.  The discussion in Section
\ref{s:5} (cf.\ Remark \ref{rmk:after-prelim4}) will show that in the
present context it would be possible to work with the following
H\"older estimate,  generalizing   \eqref{w3}(iii):  
\begin{equation}
\label{Ball-weak}
 \ |\Manx{x}{F}-  \Manx{x}{FN } | \leq C |N
 {-}\mathbf{1}|^{\alpha_W} (W(x,F) +1) 
 \quad \text{for some $\alpha_{W} \in (0,1]$}.
\end{equation}
\end{remark}
\medskip

\noindent
{\bf Assumptions on the dissipation potential:} 
The dissipation potential density $\rmetric : \Omega \times \R^{d \times d} \to
[0,\infty) $ is required to satisfy 
\begin{align}
&\tag{3.$\mathrm{R}_1$} \label{R1}
  \rmetric : \Omega
  \times \R^{d \times d} \to [0,\infty)
   \text{ is  a normal integrand,}
\\
&  \tag{3.$\mathrm{R}_2$} \label{R2}
\begin{aligned}
 \foraa   x \in \Omega: \quad   &  \rmetric(x,\cdot) : \R^{d
\times d} \to [0,\infty) \text{ is convex, with
$\rmetric(x,0)=0$, and}
\\
& \rmetric^*(x,\Xi_1) = \rmetric^*(x,\Xi_2) \ \text{for
all }  \Xi_j \in \partial \rmetric(x,V), \, j=1,2,
\text{ and } V \in \R^{d \times d},
\end{aligned}
\intertext{and  $\rmetric(x,\cdot)$ 
has a superlinear growth at infinity, uniformly with respect to  $x$, i.e.} &
 \tag{3.$\mathrm{R}_3$} \label{R3}
\begin{aligned}
\exists\, p\in (1,\infty)\, \ 
\exists\, C_{R}^1 \,  C_{R}^2,\,  C_{R}^3,\,  C_{R}^4     >0 \quad
   & \foraa   x \in \Omega \
\  \forall\, V \in\R^{d \times d}\, : \\ & 
C_R^1 |V|^p - C_R^2 \leq \rmetric(x,V) \leq C_R^3 |V|^p + C_R^4 .
\end{aligned}
\end{align}
Notice that in \eqref{R2} the symbols $\rmetric^* (x,\cdot)$ denotes
the Fenchel-Moreau conjugate of $\rmetric(x,\cdot)$, and that
\eqref{R3} implies
\begin{equation}
\label{growth-Rstar}
\widetilde{C}_R^1 |\Xi|^{p'} - \widetilde{C}_R^2 \leq \rmetric^*(x,\Xi) \leq
\widetilde{C}_R^3 |\Xi|^{p'} + \widetilde{C}_R^4  
\end{equation}
for some $\widetilde{C}_{R}^j >0$, where $p' =
\frac{p}{p-1}$.  In Example \ref{ex:R} we shall provide a normal
integrand $\rmetric :\Omega \times \R^{d \times d} \to [0,\infty)$
complying with \eqref{R1}--\eqref{R3}.\medskip

\noindent
\textbf{Setup of the generalized gradient system: energy and dissipation:}
Starting from the stored energy functional $\calI$ and from the
dissipation metric $\rmetric$, we now introduce the energy functional
$\en$ and the (Finsler) family of dissipation potentials
$(\Psi_P)_{P\in \domain}$ that will allow us to formulate the PDE
system \eqref{pde} as the abstract doubly nonlinear evolution equation
\eqref{abstract-dne} in the state space $X:=
L^p(\Omega;\R^{d\times d}).$

We consider the \emph{reduced} energy functional $\en: [0,T] \times X
\to (-\infty,\infty]$ obtained by minimizing out the deformations from
$\calI$.  Recalling $\calF$ from \eqref{admissible-phi} we set
\begin{equation}
\label{reduced-energy} \en (t,P) : = \inf\{ \calI(t,\varphi,P)\, : \
\varphi \in \calF\}.
\end{equation}
It is not difficult to check that $\dom(\en)$ is of the form $[0,T]
\times \domain$ for some $\domain \subset X $.  
 For this, we note that 
 $\en$ is given
by the sum of the functional $\calE_1$, which is independent of the
time-variable, and of a time-dependent, reduced functional $\calE_2
: [0,T]\times   X  
\to (-\infty,\infty]$,
 i.e.
\begin{equation}
\label{sum-energies} \en(t,P) = \calE_1 (P) +
\inf\{\calI_2(t,\varphi,P)\, : \ \varphi \in \calF \}  =:  \calE_1
(P) + \calE_2(t,P).
\end{equation}
In Lemma \ref{l:prelim1} we shall prove that the functional
$\calE_2$ is bounded from below. Therefore, $\en (t,P)<\infty$
implies $\calE_1 (P)<\infty$. Then,
 it follows from \eqref{form-of-h}, \eqref{kappa1}, and Poincar\'e's inequality that,
if $P \in \domain$, then $P \in W^{1,\qg}(\Omega;\R^{d\times d})$.
Since $\qg >d$,  this implies that $P \in \rmC^0
(\overline{\Omega};\R^{d\times d})$, hence $\cof(P) \in \rmC^0
(\overline{\Omega};\R^{d\times d})$. Combining this with formula
\eqref{cofactor-repre} and with the fact that $\det(P)^{-1} \in
L^{\qga}(\Omega)$ in view of \eqref{kappa1}, we conclude that
\begin{equation}
\label{domain} \domain \subset \left\{ P \in \calP\, : \inve{P} \in
 L^{\qga}(\Omega;\R^{d\times d})\right\}.
\end{equation}
In fact, we will even prove that any $P \in \domain $ satisfies
$\inve{P} \in \mathrm{C}^0(\overline{\Omega};\R^{d\times d})$, see
\eqref{enhanced-domain} ahead. 

We are now in the position to introduce the dissipation potentials $
\Psi_P : X \to [0,\infty) $, for $P \in \domain$. We set
\begin{equation}
\label{rdiss-def}
 \rdiss(x,P,V)
:= \rmetric(x,V \inve{P}) \quad \foraa   x \in \Omega, \ \text{for
all } (P,V) \in \GLD \times \R^{d \times d},
\end{equation}
and define
\begin{equation}
\label{def-Psip}
\begin{gathered}
\Psi_P (V): = \int_{\Omega} \rdiss(x,P(x),V(x))  \dd x  \quad \text{for
all $(P,V) \in \domain \times  X $}.
\end{gathered}
\end{equation}
Notice that the above formula is well-defined thanks to
\eqref{domain}.

While postponing to Section \ref{s:5} a thorough discussion of
assumptions \eqref{kappa2}--\eqref{kappa1}, \eqref{w1}--\eqref{w3},
and \eqref{R1}--\eqref{R3}, let us mention here that they 
ensure 
that the energy  $\en$ and the dissipation
potentials $(\Psi_P)_{P \in \domain}$ comply with the abstract
assumptions of Theorem \ref{thm:abstract-2}.  Thus we will deduce the
existence of EDI solutions (in the sense of Definition
\ref{def:weak-sols}) to \eqref{abstract-dne}.
\medskip

\noindent
{\bf Statement of the main result:}
In addition to  the assumptions listed above, 
Theorem
\ref{th:EDI.Viscopl} below  requires  two further conditions
suitably relating the growth exponents involved therein.

\begin{theorem}[Existence of EDI solutions for viscoplasticity]
\label{th:EDI.Viscopl}
Let the energy $\calI$ and the dissipation potentials $(\Psi_P)_{P \in \domain}$  fulfill
\eqref{kappa2}--\eqref{kappa1}, \eqref{loading},
\eqref{w1}--\eqref{w3}, and \eqref{R1}--\eqref{R3}.  Suppose in
addition that the exponents $\qphi,\, \qp, \, \qf, \, \qg, \, \qga$ in
\eqref{form-of-h}, \eqref{kappa1}, \eqref{loading}, and \eqref{w1}
comply with
\begin{align}
& \label{index-1} \frac1{\qphi} = \frac1{\qf} + \frac1{\qp}, \quad
\qphi > d,
\\
 & \label{index-2} \tilde{q} >d,  \ \text{where } \
 \frac1{\tilde{q}} := \frac2{\qga} + \frac1{\qg}.
\end{align}

Then, for every initial datum
\begin{equation}
\label{condition-datum}  P_0 \in \calP 
\end{equation}
there exist an EDI solution $(\varphi,P):[0,T]\to
W^{1,\qphi}(\Omega;\R^{d}) \times W^{1,\qg}(\Omega;\R^{d\times d})$ of
the associated generalized gradient system. These solutions satisfy $P
\in L^\infty (0,T;W^{1,\qg}(\Omega;\R^{d\times d})) \cap W^{1,p}(0,T;
L^p(\Omega;\R^{d \times d})) $ and $\varphi \in L^\infty
(0,T;W^{1,\qphi}(\Omega;\R^{d}))$. Moreover, $\Xi$ given by
\[
{\Xi}(t,x) =-\Delta_{\qg} P(t,x )  + \rmD K( P(t,x))
 + \Back{\nabla \varphi (t,x)}{P(t,x)} 
\]
(recalling the notation $\Back{\nabla \varphi }{P} = \Man{\nabla
  \varphi \inve{P}} \transpi{P}$), lies in $L^{p'}(0,T;L^{p'}
(\Omega;\R^{d \times d}))$ and the triple $(P,\varphi,\Xi)$ satisfies
the EDI
\begin{align}
 \nonumber
  \calI(t,\varphi(t),P(t)) + &\int_s^t \int_{\Omega}
  \left(\rmetric(x,\dot{P}(r,x) \inve{P(r,x)})+\rmetric^*(x, 
  ({-}{\Xi}(r,x)) \transp{P(r,x)})\right) \dd x \dd r
  \\ 
  &\label{concrete-enid}
  \leq \calI(s,\varphi(s),P(s)) -\int_s^t
   \pairing{}{W^{1,\qphi}\OOd}{\dell(r)}{\varphi(r)}\dd r
\end{align}
for all $t\in (0,T]$, almost all $s\in (0,t)$, and $s=0$. 
 \end{theorem}
We conclude by discussing 
 the 
assumptions \eqref{kappa2}--\eqref{kappa1},
\eqref{w1}--\eqref{w3}, and \eqref{R1}--\eqref{R3} for  two concrete
examples.

\begin{example}[Elastoplastic energy]
\slshape\label{ex-energy} 
We consider the  hardening  function $K: \R^{d\times d} \to [0,\infty]$
\begin{equation}
\label{ex-K} K(P) =\left\{ \begin{array}{ll} c_1 |P|^{\qp} + c_2
\left| \frac1{\det(P)} \right|^{\qga}  & \text{for } P \in \GLD,
\\
\infty & \text{otherwise}
\end{array}
\right.
\end{equation}
with $c_1,\, c_2>0$  and  $\qp,\,  \qga>d$.  
An admissible choice for  the elastic stored energy density  $W$ (focusing for simplicity on the
case of spatially-homogeneous materials) is
\begin{equation}
\label{ex-W} W(F) = \left\{ \begin{array}{ll} c_3 |F|^{\qf} + c_4
\left| \frac1{\det(F)} \right|^{\eta}  & \text{for } F \in \GLD,
\\
\infty & \text{otherwise}
\end{array}
\right. \quad \text{with }  c_3,\, c_4>0 \text{ and }  \qf>d, \, 
\eta>0.
\end{equation}
In addition, we suppose that the exponents $\qf,\, \qga, \, \qp$
fulfill \eqref{index-1}--\eqref{index-2}.  For instance, we may choose $\qphi=d+1$, $\qf=\qp=\qg=2d+2$, and $\qga = 4d+4$. 

Following the discussion in \cite[Example 3.2]{MaiMie09GERI}, we
observe that $W$ complies with \eqref{w1}--\eqref{w2}. In order to
check the Mandel stress control  \eqref{w3}, taking
into account \eqref{cof-der} we calculate $\rmD_F W(F) = c_3 \qf
|F|^{\qf-2} F - c_2\eta (\det(F))^{-\eta-1} \cof(F)$. Therefore
\begin{equation}
\label{mandel-explicit} \Man{F} = \transp{F} \rmD_F W (F) = c_3 \qf
|F|^{\qf-2} \transp{F} F - c_2\eta \det(F)^{-\eta}
\underbrace{\transp{F} \det(F)^{-1} \cof(F)}_{= \mathbf{1} \text{ by
  \eqref{cofactor-repre}}}\,. 
\end{equation}
Using \eqref{mandel-explicit} and arguing  in  the same way as in
\cite[Example 3.2]{MaiMie09GERI}, we conclude \eqref{w3}.
\end{example}

\begin{example}[Viscoplastic dissipation]\label{ex:R} \slshape  
Given two measurable functions
\begin{equation}
\label{sigma-nu}
\begin{aligned}
 & \sigma_{\mathrm{yield}}: \Omega \to (0,\infty), \ \text{such that } \exists\,
  \sigma_*>0  \ \foraa   x \in \Omega\, : \ \frac1{\sigma_*} \leq \sigma(x) \leq \sigma_*,
 \\
  & \nu : \Omega \to (0,\infty), \ \text{such that } \exists\,
  \nu_*>0  \ \foraa   x \in \Omega\, : \ \nu(x) \geq \nu_*,
\end{aligned}
\end{equation}
we set
\begin{equation}
\label{def-R} \rmetric(x,V) := \sigma_{\mathrm{yield}}(x) |V| +
\frac{\nu(x)}{2} |V|^2.
\end{equation}
It can be easily computed that
\begin{equation}
\label{def-R-conj}   \rmetric^* (x,\Xi) = \frac1{2\nu(x)}
\mathrm{dist}^2 (\Xi,\mathrm{E}(x)) \quad \foraa   x \in \Omega \quad
\text{for all $\Xi \in \R^{d \times d}$,}
\end{equation}
where $\mathrm{E}(x)$ is the \emph{elastic domain}
$
\mathrm{E}(x) = \{ \Xi \in \R^{d \times d} \, : \ |\Xi| \leq
\sigma_{\mathrm{yield}}(x) \},$
and  $\mathrm{dist} (\Xi,\mathrm{E}(x)) = \inf_{\Theta \in
\mathrm{E}(x)} |\Xi-\Theta| = \min_{\Theta \in \mathrm{E}(x)}
|\Xi-\Theta| $.   In this case it is immediate to check that
\eqref{R1}--\eqref{R3}, with $p=2$, are verified.

In engineering one often uses dissipation distances 
$R(x,V)= \sigma_{\mathrm{yield}}(x) |V| + \frac{\nu(x)}{p} |V|^p$
with $p=1+m$ and $0<m \ll 1$,
see e.g.\ \cite[Sec.\,4.2]{ZRSRZ96TDIT} where $m=0.012$ (which means
$p=1.012$ and $p'=84.33$). Note that in
that work the resolved shear stress $\tau_\alpha$ already includes
the yield stress of the slip system $\alpha$. 
\end{example}

\section{Proof of Theorem \ref{th:EDI.Viscopl}}
\label{s:5} 
We will derive
Theorem \ref{th:EDI.Viscopl} 
as  a consequence of the abstract Theorem \ref{thm:abstract-2},
applied in this  functional setting:
\begin{enumerate}
\item the state space $ X $ is $L^p (\Omega;\R^{d \times
    d})$ with $p$ from \eqref{R3};
\item the energy functional is $\en$ from  \eqref{reduced-energy};
\item the dissipation potentials $(\Psi_P)_{P \in \domain}$ are given
  by \eqref{def-Psip}.
\end{enumerate}
It will turn out to be particularly effective (cf.\ the discussion in
\cite[Sec.\,3]{MiRoSa13NADN}) to work with a notion of subdifferential
for the functional $\en(t,\cdot)$ tailored to the fact that $\en$
arises from a minimization procedure, and that the related set of
minimizers is nonempty.  Thus, we choose $\diff tP$ at a point $(t,P)
\in [0,T]\times \domain$ to be the collection of the G\^ateau
derivatives of $\{ \rmD_P \calI(t,\varphi,P) \}$ with $\varphi $ in
\begin{equation}
\label{set-of-minimizers}
\mins tP:= \mathrm{Argmin}\{ \,\calI(t,\varphi,P)\,:\ \varphi \in \calF \,\}
\end{equation}
(which is nonempty, cf.\ Lemma \ref{l:prelim1} below).  More precisely, we set 
\begin{enumerate}
  \setcounter{enumi}{3}
\item $\diffname: [0,T]\times   X 
  \rightrightarrows X^* $ 
  is the
  \emph{marginal subdifferential} (with respect to the variable $P$) of $\en$,
  defined by
  \begin{equation}
  \label{subdiff-notation} 
   \begin{aligned} \diff tP &:= \{  \rmD_P \calI(t,\varphi,P)\, : \
     \varphi \in \mins tP \}\\ 
     &\ =  -\Delta_{\qg} P +  \rmD K(P) + 
     \left\{ 
    \Backx{\cdot}{\nabla\varphi}{P}: \ \varphi\in\mins
       tP\right\}, 
   \end{aligned}
\end{equation}
(cf.\ also \eqref{more-precise-subdif} ahead), recalling that $
\Backx{\cdot}{\nabla\varphi}{P} = \Manx{\cdot}{\nabla \varphi
  \inve{P}} \transpi{P} \\ = \transp{(\nabla\varphi \inve{P})} \rmD_F
W(\cdot,\nabla\varphi \inve{P} ) \transpi{P}$.
\end{enumerate}
Finally, as suggested by \cite{MiRoSa13NADN} we will work with the
following surrogate notion of $\partial_t \en$, namely the power
functional:
\begin{enumerate}
 \setcounter{enumi}{4}
\item
   $\Ptname : \mathrm{graph}(\diffname) \to \R$  is defined  via
\begin{equation}
  \label{def-P-hat} \Pt tP{\Xi} :=
  \sup_{\phimin \in \Rt tP{\Xi}} \partial_t \calI(t,\varphi, P)
  = \sup_{\phimin \in \Rt tP{\Xi}}
  \pairing{}{W^{1,\qphi}\OOd}{-\dell(t)}{\phimin},
\end{equation}
where for all $ (t,P,\Xi) \in \mathrm{graph}(\diffname)$ we set
  \begin{align}
  \label{def-RT}
  \Rt t P \Xi &:= \{\:\phimin \in \mins tP\, : \ \Xi =  -\Delta_{\qg} P
  +  \rmD K(P) +  \gateaux{\cdot}{\phimin}{P}\:\}. 
 \end{align}
\end{enumerate} 
By construction $  \Rt t P \Xi $ is nonempty on 
$ \mathrm{graph}(\diffname)$. 

In the following series of lemmas we will show that the generalized
gradient system $(\amb,\en, \Psi, \diffname, \Ptname)$ complies with
assumptions and \eqref{basic-1}--\eqref{eq:468} and
\eqref{eq:psi-sum-1}--\eqref{eq:psi-uni}.  In all of the ensuing
results, we will specify which of conditions
\eqref{kappa2}--\eqref{kappa1}, \eqref{loading},
\eqref{w1}--\eqref{w3}, \eqref{R1}--\eqref{R3}, and
\eqref{index-1}--\eqref{index-2} actually come into play in the
corresponding proofs.  For simplicity, from now on we will omit to
explicitly indicate the $x$-dependence in $\rmetric$, $W$, and the
related quantities.  In fact, since we do not use any higher order
regularity, it is obvious that the general $x$-dependent case works as
well as the homogeneous case.  Furthermore, we will use the symbols
$c$, $C$, whose meaning may vary from line to line, for generic
positive constants depending on known quantities.

\subsection{Basic properties of $\en$}
\label{ss:4.1}
Preliminarily, we discuss the properties of the energy functional
$\calE_1$ from \eqref{ene1}.

\begin{lemma}[Coercivity properties of $\calE_1$]
\label{l:prelim0} Assume \eqref{loading},
 \eqref{w1}--\eqref{w2}, and \eqref{index-1}--\eqref{index-2}. Then,  
 \begin{equation}
 \label{prelimin-dom-en1}
 \mathrm{\dom}(\calE_1(t,\cdot))  \subset \{ P \in \calP\, : \ \inve{P} \in  W^{1,\tilde{q}}(\Omega;\R^{d\times d})\}
 \qquad \text{for every } t \in [0,T],
 \end{equation}
  and 
 \begin{equation}
 \label{added-property-for-E1}
 \begin{aligned}
 \!  \!  \!  \!  \!  \!  
 \forall\, \SE>0 \ \exists\, R>0 \ \forall\, (t,P) \in
\mathrm{dom}(\calE_1):
\   \en_1(t,P) \leq \SE \ \Rightarrow \
 \| P \|_{W^{1,\qg}}+ \|\inve{P}
 \|_{W^{1,\tilde{q}} } \leq R. 
 \end{aligned}
 \end{equation}
\end{lemma}
\begin{proof}
It follows from \eqref{form-of-h} and \eqref{kappa1} that for every $P \in \calP$
\[
\en_1(t,P)  \geq \int_\Omega C_1 \left( |P|^\qp + \det(P)^{-\qga} \right)  + \tfrac1{\qg} |\nabla P|^\qg  \dd x - C\,.
\]
Therefore, the  estimate for  $ \| P \|_{W^{1,\qg}} $ follows from  Poincar\'e's inequality, cf.\ also the arguments leading to \eqref{domain}. In order to prove the bound for  $ \|\inve{P}
 \|_{W^{1,\tilde{q}} } $,
 we compute  
 \begin{equation}
\label{nabla-inve} \nabla \inve{P} = -\inve{P} \nabla P \inve{P},
\end{equation}
 use that $   \en_1(t,P) \leq \SE $ yields an estimate for $\| \inve{P} \|_{L^{\qga}}$  (cf.\ again \eqref{domain}), and combine it with the bound for $\| \nabla P \|_{L^{\qg}}$ via 
 H\"older's inequality, relying on the condition that $\tfrac{2}{\qga} + \tfrac1{\qg} = \tfrac1{\tilde q}$ from  \eqref{index-2}.
 This provides a bound for $\| \nabla P \|_{L^{\tilde q}}$. 
  Then, with  Poincar\'e's inequality we obtain the desired estimate for $  \| \inve{P} \|_{W^{1,\tilde q}} $, and  \eqref{added-property-for-E1} ensues. 
\end{proof} 
 
We now turn to the reduced energy functional $\calE_2$ from
\eqref{sum-energies}.

\begin{lemma}[Existence of minimizers for $\calI_2$]
\label{l:prelim1} Assume \eqref{loading},
 \eqref{w1}--\eqref{w2}, and \eqref{index-1}--\eqref{index-2}. Then,
 $\calI_2$ defined in \eqref{I2}  is coercive, namely
\begin{align}
  &\label{esti-1}
   \begin{aligned}
   &\! \! \! \! \! \!   \forall\, \varrho>0 \  \exists\, c_1, \, c_2>0
   \   \forall\, 
    (t,\varphi, P)\in [0,T]\times W^{1,\qphi}(\Omega;\R^{d}) \times
    L^p(\Omega;\R^{d \times d}) \text{ s.t. } \mathrm{det} P>0 \text{
      a.e.\ in } \Omega\, :  
    \\  
   &\quad \qquad  \calI_2(t,\varphi, P) \geq  c_1\|\varphi\|_{W^{1,\qphi}}^{\qphi}  
    -\varrho\, \|P \|_{L^{\qp}}^{\qp} -c_2.
 \end{aligned}
\end{align}
Moreover, for all $(t,P) \in [0,T] \times \domain $ the set $\mins tP$
of minimizers, namely
\begin{equation}
\label{neq-empty} \mins tP: = \argmin \calI_2(t,\cdot,P) = \{ 
\varphi \in \calF\, : \ \varphi \text{ minimizes } \calI_2(t,\cdot,P) \text{ on }\calF \}
\end{equation}
is nonempty, weakly closed in $ W^{1,\qphi}(\Omega;\R^{d}) $, and contained in the 
$W^{1,\qphi}$-ball of radius 
$
\hat{r}(P):=c_3 (1+ \| P \|_{L^\qp}^{\qp})^{1/\qphi}$. 
\end{lemma}
\begin{proof} 
 In order to prove 
 coercivity, we use
   \eqref{w1} and find 
\begin{align}
\label{4.1-1}\calI_2(t,\varphi,P) &\geq  
 \int_{\Omega}
  j(x)  \dd x  +C_3 \int_{\Omega} |\nabla\varphi
    \inve{P}|^{\qf}  \dd x  -
   \pairing{}{W^{1,\qphi}\OOd}{\ell(t)}{\varphi}.
\end{align} 
Now, we observe that
\[
\begin{aligned}
\int_{\Omega} |\nabla \varphi|^{\qphi}  \dd x  = \int_\Omega
\frac{|\nabla \varphi|^{\qphi}}{|P|^{\qphi}}  |P|^{\qphi}
\dd x   & \leq \left( 
\int_{\Omega} \frac{|\nabla \varphi|^{\qf}}{|P|^{\qf}}   \dd x 
\right)^{\qphi/\qf} \left(\int_{\Omega} |P|^{\qp}  \dd x 
\right)^{\qphi/\qp} \\ & \leq \left(\int_{\Omega} |\nabla\varphi
\inve{P}|^{\qf}  \dd x  \right)^{\qphi/\qf} \left(\int_{\Omega}
|P|^{\qp}  \dd x  \right)^{\qphi/\qp},
\end{aligned}
\]
where the second inequality follows from  H\"older's inequality
(applied  to $(|\nabla \varphi|/|P|)^{\qphi}$ and
$|P|^{\qphi}$ with exponent $r = \qf/\qphi$, 
 and the  conjugate exponent $r'$ is $r'= \qf/{(\qf-\qphi)} = \qp/\qphi$ by
virtue of \eqref{index-1}), and the third estimate is due to the
first of \eqref{useful-later}. Then, 
for the  second  term on the right-hand side of 
\eqref{4.1-1}
 for every $C>0$  we have 
\begin{equation}
\label{4.1-2} \int_{\Omega} |\nabla\varphi \inve{P}|^{\qf}  \dd x  \geq
\frac{\|\nabla \varphi \|_{L^{\qphi}}^{\qf}}{\|P\|_{L^{\qp}}^{\qf}} \geq C
\frac{\qf}{\qphi}\|\nabla \varphi \|_{L^{\qphi}}^{\qphi} - \frac{\qf}{\qp}
C^{\qp/\qphi}\|P\|_{L^{\qp}}^{\qp},
\end{equation}
the latter estimate due to Young's inequality (cf.\ the second of
\eqref{useful-later}). Then for   a given  $\rho>0$, we choose the
constant $C$ in such a way that $ \frac{C_3}2 \frac{\qf}{\qp}
C^{\qp/\qphi}=\rho$ and we combine estimate \eqref{4.1-2} with
\eqref{4.1-1}.  Taking into account Korn's and Young's inequalities,
we obtain
\[
\begin{aligned}
 \calI_2(t,\varphi,P)
 &  \geq
  \int_{\Omega} j(x)  \dd x  + C_\rho \|
\varphi\|_{W^{1,\qphi}}^{\qphi} -  \rho
\|P\|_{L^{\qp}}^{\qp} - \| \ell(t)
\|_{W^{1,\qphi}(\Omega;\R^{d})^*} \|
\varphi\|_{W^{1,\qphi}} \\ & \geq 
\frac{C_\rho}2 \|
\varphi\|_{W^{1,\qphi}}^{\qphi}  - \rho
\|P\|_{L^{\qp}}^{\qp} -C' \left( 1+ \|
\ell\|_{L^\infty (0,T;W^{1,\qphi}(\Omega;\R^{d})^*)}^{\qphi'}\right).
\end{aligned}
\]
Then,
 \eqref{esti-1}  follows.

 In order to prove the existence of minimizers, see \eqref{neq-empty}, 
we preliminarily observe that
\begin{equation}
\label{enhanced-domain}
 \domain \subset \left\{ P \in \calP\, : \inve{P} \in
 W^{1,\tilde{q}}(\Omega;\R^{d\times d})\right\}.
\end{equation}
This follows from \eqref{prelimin-dom-en1} and the fact that
$\calI_2$, hence $\calE_2$, is bounded from below, cf.\
\eqref{esti-1}.  Hence, let us consider an infimizing sequence
$(\varphi_k) \subset \calF$ for the minimum problem in
\eqref{sum-energies}. Estimate \eqref{esti-1} implies that
$(\varphi_k)$ is bounded in $W^{1,\qphi}(\Omega;\R^d)$, so that there
exists $\varphi \in W^{1,\qphi}(\Omega;\R^d)$ such that, up to a
subsequence, $\varphi_k \weakto \varphi$ in
$W^{1,\qphi}(\Omega;\R^d)$. Because of \eqref{enhanced-domain}, we
then have $\nabla \varphi_k \inve{P} \weakto \nabla \varphi \inve{P}$
in $L^{\qphi}(\Omega;\R^{d\times d})$ for every $P \in \domain$, since
$\inve{P} \in \rmC^0 (\overline{\Omega};\R^{d\times d}) $ in view of
\eqref{enhanced-domain}.  Hence, the weak continuity of minors of
gradients (cf.  \cite{Resh67SCMM,Ball76CCET}) for $\qphi>d$ and the
Cauchy-Binet relations give
 \begin{equation}
 \label{Cauchy-Binet}
\mathbb M_s (\nabla\varphi_k \inve{P})=
\mathbb{M}_s(\nabla\varphi_k )\mathbb{M}_s(\inve{P})
 \weakto  \mathbb{M}_s(\nabla\varphi)\mathbb{M}_s(\inve{P}) = \mathbb M_s (\nabla\varphi \inve{P}) 
 \end{equation}
  in $L^{\qphi/s}(\Omega;\R^{\binom{d}{s}\times \binom{d}{s}})$,  for
 all $s  \in \{1,\ldots, d\}$.
Ultimately,
\[
\mathbb M (\nabla\varphi_k \inve{P}) \weakto \mathbb M(\nabla\varphi \inve{P})
 \quad \text{
 in $L^{1}(\Omega;\R^{\mu_d})$.}
\]
Then, arguing in the same way as in the proof of \cite[Thm.\
5.2]{MaiMie09GERI} to adapt classical lower semicontinuity arguments
for Carath\'eodory integrands (cf.\ e.g.\ \cite{Eise79SLSM} or
\cite{Stru90VMAN}) to the \emph{normal integrand} $\mathbb W$, we
conclude in view of \eqref{w2} that
\begin{equation}
\label{eisen}
\begin{aligned}  \liminf_{k \to \infty} \int_{\Omega}
  W(\nabla\varphi_k(x) \inve{P(x)})  \dd x & =  \liminf_{k \to \infty}
  \int_{\Omega} \mathbb W(\mathbb M(\nabla\varphi_k(x) \inve{P(x)}))\dd x 
 \\ & \geq
  \int_{\Omega} \mathbb W(\mathbb M(\nabla\varphi(x) \inve{P(x)}))  \dd x 
\end{aligned}
\end{equation}
Therefore, the direct method in the calculus of variations yields that
$\varphi \in \calF$ is a minimizer for $\calI_2 (t,\cdot,P)$,
contained in the $W^{1,\qphi}$-ball of radius $\hat{r}(P)$ in view of
\eqref{esti-1}.

These arguments also show that, up to a subsequence any sequence
of minimizers weakly converges in $ W^{1,\qphi}(\Omega;\R^{d}) $ to a
minimizer.  Hence $\mins tP$ is weakly sequentially compact in $
W^{1,\qphi}(\Omega;\R^{d}) $.
\end{proof}

From Lemma \ref{l:prelim1} we now deduce that $\en: [0,T] \times X \to
(-\infty,\infty]$ defined by \eqref{reduced-energy} is lower
semicontinuous and bounded from below as required by \eqref{basic-1}
(in fact, we will prove that $\en$ is bounded from below by a negative
constant but, as previously mentioned, we can always reduce to a
positive lower bound by adding a positive constant), and that it
complies with conditions \eqref{eq:17-bis} and \eqref{eq:diffclass_a}.

\begin{lemma}[Coercivity, lower semicontinuity, and time-dependence of $\en$]
  \label{l:prelim2} Assume 
  \\ \eqref{kappa2}--\eqref{kappa1}, \eqref{loading},
  \eqref{w1}--\eqref{w2}, and \eqref{index-1}--\eqref{index-2}. Then,
  for every $t\in [0,T]$ the energy $\en(t,\cdot)$ is lower
  semicontinuous on $X$. 
 Furthermore,
 \begin{equation}
\label{esti-2} \begin{aligned}
   &    \exists\,  c_4,\, c_5>0  \  \   \forall\, (t,P) \in
[0,T] \times \domain\ \forall\, \varphi\in \cM(t,P)\, :   \\
& \begin{aligned}  \quad \en(t,P) \geq c_4 
 \Big(   \| \nabla P \|_{L^{\qg}}^{\qg} 
  + \| P \|_{L^{\qp}}^{\qp} + \| \det(P)^{-1}
  \|_{L^{\qga}}^{\qga}  +\int_\Omega W(\nabla \varphi\,\inve{P})
   \dd x   \Big) -c_5. 
\end{aligned}
\end{aligned}
 \end{equation}
In particular, $\en$  complies with the
coercivity \eqref{eq:17-bis}. It also  satisfies
\begin{equation}
\label{conse-2}
\begin{aligned} &\forall\, \SE>0 \ \exists\, R>0 \ \forall\, (t,P) \in
[0,T] \times \domain:\\
&\qquad  \ \en(t,P) \leq \SE \ \Rightarrow \
 \| P \|_{W^{1,\qg}}+ \|\inve{P}
 \|_{W^{1,\tilde{q}} } \leq R. 
\end{aligned}
\end{equation}
Finally, $\en$ complies with  condition \eqref{eq:diffclass_a}.
\end{lemma}
\begin{proof} Estimate \eqref{esti-2} follows from \eqref{form-of-h}
  and \eqref{kappa1}, combined with \eqref{esti-1} (which clearly
  yields a lower bound for $\calE_2$), in which for instance we choose
  $\rho = C_1/4$.

Estimate \eqref{conse-2} immediately follows from the analogous
property \eqref{added-property-for-E1} for $\calE_1$ and the fact that
$\calE_2$ is bounded from below thanks to \eqref{esti-1}. 

In order to check the lower semicontinuity of $\ene t{\cdot}$, let
$P_n \to P $ in $L^p(\Omega;\R^{d\times d})$ with $\liminf_{n\to
  \infty} \en(t,P_n) <\infty$.  Up to a subsequence we have $\sup_n
\en(t,P_n) \leq C$, hence from \eqref{conse-2} we deduce that
\begin{equation}
\label{weak-w1q} P_n \weakto P \ \text{ in
$W^{1,\qg}(\Omega;\R^{d\times d}).$}
\end{equation}
Since $\qg >d$ and $\tilde{q}>d$ by \eqref{index-2}, we have
\begin{equation}
\label{important-convergences}\text{ $P_n \to P$  \ \  and
 \ \ $\inve{P}_n = (\det(P_n))^{-1} \cof(P_n) \to (\det(P))^{-1}
\cof(P)=\inve{P}$}
\end{equation}
 in $\rmC^0 (\overline{\Omega};\R^{d\times d})$, so that 
\begin{equation}
\label{weak-w1-tildeq} 
 \inve{P}_n \weakto
\inve{P} \ \text{ in $W^{1,\tilde{q}} (\Omega;\R^{d\times d}) $.}
\end{equation}
Using \eqref{weak-w1q}--\eqref{weak-w1-tildeq} and relying on
\eqref{kappa2} we immediately obtain $ \liminf_{n\to \infty}
\en_1(P_n) \geq \en_1(P).  $ Furthermore, let us choose a sequence
$(\varphi_n)_n$ with $\varphi_n \in \mins t{P_n}$. From $ \| \varphi_n
\|_{W^{1,\qphi}} \leq c_3 (1+ \| P_n \|_{L^\qp}^{\qp})^{1/\qphi} $
(cf.\ Lemma \ref{l:prelim1}), we have (after choosing a not relabeled
subsequence) that $\varphi_n \weakto \widetilde\varphi$ in
$W^{1,\qphi}(\Omega;\R^d)$ for some $ \widetilde\varphi$. Now, using
that $ \inve{P}_n \to \inve{P} $ in
$\rmC^0(\overline{\Omega};\R^{d\times d})$ by \eqref{weak-w1-tildeq},
we have that $\mathbb{M}_s(\inve{P}_n) \to \mathbb{M}_s(\inve{P})$ in
$\rmC^0$ and thus conclude (cf.\ \eqref{Cauchy-Binet}) that $\mathbb
M_s (\nabla\varphi_n \inve{P}_n) \weakto \mathbb M_s
(\nabla\widetilde\varphi \inve{P}) $ in
$L^{\qphi/s}(\Omega;\R^{\binom{d}{s}\times \binom{d}{s}})$, for all $s
\in \{1,\ldots, d\}$. Thus,

\begin{equation}
\label{liminf-nk}
\begin{aligned}
\en_2(t,P) \leq \calI_2 (t,\widetilde \varphi,P) \leq \liminf_{n \to \infty} \calI_2 (t,\varphi_n,P_n) = 
 \liminf_{n \to \infty} \en_2(t,P_n)\,.
\end{aligned}
\end{equation} 
Therefore, $\en(t,\cdot) $ is lower semicontinuous. 

Finally, to prove~\eqref{eq:diffclass_a}, we observe that for all
$P\in \domain$, for every $0 \leq s \leq t \leq T$ and every
$\varphi_s \in \mins sP$ there holds
\begin{equation}
\label{e:absolute}
\begin{aligned}\ene tP - \ene sP   &= \enei2tP - \enei2sP  
  \\
  & \leq \calI_2(t,\varphi_s,P) - \calI_2(s,\varphi_s,P)
  \\ & =  -\pairing{}{W^{1,\qphi}\OOd}{\ell(t)-\ell(s)}{\varphi_s}\\
  &
  \leq \| \ell(t) -\ell(s) \|_{W^{1,\qphi}(\Omega;\R^d)^*} \|\varphi_s
  \|_{W^{1,\qphi}\OOd} \\ 
  & \leq \| \dell\|_{L^\infty (0,T;W^{1,\qphi}(\Omega;\R^d)^*)}\,
  |t{-}s|\, c_3^{-1/\qphi}(\ene sP +c_4)^{1/\qphi} \\ & \leq C |t{-}s|
  (\ene sP + 1),
\end{aligned}
\end{equation}
where we have used~\eqref{loading} and \eqref{esti-2}, as well as the
trivial inequality $(\ene sP +c_4)^{1/\qphi} \leq \ene sP +c_4 +1$.
Exchanging the roles of $s$ and $t$, from \eqref{e:absolute}  and \eqref{energy-used-later} we infer
$ |\ene tP - \ene sP |\leq C |t-s| \mathcal{G}(P) $  for every
$s,\,t \in [0,T]$, and \eqref{eq:diffclass_a} follows.  
\end{proof}

\subsection{Properties of the dissipation potentials}
\label{ss:4.2}
We now show that the dissipation potentials defined by
\eqref{def-Psip} comply with conditions
\eqref{eq:psi-sum-1}--\eqref{eq:psi-uni}.

\begin{lemma}[Properties  of $(\Psi_P)_{P\in \domain}$]
\label{l:prelim3} Assume 
\eqref{kappa2}--\eqref{kappa1},  \eqref{loading},
 \eqref{w1}--\eqref{w2}, \eqref{R1}--\eqref{R3}, and
 \eqref{index-1}--\eqref{index-2}. Then,
\begin{equation}
\label{repre-psi-star} \Psi_P^*(\Xi) = \int_{\Omega}
\rmetric^*(\Xi(x) \transp{P(x)})  \dd x  \quad \text{for every $P
\in \domain$ and $\Xi \in X^*= L^{p'} (\Omega;\R^{d\times d})$.} 
\end{equation} 
Furthermore, the dissipation potentials $(\Psi_P)_{P \in \domain}$
comply with \eqref{eq:psi-sum-1}--\eqref{eq:psi-uni}, and
\begin{equation}
\label{enhanced-coercivity-psi}
\begin{aligned}
  & \forall\, S>0 \ \exists\, c_6,\, c_7>0 \ \forall\, P \in \domain
  \text{ with } \mathcal{G}(P) \leq S \
  \forall\, V\in  X \   \forall\, \Xi\in X^*  \, :\\
  &\qquad \Psi_P (V) \geq c_6 \| V \|_{L^p}^p -c_7 \ \text{ and } \
  \Psi_P^* (\Xi) \geq c_6 \| \Xi \|_{L^{p'}}^{p'} -c_7.
\end{aligned}
\end{equation}
\end{lemma}
\begin{proof} Clearly, for every $P \in \domain$ the functional
  $\Psi_P$ is well-defined on $L^p (\Omega;\R^{d\times d})$; from
  \eqref{R2} it follows that $\Psi_P$ is convex and lower
  semicontinuous, and that $\Psi_P (0) =0$. Furthermore, using
  \eqref{R3}, the first estimate in \eqref{useful-later}, and the fact
  that $\domain \subset \mathrm{C}^0 (\overline{\Omega};\R^{d \times
    d})$, we find for every $ (P,V) \in \domain \times L^p
  (\Omega;\R^{d\times d})$ the estimate
\begin{equation}
\label{esti-diss-1}
\begin{aligned}
 \Psi_P (V) \geq
C_R^1 \int_{\Omega} |V(x)\inve{P(x)}|^p  \dd x  -C_R^2 |\Omega| \geq
\frac{C_R^1}{\|P\|_{L^\infty}^p}\int_{\Omega}
|V(x)|^p  \dd x  - C_R^2 |\Omega|,
 \end{aligned}
\end{equation}
which in particular yields that, for $P \in \domain$ fixed with
$\mathcal{G}(P) \leq S$, the functional $\Psi_P$ has uniform
superlinear growth on sublevels of $\mathcal{G}$.

The representation formula \eqref{repre-psi-star} follows from
\[
\begin{aligned}
\Psi_P^*(\Xi) = \int_{\Omega}\rdiss^* (P(x),\Xi(x))  \dd x 
\end{aligned}
\]
(where $\rdiss^*$ is the conjugate of $\rdiss$ with respect to the third variable), and the
 calculation
\begin{equation}
\label{calcul-rdistar}
\begin{aligned}
\rdiss^* (P,\Xi) = \sup_{V\in \R^{d\times d}} (\Xi : V -
\rmetric(V\inve{P}))   & = \sup_{W\in \R^{d\times d}} (\Xi
: WP- \rmetric(W))\\ &  = \sup_{W\in \R^{d\times d}} (\Xi
\transp{P} : W - \rmetric(W)) = \rmetric^*(\Xi
\transp{P}),
\end{aligned}
\end{equation}
where the second identity follows from the change of variable $W = V
\inve{P}$ and the third one from \eqref{useful-inner}.
With \eqref{repre-psi-star} and the second of \eqref{R2} we find
that the functionals $(\Psi_P^*)_{P\in \domain} $ comply with the second of
\eqref{eq:psi-uni}.  Combining \eqref{repre-psi-star}--\eqref{calcul-rdistar} with \eqref{growth-Rstar},
we also conclude that 
\begin{equation}
\label{esti-diss-2}
\begin{aligned}
& 
 \forall\, (P,\Xi)
\in \domain \times L^{p'} (\Omega;\R^{d\times d}) \, : \\  &
\begin{aligned}
 \quad \Psi_P^*
(\Xi) = \int_{\Omega} \rmetric^*(\Xi(x)
(\transpi{P(x)})^{-1})    \dd x   &  \geq \widetilde{C}_R^1 \int_\Omega |\Xi(x) (\transpi{P(x)})^{-1}|^{p'} - \widetilde{C}_R^2 |\Omega|
  \\ & 
  \geq \frac{\widetilde{C}_R^1}{\|\inve{P}\|_{L^\infty}^{p'}}\int_{\Omega} |\Xi(x)|^{p'}  \dd x  - \widetilde{C}_R^2 |\Omega|. \end{aligned}
\end{aligned}
\end{equation}
Then, \eqref{basic-psi-2} follows from combining \eqref{esti-diss-1}
and \eqref{esti-diss-2} with estimate \eqref{conse-2} and  the
compact embeddings $W^{1,\qg}(\Omega;\R^{d\times d}) \subset
\mathrm{C}^0(\overline{\Omega};\R^{d\times d})$ and
$W^{1,\tilde{q}}(\Omega;\R^{d\times d}) \subset
\mathrm{C}^0(\overline{\Omega};\R^{d\times d})$.

It remains to prove \eqref{basic-psi-3}.  Let $(P_n)\subset \domain$,
$(V_n) \subset L^p (\Omega;\R^{d\times d})$ be sequences as in
\eqref{basic-psi-3}. In view of \eqref{conse-2}, we find that
$\inve{P}_n \to \inve{P}$ in $\mathrm{C}^0 (\overline{\Omega};
\R^{d\times d})$, therefore $V_n \inve{P}_n \weakto V \inve{P}$ in
$L^{p} (\Omega;\R^{d\times d})$, and the Ioffe theorem
\cite{Ioff77LSIF} ensures that
\[
\liminf_{n \to \infty}  \Psi_{P_n}  (V_n)  = \liminf_{n \to \infty}
\int_{\Omega} \rmetric(V_n(x) \inve{P_n(x)})  \dd x  \geq
\liminf_{n \to \infty} \int_{\Omega} \rmetric(V(x) \inve{P(x)})
 \dd x  = \Psi_P(V).
\]
Analogously, for  $(\Xi_n)_{n}$ with $\Xi_n \weakto \Xi$ in
$X^*$ we obtain $\liminf_{n \to \infty}   \Psi_{P_n}^* (\Xi_n) \geq
\Psi_P^*(\Xi)$.  
\end{proof}

\subsection{Properties of $\diffname$ and $\Ptname$}
\label{ss:4.3}
We now turn to the analysis of the \emph{marginal subdifferential}
$\diffname$ from \eqref{subdiff-notation}. Recall that its
definition involves the G\^ateaux derivatives $\rmD_P \calI
(t,\varphi,P)$ of $\calI$, evaluated at minimizers $\varphi \in \mins
tP$. Observe that, a priori, $\rmD_P \calI (t,\varphi,P)= -\Delta_\qg
P + \rmD K(P) + \Back{\nabla\varphi}{P}$ is an element in
$W^{1,\qg}(\Omega)^*$. In fact, the precise definition of $\diffname $
turns out to be
\begin{equation}
\label{more-precise-subdif}
\diff tP = \{ \Xi \in X^* = L^{p'}(\Omega;\R^{d\times d})\, : \
\exists \,\varphi \in \mins tP \text{ s.t. } \Xi = \rmD_P \cal
I(t,\varphi,P) \}\,. 
\end{equation}
With Lemma \ref{l:prelim4} below, we show (cf.\ \eqref{inclu-diff})
that for every $(t,P) \in [0,T]\times \domain$ the marginal
subdifferential $\diff tP$ contains the \emph{Fr\'echet
  subdifferential} $\partial \en(t,P)$ of $\en$ with respect to the
$L^{p}(\Omega;\R^{d\times d})$-topology.  We recall that the latter is
the multivalued operator $\partial \en :[0,T] \times \domain
\rightrightarrows L^{p'}(\Omega;\R^{d\times d})$ defined at $(t,P) \in
[0,T]\times \domain$ by
\begin{align}
  \label{def-Frechet}
  &\Xi \in \partial \en(t,P) \ \text{ if and only if } \\
  &\nonumber\quad \ene t{Q} - \ene t{P} \geq
   \int_\Omega \Xi : (Q{-}P) \dd x + o(\|
   Q{-}P\|_{L^{p}\OOdd}) \text{ as } Q \to P \text{ in }
   L^{p}(\Omega;\R^{d\times d}).
\end{align}
Property \eqref{inclu-diff} ahead will be used in Corollary
\ref{var-sum-rule} to verify the \emph{variational sum rule} required
within the abstract existence theory for gradient systems, cf.\
\eqref{eq:42-bis}.

In the proof of Lemma \ref{l:prelim4} and of subsequent results, a key
role is played by estimates \eqref{key-Alex-1}--\eqref{key-Alex-2},
derived next. Note that in the following estimates the variable
$P$ is restricted to sets such that $|P|+|\inve{P}|\leq C_P$ , whereas
$F=\nabla \varphi$ varies in all of $\GLD$. 

\begin{lemma}
\label{l:estimates-pointwise}
Assume \eqref{kappa2}--\eqref{kappa1}, \eqref{loading},
\eqref{w1}--\eqref{w3}, and \eqref{index-1}--\eqref{index-2}.  Then,
there holds
\begin{align}
&
\label{key-Alex-1}
\begin{aligned}
&
 \forall\, C_P>0 \ \exists\, C_B>0 \ \exists\, \bar{r}>0  \\ &     \forall\, 
  P \in \GLD \text{ with } |P| +|\inve{P}| \leq C_P \ \forall\, F \in
  \GLD \ \forall\, N \in \mathcal{N}_{\bar{r}} \, : \qquad
\\
& 
 \qquad \left| \Back{F}{NP} {-} \Back FP \right| \leq C_B 
  |N{-}\mathbf{1}| \big(W(F\inve{P}){+}1\big), 
\end{aligned}
\end{align}
and 
\begin{align}
&
\label{key-Alex-2}
\begin{aligned}
&
 \forall\, C_P>0 \ \exists\, C_W>0 \  \exists \, \tilde{r}>0
 \\
 & 
  \forall\, 
P_1,\, P_2 \in \GLD \text{ s.t. } |P_1| +|\inve{P}_1| \leq C_P,   \
|P_1{-}P_2| \leq \tilde{r}  \text{ and }  \forall\, F \in \GLD \, :  
\\
& 
\begin{aligned}
&
 \left| W(F\inve{P}_1) {-} W(F\inve{P}_2) {-} \Back{F}{P_1} {:}
   (P_1{-}P_2) \right|  \leq C_W  (W(F\inve{P}_1){+}1)
 |P_1{-}P_2|^2\,, 
 \end{aligned} 
\end{aligned}
\end{align}
\end{lemma}
\begin{proof} 
\textbf{Ad \eqref{key-Alex-1}:} let us set $P_N: = NP$, and observe
that $\inve{P}_N  = \inve{P} \inve{N}$. Furthermore,  we have 
\begin{equation}
\label{prelim-est-P-1}
| \inve{P}_N {-} \inve{P} | \leq C_P |\inve{N} {-} \mathbf{1}| \leq 2C_P |N{-} \mathbf{1}|,
\end{equation}
Let us define  $\bar{r}: = \min\{  |\inve{P}|/4C_P, \delta\}$, with $\delta>0$ from \eqref{w3}. 
 For $N \in \mathcal{N}_r$ with  $r \leq \bar{r}$, we have that $2C_P |N{-} \mathbf{1}| \leq  |\inve{P}|/2$. 
Therefore we find 
\begin{equation}
\label{prelim-est-P-2}
|\inve{P}_N| \leq  | \inve{P}_N {-} \inve{P} | + |\inve{P} |  \leq 2 |\inve{P}| \leq 2C_P.
\end{equation}
Hence
\[
\begin{aligned}
\left| \Back{F}{NP} {-} \Back FP \right| & \leq | M(F\inve{P}_n)\transpi{P}_N {-} M(F \inve{P}) \transpi{P}| 
\\
& \leq   | (M(F\inve{P}_n)\transpi{P}_N {-} M(F \inve{P})) \transpi{P}_N| +  |M(F \inve{P}) (\transpi{P_N}{-} \transpi{P})|
\\
& \stackrel{(1)}{\leq}  C_5 |\inve{N}{-} \mathbf{1}| (W(F\inve{P}){+}1)  2C_P    + C_4 (W(F\inve{P}){+}1) C_P |\inve{N} {-} \mathbf{1}| 
\\
& \stackrel{(2)}{\leq}   6C_P (C_4+C5) (W(F\inve{P}){+}1) |N{-} \mathbf{1}|,
\end{aligned}
\]
where {\footnotesize (1)} follows from the multiplicative stress
control conditions \eqref{w3}(ii) and \eqref{w3}(iii), combined with
estimates \eqref{prelim-est-P-1} and \eqref{prelim-est-P-2}, while
{\footnotesize (2)} is due to the last inequality in
\eqref{prelim-est-P-1}.  We thus conclude \eqref{key-Alex-1}.
\\
\textbf{Ad \eqref{key-Alex-2}:} By the chain rule we have that
\[
W(F\inve{P}_1) {-} W(F\inve{P}_2) =\! \int_0^1\!\!  \Back{F}{N_\sigma P_1}
 {:}   (P_1{-}P_2) \dd \sigma \text{ with }N_\sigma: = \mathbf{1} +
(1{-}\sigma)(P_2{-}P_1) \inve{P}_1. 
\]
Clearly, $N_\sigma P_1 =  (1{-}\sigma)P_2 + \sigma P_1 $ and 
\begin{equation}
\label{even-this-used}
|N_\sigma{-} \mathbf{1}| \leq (1{-}\sigma) C_P |P_1{-}P_2|\,.
\end{equation}
Let us set $\tilde{r}:= \bar{r}/2C_P$ with $\bar{r}>0$ from \eqref{key-Alex-1}, 
then for $|P_1{-}P_2 | \leq \tilde{r}$ we have that 
\begin{equation}
\label{for-later-use-N}
|N_\sigma{-} \mathbf{1}| \leq \bar{r}/2.
\end{equation}
 Therefore, 
\[
\begin{aligned}
\left| W(F\inve{P}_1) {-} W(F\inve{P}_2) {-} \Back{F}{P_1} {:}
  (P_1{-}P_2) \right| 
 & \leq  \int_0^1\! \left| \Back{F}{N_\sigma P_1} {-} \Back{F}{P_1}
 \right|\, |P_1{-}P_2 |  \dd \sigma  
\\
 \stackrel{(1)}{\leq} C_B  (W(F\inve{P}_1){+}1)|P_1{-}P_2| \int_0^1
|N_\sigma{-}\mathbf{1}|  \dd \sigma  
&
\stackrel{(2)}{\leq}  C_P C_B (W(F\inve{P}_1){+}1)|P_1{-}P_2|^2, 
\end{aligned}
\]
where {\footnotesize (1)} is due to the previously proved estimate
\eqref{key-Alex-1}, (which applies since $N_\sigma \in
\mathcal{N}_{\bar{r}/2}$ by \eqref{for-later-use-N}), whereas
{\footnotesize (2)} ensues from \eqref{even-this-used}.  We have thus
established \eqref{key-Alex-2}. \end{proof} 

We are now in the position to prove Lemma \ref{l:prelim4}, where
we also derive a bound for the $L^1(\Omega;\R^{d\times d})$-norm
of the backstress contribution $B$ to $\diff tP$ from the
multiplicative stress control condition \eqref{w3}(ii). This will turn
out to be crucial in the proof of the closedness property
\eqref{eq:468}, cf.\ Lemma \ref{l:prelim5} ahead.

\begin{lemma}[Marginal subdifferential $\diffname$ and Fr\'echet
  subdifferential $\frsub\en$]
\label{l:prelim4} Assume \eqref{loading}, \eqref{kappa2}--\eqref{kappa1}, 
 \eqref{w1}--\eqref{w3}, and \eqref{index-1}--\eqref{index-2}.
Then, the following holds:
\\[0.3em]
1. For every $(t,P) \in [0,T]\times \domain $ with $\frsub \ene tP
  \neq \emptyset$ we have:
\begin{equation}
  \label{stronger-form}
  \forall\,\Xi \in \frsub \ene tP \ \forall \, 
   \phimin \in  \mins t P : \ \Xi - \rmD K(P) -
   \Back{\nabla\varphi}{P} = -\Delta_{\qg} P  
    \text{ in } L^1(\Omega;\R^{d\times d}),
\end{equation}
which in particular implies
\begin{equation}
\label{inclu-diff} \frsub \ene tP \subset \diff tP \ \subset X^* 
   \quad \text{ for every } (t,P) \in [0,T]\times \domain\,.
\end{equation} 
\smallskip

\noindent 2. For every $(t,P) \in [0,T]\times \domain  $ and 
$\phimin \in \mins tP$ we have  $ \Back{\nabla\varphi}{P}  \in
L^1(\Omega;\R^{d\times d})$, and
\begin{equation}
\label{even-better2}
\exists\, c_8{>}0 \   \forall\,(t,P) {\in} [0,T]{\times} \domain 
\ \forall\, \phimin {\in} \mins tP{:} \  
 \|  \Back{\nabla\varphi}{P} \|_{L^1} \leq c_8
\|\inve{P} \|_{L^\infty} ( \mathcal{G}(P){+}1 ).
\end{equation}
\end{lemma}
\begin{proof}
  \textbf{Ad \eqref{stronger-form}:} For the proof of
  \eqref{stronger-form}, we shall use the notation
\[
  \overline{\en}_1 (P) : = \frac1{\qg} \int_{\Omega} |\nabla P(x)|^{\qg}
  \dd x.  
\]
Since the Fr\'echet subdifferential of $\overline{\en}_1$ with respect
to the $W^{1,\qg}(\Omega;\R^{d\times d})$-topology fulfills
$ \partial_{W^{1,\qg}} \overline{\en}_1 (P) =\{ -\Delta_{\qg} P\}
\subset \big(W^{1,\qg}(\Omega;\R^{d\times d}) \big)^*$, we will
establish \eqref{stronger-form} by showing that for all $\Xi \in
\frsub \ene tP \subset X^*=L^{p'}(\Omega;\R^{d\times d})$ and $\phimin
\in \mins t P$ we have $\Xi - \rmD K(P) - \Back{\nabla\varphi}{P}
\in \partial_{W^{1,\qg}} \overline{\en}_1 (P)$. Then,
\eqref{stronger-form} holds, because all terms on the left-hand side
lie in $L^1(\Omega;\R^{d\times d})$.

Hence, by the definition of Fr\'echet subdifferentials,
we have to prove that
\begin{equation}
  \label{claim-bis}
 \begin{aligned}
   \Lambda (\bar{\en}_1) &  : = \overline{\en}_1(P_n) - \overline{\en}_1({P}) 
     -\int_{\Omega} \left(\Xi - \rmD K(P) -  \Back{\nabla\varphi}{P} 
  \right):  ( P_n {-} P )  \dd x   
 \\  &   \geq 
  o(\|P_n - P\|_{W^{1,\qg} }) \ \text{ as } P_n
  \to P \text{ in } W^{1,\qg} (\Omega;\R^{d\times d})\,. 
 \end{aligned}
\end{equation}
For this we  will in fact exploit that, since $\qg >d$  by  \eqref{index-1}, we also have $ P_n \to P
$ in $\rmC^0 (\overline{\Omega};\R^{d\times d})$.
 Now, we observe that
\begin{align*}
 \Lambda (\bar{\en}_1)  &  = \ene{t}{P_n} -
 \ene{t}{P} -  \int_{\Omega}\Xi : (P_n-P)   \dd x 
\\ & \quad + \int_{\Omega} \left( K(P) - K(P_n) + \rmD K(P) :
  (P_n {-}P) \right)  \dd x  
\\ & \quad 
 + \en_2 (t,P) - \en_2 (t,{P_n}) + \int_{\Omega}
 \Back{\nabla\varphi}{P} : (P_n {-}P)  \dd x 
  =:  \Lambda (\cE) + \Lambda (K) + \Lambda (\cE_2)\,.
\end{align*}
In what follows, we will estimate from below the three
terms $\Lambda(\cdot)$ individually. First of all, by the definition
\eqref{def-Frechet} of the Fr\'echet subdifferential, there holds
\[
  \Lambda (\cE)  \geq o(\|P_n {-}P\|_{L^{p} }) 
\quad \text{as } P_n \to P \text{ in } L^{p} (\Omega;\R^{d\times d}).
\]
As for $\Lambda (K)$, we observe that images of the sequence
$(P_n)_n$ and $P$ belong to a compact subset of $\GLD$.  Hence, we
have
\begin{equation}
\label{chain-K}
\begin{aligned}
 \Lambda (K) & = 
  \int_{\Omega}   \left(K(P) - K(P_n) + \rmD K(P) : (P_n
 {-}P)\right) \dd x   
 \\
 &  = \int_{\Omega} \int_0^1 \left( \rmD K(P)  -
 \rmD K((1{-}\sigma)P_n + \sigma P) \right) : (P_n {-} P)  
  \dd\sigma  \dd x 
  \\ &  \geq
  - C \| P_n {-}P\|_{L^{2} }^2 = o(\| P_n
  -P\|_{W^{1,\qg}}) 
 \end{aligned}
\end{equation}
as $P_n\to P$ in $W^{1,\qg} (\Omega;\R^{d\times d})$, where the second
estimate follows from the fact that $K$ is of class $\rmC^2$ on
$\GLD$.  Finally, for $\Lambda (\en_2)$ we choose $\varphi \in
\mins tP$ and use $\enei2 t{P_n} \leq \calI_2(t,\varphi,P_n)$ to
obtain
\begin{equation}
\label{e:old-3.21}
 \begin{aligned}
  \Lambda(\en_2) &= 
  \enei 2 t{P} - \enei2 t{P_n} + \int_{\Omega}
  \Back{\nabla\varphi}{P} : (P_n{-}P)
  \dd x   \\ & \geq \int_{\Omega} W\argut{\varphi}{P} -
  W\argut{\varphi}{P_n} +  \Back{\nabla\varphi}{P} : (P_n 
  {-} P) \dd x 
\\ &\stackrel{(1)}{\geq}  -  C_W  \int_\Omega (W(\nabla
  \varphi\inve{P}){+}1) |P_n{-}P|^2 \dd x  
\\
& \stackrel{(2)}{\geq} - C \left( \en(t,P){+}1\right) \| P_n {-} P\|_{L^\infty}^2 
   \ \stackrel{(3)}{\geq} \  - C \left( \calG(P) {+}1\right)  \| P_n {-}
  P\|_{W^{1,\qg}}^2\,.  
 \end{aligned}
\end{equation}
Here, {\footnotesize (1)} follows from estimate \eqref{key-Alex-2},
which applies since $P \in \domain $ and we may suppose that $\| P_n
{-} P \|_{L^\infty} \leq \tilde{r}$, with $\tilde{r} $ from
\eqref{key-Alex-2}, as $ P_n \to P $ in $\rmC^0
(\overline{\Omega};\R^{d\times d})$.  Then, estimate {\footnotesize
  (2)} ensues from \eqref{esti-2}, and {\footnotesize (3)} from
\eqref{energy-used-later}.  All in all, we infer that $\Lambda (\en_2)
\geq o(\|P_n - P \|_{W^{1,\qg}}) $ as $P_n \to P$ in
$W^{1,\qg}(\Omega;\R^{d\times d}). $ This gives \eqref{claim-bis},
hence \eqref{stronger-form}.

\textbf{Ad \eqref{even-better2}:} In view of   \eqref{w3}(ii),  there holds
\[
\begin{aligned}
\int_{\Omega}   | \Back{\nabla\varphi(x)}{P(x)} |  \dd x   & \leq C_4
\int_{\Omega} \left(W(\nabla \varphi(x) \inve{P(x)}) +1 \right)
|\transpi{P(x)}| \dd x  \\ & \leq C_4 \|\inve{P}
\|_{L^\infty} \left(c_3^{-1} \ene tP +
|\Omega|\right),
\end{aligned}
\]
where the last inequality ensues from \eqref{esti-2}.  Hence,
\eqref{even-better2} is established.  This concludes the proof.
\end{proof}

\begin{remark}\slshape \label{rmk:after-prelim4}
 If  \eqref{w3}(iii)  holds in the more general form \eqref{Ball-weak},
 cf.\ Remark \ref{rmk:ext-Ball}, the proof of  \eqref{stronger-form}  still
 goes through.  Indeed, in this case we would have
\[
\begin{aligned}
\Lambda (\en_2) & \geq -  C_W \int_\Omega (W(\nabla \varphi\inve{P}){+}1)
|P_n{-}P|^{1+ \alpha_W} \dd x  
\\
& \geq  - C \left( \calG(P) {+}1\right)  \| P_n {-}
P\|_{W^{1,\qg}}^{1+\alpha_W} = o( \| P_n {-} P\|_{W^{1,\qg}})\,.  
\end{aligned}
\]
\end{remark}

\begin{lemma}[Properties of $\Ptname$]\label{l:prelim4-5} 
Assume \eqref{kappa2}--\eqref{kappa1},
  \eqref{loading}, \eqref{w1}--\eqref{w3}, and
  \eqref{index-1}--\eqref{index-2}.  Then,
\begin{equation}
\label{def-P-hat-better} 
\Pt tP{\Xi} :=  \sup_{\phimin \in \Rt tP{\Xi}}
 \!\!\pairing{}{W^{1,\qphi}\OOd}{-\dell(t)}{\phimin}   \ = \ 
\max_{\phimin \in \Rt tP{\Xi} \rule{0pt}{0.7em}} \!\!\pairing{}{W^{1,\qphi}\OOd}{-\dell(t)}{\phimin},
\end{equation}
and  the function $\Ptname: \mathrm{graph}(\diffname) \to \R$ complies
with \eqref{e:ass-p-a}. 
\end{lemma}
\begin{proof}
We first observe that   for every $ (t,P,\Xi) \in
\mathrm{graph}(\diffname)$ the set 
\begin{equation}
\label{weakly-compact} 
\Rt t{P}{\Xi}   \text{ is nonempty and  weakly sequentially 
compact in } W^{1,\qphi}(\Omega;\R^d).
\end{equation}
Indeed, every sequence $(\phiminn)_n \subset \Rt t{P}{\Xi}$ is bounded
in $W^{1,\qphi}(\Omega;\R^d)$ thanks to \eqref{esti-2}, hence up to a
subsequence it  weakly  converges to some $\varphi$. From the arguments in the
proof of Lemma \ref{l:prelim5} ahead (cf.\ Step 5), it will follow
that $\varphi \in \Rt t{P}{\Xi}$. Thus, it is immediate to see that
the $\sup$ in formula~\eqref{def-P-hat-better} is indeed a $\max$.

Furthermore, the function $\Ptname: \mathrm{graph}(\diffname) \to \R$
defined by~\eqref{def-P-hat} is a Borel function because it is the
$\sup$ of Borel functions.

For the lower estimate of $\Ptname$ in \eqref{e:ass-p-a}, we
consider first $(t,P) \in [0,T]\times \domain$, $h\in (0,T{-}t]$
and $\phimin(t)\in \mins tP$ to obtain 
\[
\frac{\en(t+h,P) -\en (t,P)}{h}= \frac{\en_2 (t+h,P) -\en_2
(t,P)}{h}\leq \frac{1}h \pairing{}{W^{1,\qphi}\OOd}{-\ell(t{+}h)
 +\ell(t)}{\phimin(t)},
\]
whence $\limsup_{h \down 0}\frac{\en(t+h,P) -\en(t,P)}{h}\leq \Pt
tP\Xi$.  In order to prove the upper estimate $\Pt
tP\Xi \leq \lim_{h \down 0}\tfrac{\en(t,P) -\en(t-h,P)}{h}$, it is
sufficient to observe that 
\[
\frac{\en(t{+}h,P) -\en(t,P)}{h} \geq  \frac{1}h
\pairing{}{W^{1,\qphi}\OOd}{\ell(t{+}h) -\ell(t)}{\phimin(t)},
\]
for every $\phimin(t)\in \mins tP$, and then take the limit as
$h\downarrow 0$. On the other hand, it follows
from~\eqref{loading} and~\eqref{esti-2} that
\[
\begin{aligned}
 |\Pt tP\Xi|   \leq  \|\dell(t)
\|_{(W^{1,\qphi})^*} \cdot \sup_{\phimin \in \mins tP} \|
\phimin \|_{W^{1,\qphi}}  & \leq  c_3  \|\dell(t)
\|_{(W^{1,\qphi})^*}  \left( 1{+} \|P\|_{L^\qp}^\qp\right)^{1/\qphi} 
\\
&  \leq C ( \calG(P) {+} 1)^{1/\qphi}     \leq \widetilde{C}
(\calG(P) {+} 1). 
\end{aligned}
\]
Therefore, \eqref{e:ass-p-a} is fulfilled.
\end{proof}

We will now show that  the  triple  $(\en, \diffname,\Ptname)$ complies  with  a more general form of the  closedness condition \eqref{eq:468},
where a sequence $(t_n)_n$ is also considered.   This is the most difficult step, which involves the ideas developed in \cite{DaFrTo05QCGN} and, more abstractly, in \cite{FraMie06ERCR,KnZaMi10CGPM}. 
\begin{lemma}[Closedness for $(\en,\diffname,\Ptname)$]
\label{l:prelim5} Assume 
\eqref{kappa2}--\eqref{kappa1}, \eqref{loading},
 \eqref{w1}--\eqref{w3}, and \eqref{index-1}--\eqref{index-2}.
Then,  \eqref{eq:468} holds. 
\end{lemma}
\begin{proof} 
Let 
 $( t_n )_n \subset [0,T]$, $(P_n)_n\subset
L^p(\Omega;\R^{d\times d})$, and $ (\Xi_n)_n \subset
L^{p'}(\Omega;\R^{d\times d}) $ with $\Xi_n \in \diff{t_n}{P_n}$ for
all $n \in \N$ fulfill as $ n \to \infty$
\begin{equation}
\label{convs-4-closedness}
\begin{aligned}
&
  t_n \to t, \ \   P_n \weakto P \
\text{in $L^p(\Omega;\R^{d\times d})$,}\ \ \Xi_n\weakto \Xi \
\text{in $L^{p'}(\Omega;\R^{d\times d})$,}
\\
&
  \ene{t_n}{P_n} \to \EE \text{ and } \Pt {t_n}{P_n}{\Xi_n} \to \mathscr{P} \text{ in $\R$}.
  \end{aligned}
\end{equation}
We will split the proof that 
\begin{equation}
\label{contin-prop}
\begin{aligned}
\Xi \in \diff{t}P,  \ \ \EE=\ene tP, \ \ \mathscr{P} \leq \Pt tP{\Xi}
\end{aligned}
\end{equation} 
in  several steps.

\paragraph{\bf Step 1:  lower semicontinuity inequality for the energies.} 
From $ \sup_n \ene{t_n}{P_n} <\infty$ and from \eqref{conse-2} we
infer convergences \eqref{weak-w1q}--\eqref{weak-w1-tildeq} for
$(P_n)_n$ and $(P_n^{-1})_n$. Thus we have $\liminf_{n \to \infty}
\en_1 (P_n ) \geq \en_1 (P)$ as well as, with the very same argument
as in the proof of Lemma \ref{l:prelim2}, $\liminf_{n \to \infty}
\en_2 (t_n,P_n ) \geq \en_2 (t,P).$

\paragraph{\bf Step 2:  convergence of the minimizers.} 
Take any $(\varphi_n)_n$ with $\varphi_n \in \mins{t_n}{P_n}$ for
every $n\in \N$. As we have seen in the proof of Lemma
\ref{l:prelim2}, up to a subsequence $\varphi_n \weakto \varphi$ for
some $\varphi \in \calF$.  From $\phiminn \in \mathrm{Argmin}_{\varphi
  \in \calF} \calI_2(t_n,\varphi_n, P_n)$ we deduce
\begin{equation}
  \label{comparison-n} \int_\Omega W\argut{\phiminn}{P_n}   \dd x  -
  \pairing{}{W^{1,\qphi}}{\ell(t_n)}{\phiminn}
  \leq \int_\Omega  W\argut{\eta}{P_n}   \dd x  -
  \pairing{}{W^{1,\qphi}}{\ell(t_n)}{\eta} 
\end{equation}
for all $\eta \in \calF$. Moreover, using
\eqref{important-convergences} we see that $\int_\Omega
W\argut{\eta}{P_n} \dd x -
\pairing{}{W^{1,\qphi}\OOd}{\ell(t_n)}{\eta} \to \int_\Omega
W\argut{\eta}{P} \dd x - \pairing{}{W^{1,\qphi}\OOd}{\ell(t)}{\eta} $
for all $\eta \in \calF$, and combining~\eqref{liminf-nk} with
\eqref{comparison-n}, we conclude that $\varphi \in \mins tP$.
Moreover, choosing $\eta = \varphi$ in \eqref{comparison-n}, we find
\[
\begin{aligned}
  & \limsup_{n\to \infty} \left(\int_\Omega W\argut{\phiminn}{P_n} \dd
    x - \pairing{}{W^{1,\qphi}\OOd}{\ell(t_n)}{\phiminn}\right)
  \\
  & \leq \limsup_{n\to \infty} \left(\int_\Omega W\argut{\phimin}{P_n}
    \dd x - \pairing{}{W^{1,\qphi}\OOd}{\ell(t_n)}{\varphi}\right) \\
  & = \int_\Omega W\argut{\phimin}{P} \dd x -
  \pairing{}{W^{1,\qphi}\OOd}{\ell(t)}{\varphi} = \enei 2tP.
\end{aligned}
\]
Combining this with \eqref{liminf-nk}, we ultimately have $ \enei2
{t_n}{P_n} \to \enei2tP $ as $ n \to \infty$, whence
\begin{equation}
\label{truelim} 
 \lim_{n \to \infty} \int_\Omega W\argut{\phiminn}{P_n} \dd x  =
\int_\Omega W\argut{\phimin}{P} \dd x.
\end{equation}

\paragraph{\bf Step 3: further compactness arguments and convergence of the energies.}
In view of \eqref{subdiff-notation}, the sequence $\Xi_n \in
\diff{t_n}{P_n}$ in \eqref{contin-prop} is given for every $n \in \N$
by
\begin{equation}
 \label{xi-n-form}
 \Xi_n = -\Delta_{\qg} P_n + \rmD K(P_n) + \Back{\nabla\varphi_n}{P_n} 
 \quad \text{ for some $\phiminn \in  \mins{t_n}{P_n}$.} 
\end{equation}
Estimate \eqref{even-better2}, $\sup_n \ene{t_n}{P_n} <\infty$,
\eqref{conse-2}, and $\tilde{q}>d$ yields $ \sup_n \|
\Back{\nabla\varphi_n}{P_n} \|_{L^1} \leq C$. Moreover, unsing $K
\in \rmC^1(\GLD;\R)$ and \eqref{important-convergences} we deduce $
\sup_n \| \rmD K(P_n) \|_{L^\infty} \leq C$. By comparison in
\eqref{xi-n-form} we infer that
\begin{equation}
\label{esti-q-lapl} \sup_{n \in \N}\| {-}\Delta_{\qg} P_n  \|_{L^1} \leq C.
\end{equation}
Since $\qg>d$,  we have $L^1 (\Omega;\R^{d\times d}) \Subset
W^{1,\qg}(\Omega;\R^{d\times d})^*$, such that, up to a (not
relabeled) subsequence, $-\Delta_{\qg} P_n \to \chi$ in
$W^{1,\qg}(\Omega;\R^{d\times d})^*$ for  some $\chi$. Combining this with
\eqref{weak-w1q} and taking into account that the monotone map $P
\mapsto -\Delta_{\qg} P$ has a  weakly-strongly closed graph in
$W^{1,\qg}(\Omega;\R^{d\times d}) \times
W^{1,\qg}(\Omega;\R^{d\times d})^*$, we
 ultimately conclude that $\chi = -\Delta_{\qg} P$, and that along
 the same sequence as in \eqref{contin-prop} there holds
\begin{equation}
\label{strong-w1qg-star}  -\Delta_{\qg} P_n \to  -\Delta_{\qg} P  \
\text{ in  $W^{1,\qg}(\Omega;\R^{d\times d})^*.$}
\end{equation}
Thus,   the assumed weak convergence $P_n \weakto P$ in $W^{1,\qg}(\Omega;\R^{d\times d})$, cf.\  \eqref{weak-w1q}, improves to 
\begin{equation}
\label{strong-w1qg} P_n \to P \ \text{ in
$W^{1,\qg}(\Omega;\R^{d\times d}).$}
\end{equation}
Now, taking into account that $K$ is of class  $\rmC^0$   on $\GLD$
 we conclude that
\begin{equation}
\label{convex-k-enne} \int_{\Omega} K(P_n)  \dd x  \to  \int_{\Omega}
K(P)  \dd x   \ \ \text{ as $n \to \infty$.}
\end{equation}
Combining this with \eqref{strong-w1qg} and recalling
\eqref{ene1}--\eqref{form-of-h}, we have $ \calE_1 (P_n)
\to \calE_1 (P)$ so that,  by \eqref{truelim} we have 
\begin{equation}
\label{conve-ene1} \ene {t_n}{P_n} \to \ene tP \ \text{ as $n \to
\infty$,}
\end{equation}
whence the second of \eqref{contin-prop}. 
\paragraph{\bf Step 4: convergence of the stresses.} With
\eqref{important-convergences} and the $\rmC^0$-regularity of $\rmD K$
we infer
\begin{equation}
\label{weak-dkappa-1}
\begin{aligned}
  & \rmD K(P_n) \to \rmD K(P) \ \text{ in $L^\infty
(\Omega;\R^{d\times d})$}.
\end{aligned}
\end{equation}
We now show that
\begin{equation}
\label{weak-dw}
\begin{aligned}
  &  \Back{\nabla\varphi_n}{P_n} \weakto  \Back{\nabla\varphi}{P}
  \ \text{ in $L^1 (\Omega;\R^{d\times d})$, viz. } \\ & \int_{\Omega}
  \Back{\nabla\varphi_n}{P_n} : Q  \dd x  \to
\int_{\Omega} \Back{\nabla\varphi}{P} : Q  \dd x 
\end{aligned}
\end{equation}
for all $Q \in L^\infty (\Omega;\R^{d\times d})$, 
where $\varphi_n \in \mins {t_n}{P_n}$ and $\varphi \in \mins tP$ are
from the previous steps.  To this aim, we mimic the 
proof of \cite[Prop.~3.3]{FraMie06ERCR}. We fix $Q \in L^\infty
(\Omega;\R^{d\times d})$ and $h>0$.  On the one hand,
 we have 
\begin{equation}
\label{diff-quot-argum}
\begin{aligned}
    &  \Big| \frac1h \int_{\Omega}\big(
W\argut{\varphi_n}{(P_n {+} h Q)} - W\argut{\varphi_n}{P_n} - h
 \Back{\nabla\varphi_n}{P_n} {:} Q \big)  \dd x \Big|
\\
 &  \stackrel{(1)}{\leq}
\frac1h C_W \int_\Omega \left(W\argut{\varphi_n}{P_n}  {+}1\right) | hQ|^2 \dd x 
\ \stackrel{(2)}{\leq} \
 h C \left( \calG(P_n){+}1 \right) \|Q\|_{L^\infty}^2 =: \omega(h),
\end{aligned}
\end{equation}
where {\footnotesize (1)} follows from applying estimate
\eqref{key-Alex-2} with the choices $P_1: = P_n$ and $P_2:= P_n {+} h
Q$. Note that $\|P_n\|_{L^\infty} + \| \inve{P}_n\|_{L^\infty} \leq C$
thanks to the energy bound $\sup_n \en(t_n,P_n) <\infty$, combined
with the coercivity property \eqref{conse-2} and the continuous
embeddings \\ $W^{1,\qga}(\Omega;\R^{d\times d}) \subset
L^\infty(\Omega;\R^{d\times d})$ and  $ W^{1,\tilde q}(\Omega;\R^{d\times
  d}) \subset L^\infty(\Omega;\R^{d\times d})$. We may then take $h$
sufficiently small, in such a way that $\|P_1{-}P_2\|_{L^\infty} =
\|P_n {-}( P_n {+} h Q)\|_{L^\infty} = h \| Q\|_{L^\infty}\leq
\tilde{r}$. Then, {\footnotesize (2) } ensues from very same
calculations as in \eqref{e:old-3.21}.  Hence, $\omega (h) \to 0$
as $h \to 0$.  On the other hand, the polyconvexity of $W$ combined
with the very same arguments as in the proof of Lemma \ref{l:prelim1},
and the second of \eqref{truelim} yield
\begin{align*}
&  \liminf_{n \to \infty}
\int_{\Omega}\Big( W\argut{\varphi_n}{(P_n {+} h Q)} -
              W\argut{\varphi_n}{P_n} \Big) \dd x \\
&\geq \frac1h \int_{\Omega}\Big(W\argut{\varphi}{(P {+} h Q)} 
               - W\argut{\varphi}{P} \Big) \dd x.
\end{align*}
Estimate \eqref{diff-quot-argum} and the above inequality yield, for
all $h>0$, the estimate
\begin{equation}
\label{modulus}
\begin{aligned}
 & \liminf_{n \to \infty}
\int_{\Omega}  \Back{\nabla\varphi_n}{P_n} : Q  \dd x 
\\  & \quad   \geq \limsup_{n \to \infty} \frac1h
\int_{\Omega}\Big( W\argut{\varphi_n}{(P_n {+} h Q)} -  W\argut{\varphi_n}{P_n} 
 \Big) \dd x - \omega (h)
\\ & \quad  \geq \frac1h \int_{\Omega}\ Big(W\argut{\varphi}{(P {+} h
  Q)} - W\argut{\varphi}{P} \Big) \dd x   - \omega (h)
\\ & \quad  \geq \int_{\Omega}  \Back{\nabla\varphi}{P} : Q  \dd x    - 2\omega (h).
\end{aligned}
\end{equation}
where the last inequality follows from \eqref{diff-quot-argum} written
for $P$ instead of $P_n$.  Exchanging $Q$ with $-Q$, we analogously
infer that
\begin{equation}
\label{modulus-2} \limsup_{n \to \infty}   \int_{\Omega}
 \Back{\nabla\varphi_n}{P_n} : Q  \dd x  
 \leq   \int_{\Omega} \Back{\nabla\varphi}{P} : Q  \dd x   + 2
\omega(h).
\end{equation}
We conclude \eqref{weak-dw} taking the limit of \eqref{modulus} and
\eqref{modulus-2} as $h \to 0$.  

All in all, \eqref{strong-w1qg-star}, \eqref{weak-dkappa-1}, and
\eqref{weak-dw} imply that the sequence $\Xi_n$ in \eqref{contin-prop}
weakly converges in $L^{p'} (\Omega;\R^{d\times d})$ to $\Xi =
-\Delta_{\qg}P + \rmD K(P) + \Back{\nabla\varphi}{P}$, which
belongs to $\diff{t}P$. This proves the first of \eqref{contin-prop}.

\paragraph{\bf Step 5: Upper semicontinuity of the powers.}
The sequence $(\Pt {t_n}{P_n}{\Xi_n})_n$ from
\eqref{convs-4-closedness} is given for every $n\in \N$ by $\Pt
{t_n}{P_n}{\Xi_n} = -\pairing{}{W^{1,\qphi}}{
  \dell(t_n)}{\tilde{\varphi}_n}$ for some $\tilde{\varphi}_n \in
\Rt{t_n}{P_n}{\Xi_n}$, which thus fulfills $\Xi_n = - \Delta_{\qg}P_n
+ \rmD K(P_n) + \Back{\nabla\tilde{\varphi}_n}{P_n}$.  The
arguments from Steps 2 and 4 yield that, up to a further (not
relabeled) subsequence, $\tilde{\varphi}_n$ converges weakly in
$W^{1,\qphi}(\Omega;\R^d)$ to some $\tilde{\varphi} \in \mins{t}P $
and, in addition, $\Back{\nabla \tilde{\varphi}_n}{ P_n} \weakto
\Back{\nabla \tilde{\varphi}}P $ in $L^1(\Omega;\R^{d\times
  d})$. Hence $\Xi = \rmD K(P) - \Delta_{\qg}P +
\Back{\nabla\tilde{\varphi}}{P}$, i.e.\ $ \tilde\varphi \in
\Rt tP{\Xi}$.  Since $\ell \in
\rmC^1([0,T]);W^{1,\qphi}(\Omega;\R^d)^*)$, we have
\begin{equation}
 \label{proof-power}
 \lim_{n\to \infty}  -\pairing{}{W^{1,\qphi}}{ \dell(t_n)}{\tilde{\varphi}_n}  =  - \pairing{}{W^{1,\qphi}}{ \dell(t)}{\tilde{\varphi}} \leq \Pt {t}{P}{\Xi}
\end{equation}
  where the last inequality is  due to the definition of $ \Pt t{P}{\Xi}$  in \eqref{def-P-hat}.  From
\eqref{proof-power} we conclude the third relation in 
\eqref{contin-prop} and Lemma \ref{l:prelim5} is proved. 
\end{proof}

Finally, it remains to verify condition \eqref{eq:42-bis}, 
i.e.\ the \emph{variational sum rule}. For this, we apply Proposition  4.2 from \cite{MiRoSa13NADN}, which indeed holds for a general subdifferential $\diffname$ and for a family 
$(\Psi_P)_{P\in \domain}$ of dissipation potentials complying with \eqref{eq:psi-sum-1} and \eqref{basic-psi-2}. 
It
states that, if for every $(t,P) \in [0,T]\times \domain$ the set $\diff tP$ contains the Fr\'echet subdifferential $\partial\en(t,P)$, and 
if the pair $(\en,\diffname)$ complies with the closedness condition, then  \eqref{eq:42-bis} holds. 
Hence, the following result is a consequence of Lemmas \ref{l:prelim3}, \ref{l:prelim4} (guaranteeing the crucial
subdifferential inclusion \eqref{inclu-diff}), and  of Lemma  \ref{l:prelim5}. 

\begin{corollary}[Variational sum rule]\label{var-sum-rule} 
  Assume \eqref{kappa2}--\eqref{kappa1}, \eqref{loading},
  \eqref{w1}--\eqref{w3}, and \eqref{index-1}--\eqref{index-2}.  Then,
  the dissipation functionals $(\Psi_z)_{z \in \domain }$ and the
  reduced energy functional $\cE$ comply with the variational sum
  rule~\eqref{eq:42-bis}.
\end{corollary}

\subsection{Conclusion of the proof of Theorem  \ref{th:EDI.Viscopl}}
\label{suu:ProofThMarg}

In view of Lemmas \ref{l:prelim1}--\ref{l:prelim5} and Corollary
\ref{var-sum-rule}, our Theorem \ref{thm:abstract-2} applies to the
generalized gradient system $(X,\en, \Psi,\diffname,\Ptname)$,
yielding the existence of EDI solutions to the viscoplastic Cauchy
problem in the form of the abstract doubly nonlinear equation
\eqref{abstract-dne}.  This finishes the proof of Theorem
\ref{th:EDI.Viscopl}.  \hfill $\Box$%

\begin{remark}[Missing chain rule]\slshape\label{rmk:failure}
  It remains an open problem to prove  or disprove  that the functional $\en$
   with differentials $(\diffname, \Ptname)$ 
  fulfills the CRI \eqref{eq:48strong}. Indeed, let  a curve   $P\in
  \AC([0,T];L^p(\Omega;\R^{d\times d}))$ and $\Xi\in L^1(0,T
  ;L^{p'}(\Omega;\R^{d\times d}))$ fulfill \eqref{conditions-1}. Then,
  $\dot P \in L^p(0,T;L^p(\Omega;\R^{d\times d}))$ and $\Xi \in
  L^{p'}(0,T; ;L^{p'}(\Omega;\R^{d\times d}))$, hence their duality
  pairing is well defined for almost all $t\in (0,T)$. But, in order
  to estimate it as required in \eqref{eq:48strong}, one would have to
  use the full information provided by the structure of the marginal
  subdifferential, i.e.\ by the formula $\Xi = -\Delta_{\qg}P+\rmD
  K(P) + \Back{\nabla\varphi}{P}$ with $\varphi \in \mins tP$.  Since
  there are no compensations between the terms contributing to $\Xi$,
  this would ultimately boil down to estimating individually the
  duality pairings $\pairing{}{}{-\Delta_{\qg}P}{\dot P}$,
  $\pairing{}{}{\rmD K(P)}{\dot P}$, and
  $\pairing{}{}{\Back{\nabla\varphi}{P} }{\dot P}$,  which seems to be out of
  reach. In fact, the stress control condition  \eqref{w3}(ii)  only
  ensures that  $ \Back{\nabla\varphi}{P} \in L^\infty (0,T;
  L^1(\Omega;\R^{d\times d}))$, and accordingly by a comparison
  argument we may only conclude that $-\Delta_{\qg}P \in L^\infty (0,T;
  L^1(\Omega;\R^{d\times d}))$.  Hence, neither the first, nor the
  third duality pairings are well defined.
 
  As we will see in Section \ref{s:6}, this problem can be circumvented
  if we add to the stored energy $\calI$ a term which improves the
  spatial estimates for the Mandel stress tensor.
  \end{remark}

\section{Existence of energy solutions for a  regularized system}
\label{s:6}
We now investigate an alternative model for finite-strain
viscoplasticity, where the stored energy $\calI$ from
\eqref{stored-energy} is augmented by a regularizing contribution $\calI_3$,
 multiplied  by a (small, but positive) parameter $\rpam>0$. We will show
that the reduced energy $\calE_\rpam$ accordingly obtained by
minimizing out the deformations complies with the CRI
\eqref{eq:48strong}, in addition to conditions
\eqref{basic-1}--\eqref{eq:468}. From Corollary \ref{cor:abstract-1}
we will thus deduce the existence of EDB solutions to the generalized
gradient system \eqref{abstract-dne} driven by the regularized energy
$\calE_\rpam$.  This will yield solutions to a version of the PDE
system \eqref{pde} with a different minimum problem for $\varphi$, but
the \emph{same} flow rule for $P$. For this, it will be crucial that
$\calI_3$  in \eqref{calF-reg} below does not depend on the plastic variable $P$.
\subsection{The regularized model}
\label{su:6.1}
In the next lines, we specify the form of the regularizing contribution  to the stored energy and accordingly introduce the energy functional driving the 
regularized gradient system.
\subsubsection*{The regularized stored energy} 
For $\rpam>0$ fixed we define the functional $\calI_\rpam : [0,T]\times W^{1,\qphi} (\Omega;\R^{d}) \times L^p(\Omega;\R^{d\times d}) \to (-\infty,\infty]$ by
\begin{equation}
\label{stored-energy-rpam} \calI_\rpam(t,\varphi,P) := \calE_1 (P) + \calI_2
(t,\varphi,P) + \rpam \calI_3 (\varphi),
\end{equation}
with
$\calE_1$ and $\calI_2$ from \eqref{ene1} and \eqref{I2}, respectively, and $\calI_3: W^{1,\qphi} (\Omega;\R^{d})  \to [0,\infty]$ defined by  
\begin{equation}
\label{calI.3}
\calI_3 (\varphi): =\left\{
\begin{array}{ll}
 \int_\Omega \widetilde{W} (x,\nabla \varphi(x)) \dd x & \text{if }  \widetilde{W}(\cdot,\nabla \varphi) \in L^1 (\Omega),
 \\
 \infty & \text{otherwise}.
 \end{array}
 \right.
\end{equation}
Accordingly, the set of admissible deformations is now
\begin{equation}
\label{calF-reg}
\rcalF:= \calF \cap \{ \varphi \in W^{1,\qphi}(\Omega;\R^d)\, : \
\widetilde{W}(\cdot,\nabla \varphi) \in L^1 (\Omega)\}. 
\end{equation}
We impose that the elastic energy $\widetilde{W}  $ 
in \eqref{calI.3} 
is bounded from below, that 
\begin{equation}
\label{w2-tilde} \tag{5.$\widetilde{\mathrm{W}}_1$}
\text{for all } x \in \Omega \text{ the functional } \widetilde{W}
(x,\cdot) : \R^{d \times d} \to [0,\infty] \text{ is
  polyconvex}, 
\end{equation}
and that $\widetilde W$ controls $W$ in the following
sense
\begin{equation}
\label{w1-tilde} \tag{5.$\widetilde{\mathrm{W}}_2$}
\begin{aligned}
\forall\, S>0 \ \exists \,  C_1^S>0  \  &
\forall\, x \in \Omega \
 \forall\, F, \, P \in  \R^{d \times d} \text{ with } |P| + |P^{-1}| \leq S\, : 
\\
& |W(x,F P^{-1})|^{p'} \leq  C_1^S  (\widetilde{W} (x,F) +1).
\end{aligned}
\end{equation}
Recall that $p'= \frac{p}{p-1}$ is the conjugate exponent to $p$: in
fact, \eqref{w1-tilde} is tuned to the coercivity/growth properties of
the dissipation metric $\rmetric$.

\begin{example}\slshape
\label{ex-tildeW}
Condition \eqref{w1-tilde} is satisfied if, for example, in addition
to \eqref{w1}--\eqref{w3} the functional $W$ fulfills for some $q_W>1$
 the upper estimate 
\begin{subequations}
\label{exs-W}
\begin{equation}
\label{exs-W-1}
\exists\, C_6>0  \ \forall\, x \in \Omega \  \forall\, F \in \GLD \, :
\ \ W(x,F) \leq C_6 (|F|^{q_W} + |F^{-1}|^{q_W}+1), 
\end{equation}
 whereas  $\widetilde W$ is $\infty$ on $\R^{d\times
  d}{\setminus} \GLD$
  and satisfies the coercivity estimate 
\begin{equation}
\label{exs-W-2}
\exists\, C_7, \, C_8 >0  \ \forall\, F \in \GLD \, : \ \
\widetilde{W}(x,F) \geq C_7  (|F|^{p'q_W} + |F^{-1}|^{p'q_W}) - C_8. 
\end{equation}
\end{subequations}
\end{example}

\subsubsection*{The regularized reduced energy}
We now introduce the energy functional obtained by minimizing out the
deformations from $\calI_\rpam$, viz.
\begin{equation}
\label{reduced-energy-eps} \en_\rpam (t,P) : = \inf\{
\calI_\rpam(t,\varphi,P)\, : \ \varphi \in \rcalF\} \quad \text{for }
(t,P) \in [0,T]\times X,  
\end{equation}
with $X = L^p(\Omega;\R^{d\times d})$.
Observe that $\mathrm{dom} (\en_\rpam) = \mathrm{dom} (\en) = [0,T]
\times \domain$, with $\domain$ fulfilling \eqref{domain}. We also
consider the reduced energy
\begin{equation}
\label{E2-rpam}
\calE_{2,\rpam}(t,P) : = \inf\{\calI_2(t,\varphi,P) + \rpam  \calI_3
(\varphi)\, : \ \varphi \in \rcalF \}, 
\end{equation}
so that $\en_\rpam(t,P) = \calE_1 (P) + \calE_{2,\rpam}(t,P)$.

\subsubsection*{The marginal subdifferential} 
Let $\diffname_\rpam : [0,T]\times X 
\rightrightarrows X^*$  be given by
\begin{equation}
\label{regularized-diff}
\diffe {\rpam}tP= \{\rmD_P \calI_\rpam  (t,\varphi,P)\, : \ \varphi
\in \minse {\rpam}{t}P\},
\end{equation}
where  $\minse \rpam tP$  is the set of minimizers for $\calI_\rpam
(t,\cdot, P) $ over $\rcalF$. Since the functional $\calI_3$ does not
depend on $P$, we have
\[
\diffe {\rpam}tP= \{\rmD_P \calI  (t,\varphi,P)\, : \ \varphi \in
\minse {\rpam}{t}P\} = \rmD \en_1(P) + \{\Backx{\cdot}{\nabla\varphi}{P}\,
: \ \varphi \in \minse \rpam tP \}. 
\]
Hence, the doubly nonlinear evolution equation 
\begin{equation}
\label{regularized-dne}
\partial \Psi_P (\dot
P(t)) + \diffe \rpam t{P(t)} \ni 0 \qquad \text{a.e.\ in $(0,T)$}
\end{equation}
 associated with
the generalized gradient system $\grasyse{\rpam}$ (with the
dissipation potentials $(\Psi_P)_{P\in \domain}$ from \eqref{def-Psip}
and $\Ptname_\rpam: \mathrm{graph}(\diffname_\rpam) \to \R$ defined by
analogy with \eqref{def-P-hat}), yields a solution to the
\emph{regularized PDE system}
\begin{subequations}
\label{pde1}
\begin{align}
&
\label{pde1-1}
 \varphi(t) \in \argmin\left \{ \int_{\Omega}
 W(x,\nabla \widetilde \varphi(x)  P^{-1}(t,x)) +\eta  \widetilde{W} (x, \nabla
 \widetilde{\varphi}(x))  \dd x  \, : \ \widetilde{\varphi}
\in \rcalF \right \},  
\\
& 
\label{pde1-2}
0 \in \partial \rmetric(x,\dot{P} \inve{P})\transpi{P} -\Delta_{\qg}P
+ \rmD K(P) + \Backx{x}{\nabla \varphi}{P} \ \text{ in } \Omega,
\end{align}
\end{subequations}
where \eqref{pde1-2} is supplemented with homogeneous
Neumann boundary conditions  for $P$  on $\partial\Omega$.

\subsubsection*{Existence for the regularized system}
Our next result states the existence of  EDB solutions for  the
Cauchy problem associated with  system \eqref{pde1}.

\begin{theorem}[EDB solutions for viscoplasticity] 
\label{th:EDB.Viscopl} Assume
\eqref{kappa2}--\eqref{kappa1}, \eqref{loading},
\eqref{w1}--\eqref{w3}, \eqref{R1}--\eqref{R3}, and
\eqref{w2-tilde}--\eqref{w1-tilde}.  Suppose that $\eta>0 $ and
that the exponents $ \qp,$ $\qf, $ $ \qg,$ and $ \qga$ in
\eqref{form-of-h}, \eqref{kappa1}, \eqref{loading}, and \eqref{w1}
comply with \eqref{index-1} and \eqref{index-2}.

Then, for every initial datum $P_0$ as in \eqref{condition-datum}
there exist $P \in L^\infty (0,T;W^{1,\qg}(\Omega;\R^{d\times d}))
\cap W^{1,p} (0,T; L^p(\Omega;\R^{d \times d}))$ and
$\varphi \in L^\infty (0,T;W^{1,\qphi}(\Omega;\R^{d}))$ solving the
Neumann boundary value problem for system \eqref{pde1} and complying
with the initial condition $P(0,\cdot) = P_0(\cdot)$ a.e.\ in
$\Omega$.

Moreover, the function $t \mapsto \calI_\rpam(t,P(t))$ is absolutely
continuous on $[0,T]$ and, setting 
\[
\Xi(t,x) =-\Delta_{\qg} P(t,x )  + \rmD K( P(t,x))
 + \Backx{x}{\nabla\varphi(t,x)}{P(t,x)}, 
\]
we have $\Xi \in L^{p'} (0,T;L^{p'}(\Omega;\R^{d \times d}))$, and the
pair $(P,\Xi)$ satisfies for all $0 \leq s \leq t \leq T$ the
energy-dissipation balance
\begin{align}
\nonumber
 \calI_\param(t,\varphi(t),P(t)) + \int_s^t \!\int_{\Omega}\!  & \left(
   \rmetric(x,\dot{P}(r,x) 
\inve{P(r,x)} ) + \rmetric^*(x,-\Xi(r,x) \transp{P(r,x)}
)\right)  \!\dd x  \dd r 
\\ &  \label{concrete-enid-rpam}
=
\calI_\param(s,\varphi(s),P(s)) -\int_s^t \!
\pairing{}{W^{1,\qphi}\OOd}{\dell(r)}{\varphi(r)}\dd r.
\end{align}
\end{theorem}

\subsection{Proof of Theorem \ref{th:EDB.Viscopl} }
\label{su:6.2}
In what follows, we will verify that the regularized energy
$\en_\param$ complies with the conditions of Corollary
\ref{cor:abstract-1}, again omitting to denote the $x$-dependence of
$W$ and $R$, and the related quantities, in the proofs of the various
lemmas ahead.

The analogs of Lemmas \ref{l:prelim0}, \ref{l:prelim1}, and
\ref{l:prelim2} hold for the regularized energy $\en_\param$, which
therefore complies with \eqref{basic-1}, \eqref{eq:17-bis}, and
\eqref{eq:diffclass_a}.  We now turn to examining the marginal
subdifferential $\diffname_\rpam$. With the following result we prove
that $ \diffe \rpam tP$ contains the Fr\'echet subdifferential
$\partial \en_\rpam (t,P) $ at every $(t,P) \in [0,T]\times \domain $.
Lemma \ref{l:prelim4-reg} below also features an
$L^{p'}(\Omega;\R^{d\times d})$-estimate of the backstress
contribution $B$ to $ \diffe \rpam tP $ (cf.\ \eqref{even-better2-eps}
ahead, to be compared with estimate \eqref{even-better2} in Lemma
\ref{l:prelim4}), and a sort of \emph{uniform superdifferentiability}
property for the regularized energy $\en_\rpam$, cf.\
\eqref{unif-superdiff} ahead. Once again, for its proof we resort to
Lemma \ref{l:estimates-pointwise}.  Estimate \eqref{unif-superdiff}
will play a crucial role in the proof of the CRI \eqref{eq:48strong},
cf.\ Lemma \ref{l:prelim6} ahead.

\begin{lemma}[Properties of the marginal subdifferential 
$\diffname_\param $ for $\en_\param$]\label{l:prelim4-reg} 
Assume \eqref{kappa2}--\eqref{kappa1},
\eqref{loading}, \eqref{w1}--\eqref{w3},
\eqref{index-1}--\eqref{index-2}, and
\eqref{w2-tilde}--\eqref{w1-tilde}.  Then, we have the
following:\medskip

\noindent
1. For every $(t,P) \in [0,T]\times \domain  $ and
$\phimin \in \minse \rpam tP$ we have  $ \Back{\nabla\varphi}{P} \in
L^{p'}(\Omega;\R^{d\times d})$, and
\begin{equation}
\label{even-better2-eps}
\begin{aligned}
\exists\,   \tilde{c}_7>0  \  \  & \forall\,(t,P) \in [0,T]{\times} \domain 
\ \forall\, \phimin \in  \minse \rpam tP:   
 \\ & 
\|\Back{\nabla\varphi}{P} \|_{L^{p'}} \  \leq \ \tilde{c}_7
\|\inve{P} \|_{L^\infty}  \left( \mathcal{G}_\param(P){+}1
\right) \,. \end{aligned} 
\end{equation}\medskip

\noindent
2. For every $(t,P)\in [0,T]\times \domain$
we have $\partial \en_\rpam (t,P) \subset \diffe \rpam tP$.\medskip

\noindent 3.  There exists $\theta \in (0,1)$ and for all $S>0$ there
exist constants $C_3^S>0$ and $r^*>0$ such that for every $P_1,\, P_2
\in \domain$ such that
\begin{subequations}
\label{unif-superdiff}
\begin{equation}
\label{conditions-G}
\begin{aligned}
\max\{ \mathcal{G}_\eta(P_1),\, \mathcal{G}_\eta(P_2) \} \leq S   \text{ and } \| P_1{-}P_2\|_{L^p} \leq r^*
\end{aligned}
\end{equation}
and for all $\varphi_1 \in \minse \rpam {t}{P_1}$ there holds
\begin{equation}
\label{used-ch-rule-magari}
\begin{aligned} 
\en_{2,\rpam} (t,P_2) - \en_{2,\rpam} (t,P_1)  - \int_\Omega 
\Back{\nabla\varphi_1}{P_1}  {:} (P_2  {-}P_1)   \dd x  
  \ \leq \   C_2^S \|P_2 {-} P_1 \|_{L^p}^{2-\theta}\,.
\end{aligned}
\end{equation}%
\end{subequations}
\end{lemma}
\begin{proof} 
 \textbf{Ad (1):} The proof follows using the argument developed for
 \eqref{inclu-diff}.\medskip 

\noindent\textbf{Ad (2):}
For the proof of \eqref{even-better2-eps}, it is sufficient to combine
\eqref{w3}(ii)  and \eqref{w1-tilde}, viz.%
\begin{equation}
\label{before-too-bad}
\begin{aligned}
&\int_{\Omega}  |\Back{\nabla\varphi(x)}{P(x)} |^{p'}  \dd x
\ \leq \   C  
\int_{\Omega} \left(W(\nabla \varphi(x) \inve{P(x)}) +1 \right)^{p'}
|\transpi{P(x)}|^{p'} \dd x  
\\  
&\leq C   \int_{\Omega} \left(\widetilde{W}(\nabla \varphi(x)) + 1  \right)
|\transpi{P(x)}|^{p'} \dd x 
\  \leq C  \
 \| \inve{P}\|_{L^\infty (\Omega;\R^{d\times d})}^{p'}  \left(\en_{2,\rpam} (t,P) +
1 \right).
\end{aligned}
\end{equation}\medskip

\noindent\textbf{Ad (3):}
 Let $P_1,\, P_2 \in \domain $ comply with \eqref{conditions-G}  and let   $\varphi_1\in M_\rpam (t,P_1)$. 
We have
\begin{equation}
\label{too-bad-1}
\begin{aligned}
&
\en_{2,\rpam} (t,P_2) - \en_{2,\rpam} (t,P_1)
\\
& \quad   \leq
\int_{\Omega} \left( W\argut{\varphi}{P_2} {-} W\argut{\varphi}{P_1} {-} \Back{\nabla\varphi}{P_1}(P_2{-}P_1) \right) \dd x
\\ & \quad 
\stackrel{(1)}{\leq} C_W \int_\Omega \left( W\argut{\varphi}{P_1}{+} 1 \right) |P_1{-}P_2|^2 \dd x 
\\
& \quad 
\stackrel{(2)}{\leq} C_W \| W\argut{\varphi}{P_1}{+} 1\|_{L^{p'}} \|P_1{-}P_2\|_{L^p} \|P_1{-}P_2\|_{L^\infty}
\\
&
 \quad 
\stackrel{(3)}{\leq} C \| \widetilde{W} (\nabla \varphi_1) {+} 1\|_{L^{1}}^{1/p'} \|P_1{-}P_2\|_{L^p}^{2-\theta} \|P_1{-}P_2\|_{W^{1,\qg}}^\theta
%
\quad \stackrel{(4)}{\leq} \ C \|P_1{-}P_2\|_{L^p}^{2-\theta}\,.
\end{aligned}
\end{equation}
Here, {\footnotesize (1)} is derived via estimate \eqref{key-Alex-2}
from Lemma \ref{l:estimates-pointwise} as follows. First we use
\eqref{conditions-G} giving the upper bounds 
\begin{equation}
\label{conse-of-conse}
 \|P_1\|_{W^{1,\qg}} + \|P_2\|_{W^{1,\qg}} \leq C_3^S
\end{equation}
by the analog of 
\eqref{conse-2}. Using $\qg>d$ the Gagliardo-Nirenberg inequality
\begin{equation}
\label{GN-now-used}
  \|P_1{-}P_2\|_{L^\infty} \leq C_{\mathrm{GN}} \|P_1{-}P_2\|_{W^{1,\qg}}^\theta  
  \|P_1{-}P_2\|_{L^p}^{1-\theta} \quad \text{with } \theta = 
  \frac{d\qg}{\qg p-dp+d\qg} \in (0,1) 
\end{equation}
provides smallness of $\|P_1{-}P_2\|_{L^\infty}$. Indeed, choosing
$r^* = \tfrac12 \left( \tilde{r} /(2C_3^S)^\theta
\right)^{1/(1{-}\theta)}$ with $\tilde r$ from \eqref{key-Alex-2}, we
find $\|P_1{-}P_2\|_{L^\infty} \leq (2C_3^S)^\theta (r^*)^{1-\theta}
<\tilde {r}$, such that \eqref{key-Alex-2} is applicable.

Estimate {\footnotesize (2)} is H\"older's
inequality, while to obtain {\footnotesize (3)} we use estimate
\eqref{w1-tilde}, observing that $\|P_1\|_{L^\infty} + \|
\inve{P}_1\|_{L^\infty} \leq C$ due to \eqref{conse-of-conse} and the
continuous embeddings $W^{1,\qg}(\Omega;\R^{d\times d}) \subset
L^\infty(\Omega;\R^{d\times d})$ and
$W^{1,\tilde{q}}(\Omega;\R^{d\times d}) \subset
L^\infty(\Omega;\R^{d\times d})$. We also again resort the
Gagliardo-Nirenberg inequality \eqref{GN-now-used}. Finally,
{\footnotesize (4)} ensues from the energy bound $S \geq
\mathcal{G}_\eta(P_1) \geq c \| \widetilde{W} (\nabla \varphi_1) {+}
1\|_{L^{1}} - C$ by the analog of \eqref{esti-2}, and again from the
bound \eqref{conse-of-conse}. We have thus established
\eqref{used-ch-rule-magari}. 
\end{proof}

With the very same argument as in the proofs of Lemmas
\ref{l:prelim4-5} and \ref{l:prelim5} it is possible to check
conditions \eqref{e:ass-p-a} and \eqref{eq:468},  which implies  the
variational sum rule \eqref{eq:42-bis}.  We now check the CRI
\eqref{eq:48strong}.

\begin{lemma}[Chain-rule inequality for the regularized energy $\en_\param>0$]
  \label{l:prelim6} Assume that \eqref{kappa2}--\eqref{kappa1},
  \eqref{loading}, \eqref{w1}--\eqref{w3},
  \eqref{index-1}--\eqref{index-2}, and
  \eqref{w2-tilde}--\eqref{w1-tilde} hold.  Then, the triple
  $(\en_\rpam, \diffname_\rpam, \Ptname_\rpam)$ fulfills the CRI
  \eqref{eq:48strong}.
\end{lemma}
\begin{proof} Let us fix a curve $ P\in \AC ([0,T];L^p(\Omega;\R^{d\times
  d}))$ and a function $\Xi \in L^1 (0,T;L^{p'}(\Omega;\R^{d\times
  d}))$ fulfilling \eqref{conditions-1}.  Taking into account
\eqref{esti-2}, \eqref{conse-2}, and \eqref{enhanced-coercivity-psi},
we have a fortiori that
\begin{equation}
\label{additional-regu}
\begin{aligned}
&
   P \in L^\infty (0,T;
W^{1,{\qg}}(\Omega;\R^{d\times d})) \cap W^{1,p}(0,T; L^p
(\Omega;\R^{d\times d})), 
\\
 &   \inve{P} \in L^\infty (0,T;
W^{1,{\tilde{q}}}(\Omega;\R^{d\times d})), \\  &   \Xi \in L^{p'}
(0,T;L^{p'}(\Omega;\R^{d\times d}))
\end{aligned}
\end{equation}
with 
\begin{equation}
\label{selection-xi} \Xi(t) = -\Delta_{\qg} P(t) + \rmD K(P(t)) +
\Back{\nabla\varphi(t)}{P(t)}  \quad \foraa   t \in (0,T).
\end{equation}
Here, $t \mapsto \phimin(t) \in \minse \rpam{t}{P(t)}$ is a measurable
selection.  Each of the three terms on the right-hand side of
\eqref{selection-xi} is in $L^{p'} (0,T;L^{p'}(\Omega;\R^{d\times d}))
$ and in fact we have
\begin{equation}
\label{quoted-later}
\begin{aligned}
\| {-}\Delta_{\qg} P\|_{L^{p'}(0,T;L^{p'}(\Omega;\R^{d\times d}))} &  + \|
\rmD K(P) \|_{L^\infty(0,T;L^\infty(\Omega;\R^{d\times d}))}\\ &  +
\| \Back{\nabla\varphi}{P}   \|_{L^\infty(0,T;L^{p'} 
  (\Omega;\R^{d\times d}))} \leq C. 
\end{aligned}
\end{equation}
The estimate for $\Back{\nabla\varphi}{P} $ follows from
\eqref{even-better2-eps} using $\sup_{t \in [0,T]} \en_\rpam
(t,P(t))<\infty$. The bound for $\rmD K(P)$ is due to the fact that
 \begin{equation}
\label{continuity}
P\in \mathrm{C}^0 ([0,T];\mathrm{C}^0
(\overline{\Omega};\R^{d\times d})),  \text{ which in particular
  implies } \det(P(t,x)) \geq \pi >0, 
\end{equation}
(where we have used the compact embedding
$W^{1,{\qg}}(\Omega;\R^{d\times d})\Subset \mathrm{C}^0
(\overline{\Omega};\R^{d\times d}) $), combined with the fact
that $\rmD K$ is continuous on $\GLD$.  Then, the estimate for
$-\Delta_{\qg} P$ ensues from a comparison argument, using $\Xi \in
L^{p'}(0,T;L^{p'}(\Omega;\R^{d\times d}))$.

Now, the chain rule for $\en^1$, given by the sum of a convex and of a
Fr\'echet differentiable functional, yields that
\begin{equation}
\label{ch-rule-e1}
\begin{aligned}
&
\text{the map } t \mapsto \en^1 (t,P(t)) \text{ is absolutely
  continuous and } 
\\
& \frac{\dd}{\dd t }\en^1(t,P(t)) = \int_\Omega (-\Delta_{\qg} P(t) +
\rmD K(P(t))  ){:} \dot P(t) \dd x \quad \foraa  t \in (0,T). 
\end{aligned}
\end{equation}
Let us now prove that the map $t \mapsto \en_{2,\rpam} (t,P(t))$ is
absolutely continuous, by estimating the difference $| \en_{2,\rpam}
(t, P(t)) - \en_{2,\rpam} (s, P(s))|$. We will use that
$\sup_{t\in [0,T]} \mathcal{G}_\rpam (P(t)) \leq S$.  Without loss of
generality, we may also suppose that $\|P(t) {-} P(s)\|_{L^p} \leq
r^*$ with $r^*>0$ from \eqref{conditions-G}.  We have
\begin{equation}
\label{later-4-diff-quot}
\begin{aligned}
 &  \en_{2,\rpam} (t, P(t))  - \en_{2,\rpam} (s, P(s))
\\ &  =  \en_{2,\rpam} (t, P(t))  - \en_{2,\rpam} (t, P(s))+  \en_{2,\rpam} (t, P(s))  - \en_{2,\rpam} (s, P(s))
 \\ &
  \leq \en_{2,\rpam}(t, P(t))  - \en_{2,\rpam} (t, P(s)) + \calI_2 (t,\varphi(s), P(s))  - \calI_2 (s,\varphi(s), P(s))
 \\ &
  \stackrel{(1)}{\leq}
  \int_\Omega   \Back{\nabla\varphi(s)}{P(s)}  {:}(P(t){-}P(s))   \dd
  x   
  \\
  & \qquad
  +   C_2^S \| P(t)  {-} P(s)\|_{L^p\OOdd}^{2-\theta}   -\!
  \pairing{}{W^{1,\qphi}\OOd}{\ell(t){-}\ell(s)}{\varphi(s)},  
\end{aligned}
\end{equation}
where  $\varphi(s)$ is a selection in  $\minse\rpam{s}{P(s)}$, and 
for {\footnotesize (1)}
we have used
estimate \eqref{used-ch-rule-magari}. 
 Exchanging the
role of $s$ and $t$,
 we thus conclude, for every $0 \leq s \leq t
\leq T$, the estimate 
\begin{equation}
\label{2abscont}
\begin{aligned}
 & |\en_{2,\rpam} (t, P(t))  - \en_{2,\rpam} (s, P(s)) |  \\ &
  \leq  \|P(t) {-} P(s) \|_{L^p \OOdd}
 \left(  C_2^S  \|P(t) {-} P(s) \|_{L^p \OOdd}^{1-\theta}  +
\sup_{t \in [0,T]} \|    \Back{\nabla\varphi(t)}{P(t)}  \|_{L^{p'}\OOdd } \right)
\\ & \quad
+ \| \ell(t) {-}\ell(s) \|_{W^{1,\qphi}(\Omega;\R^d)^*}  (\|
\varphi(t)\|_{W^{1,\qphi}\OOd} + 
\| \varphi(s)\|_{W^{1,\qphi}\OOd} )
\\ &
\stackrel{(2)}{\leq} C \left( \|P(t) {-} P(s) \|_{L^{p}\OOdd} + |t{-}s| \right)\,.
\end{aligned}
\end{equation}
Indeed, for {\footnotesize (2)} we have used condition \eqref{loading}
on $\ell$, the coercivity property \eqref{esti-1} for $\calI_2$, and
combined \eqref{even-better2-eps} with estimate \eqref{conse-2}, so
that
\[
  \sup_{t \in [0,T]} \|    \Back{\nabla\varphi(t)}{P(t)}
  \|_{L^{p'}\OOdd } \leq \tilde{c}_7 \sup_{t \in [0,T]}  \|
  \inve{P(t)}\|_{L^\infty} ( \calG_\eta(P(t)) {+} 1) \leq C. 
\]
Thus, from \eqref{2abscont} and the fact that $P \in \AC
([0,T];L^p(\Omega;\R^{d\times d}))$ we infer that $t \mapsto
\en_{2,\rpam} (t,P(t))$ is absolutely continuous on $[0,T]$.

In order to prove the CRI \eqref{eq:48strong},
let us fix $t \in (0,T)$ outside a negligible set such that formula
\eqref{ch-rule-e1} holds and that for $h\down 0$ we have 
\begin{equation}
\label{choice-of-t}  \left\{
\begin{array}{ll}
 \exists \, \frac{\rmd}{\rmd t}\mathcal{E}_{2,\rpam}(t, P(t)),
 \\[0.4em]
 \exists\, \dot P(t)  \text{ and }  \frac1h (P(t{+}h) {-}P(t)) \to
 \dot P(t) \text{  in } L^p ( \Omega;\R^{d\times d}), 
 \\[0.4em]
 \exists \, \dell(t) \text{ and }  \frac1h (\ell(t{+}h) {-}\ell(t)) \to
 \dell(t) \text{ in } W^{1,\qphi} (\Omega;\R^{d})^* .
\end{array}
\right.
\end{equation}
We now use estimate \eqref{later-4-diff-quot} for $\tilde\varphi
\in \Rtname_\rpam (t,P(t), \Xi (t))$, divide by $h<0$, and obtain
\[
\begin{aligned}
 & \frac{1}h \left( \en_{2,\rpam} (t{+}h, P(t{+}h)) - \en_{2,\rpam}(t,
   P(t))\right) 
\\
& \geq  \frac1h  \int_\Omega    \Back{\nabla\tilde{\varphi}(t)}{P(t)}  {:} (P(t{+}h)
 {-} P(t)) \dd x +C_2^S \frac{\|P(t{+}h) {-} P(t)
  \|_{L^p \OOdd}^{2-\theta} }{h}
  \\
  & \qquad  -\frac1h \!\pairing{}{W^{1,\qphi}\OOd}{ 
   \ell(t{+}h){-}\ell(t)}{\tilde\varphi}.
   \end{aligned}
\]
Taking the limit as $h \up 0$ in the above inequality, we observe that
the second term on the right-hand side converges to $0$ by
the second of \eqref{choice-of-t} giving  $\| P(t{+}h){-}P(t)\|\leq 2
  \|h\dot P(t)\|$ for small $h$ and by $2{-}\theta>1$.  
Thus, for all $\tilde\varphi \in \Rtname_\rpam (t,P(t), \Xi (t))$
we have
\[
\frac{\rmd}{\rmd t}\mathcal{E}_{2,\rpam} (t, P(t)) \geq  \int_\Omega
 \Back{\nabla\tilde{\varphi}}{P(t)} {:} \dot P(t)   \dd x  
- \pairing{}{W^{1,\qphi}\OOd}{\dell(t)}{\tilde\varphi}.
\]
It follows from the definition of $ \Rtname_\rpam (t,P(t), \Xi (t))$ that
\[
\begin{aligned}
\Back{\nabla\tilde{\varphi}_1}{P(t)}   =
 \Back{\nabla\tilde{\varphi}_2}{P(t)}  \quad \text{ for all } \tilde\varphi_1, \,
\tilde\varphi_2 \in \Rtname_\rpam (t,P(t), \Xi (t)).
\end{aligned}
\]
Therefore, we have
\begin{equation}
\label{ch-rule-e2}
\frac{\rmd}{\rmd t}\mathcal{E}_{2,\rpam} (t, P(t)) \geq  \int_\Omega
\Back{\nabla\varphi(t)}{P(t)}  {:} \dot P(t)   \dd x  
+   \Ptname_\rpam  (t,P(t),\Xi(t)),   
\end{equation}
for every selection $t \in (0,T) \mapsto \varphi(t) \in \Rtname_\rpam
(t,P(t), \Xi (t))$.  Combining \eqref{ch-rule-e1} and
\eqref{ch-rule-e2} yields the CRI \eqref{eq:48strong}.
\end{proof}

\begin{remark}\slshape
\label{rmk:YM}
It is natural to wonder whether EDB solutions to the doubly nonlinear
evolution equation \eqref{regularized-dne} converge to a solution of
the generalized gradient system driven by the original energy $\calE$
from \eqref{reduced-energy}, as the regularizing parameter $\rpam$
vanishes.  This is far from guaranteed.  Indeed, mimicking the
variational arguments that we have employed for passing to the
time-continuous limit in the proof of Theorem \ref{thm:abstract-2}, it
should be possible to show that EDB solutions to
\eqref{regularized-dne} converge to a solution of the generalized
gradient system driven by the energy functional
$ 
\widetilde{\calE}(t,P):= \inf\{ \calI_{0}(t,\varphi,P)\, : \, \varphi \in \widetilde{\calF}\},
$ 
with $ \calI_{0}(t,\cdot,P)$ the $\Gamma$-limit as $\eta\down 0$ of $ \calI_{\eta}(t,\cdot,P)$
and $ \widetilde{\calF}$ 
 the set  of admissible deformations from \eqref{calF-reg}. 
Note that $\widetilde{\calE}$ need not coincide with $\calE$, as there might be a Lavrentiev phenomenon. 
\end{remark}

\section{Extensions}
\label{s:7}
In this section we briefly discuss directions  in which our analysis
could be extended.

First, as in \cite{MaiMie09GERI} we might easily couple the
evolution of the plastic variable $P$ with the (rate-dependent)
evolution of some other hardening variable $p \in \R^m$, with $m \geq
1$.

Second, as outlined in \cite{MaiMie09GERI}, it would be
possible to encompass in our model also a global version of the
non-self-interpenetration condition, as proposed in
\cite{Ciar82IANO}. For this, it would be sufficient to replace the
space of admissible deformations $\calF$
 by 
\begin{equation}
\label{ciarlet-necas}
\begin{aligned}
\calF_{\mathrm{nsi}} = \{ \varphi \in W^{1,\qphi}(\Omega;\R^d)\, : \ & \varphi  = \varphi_{\mathrm{Dir}}  \text{ on }\Gamma_{\mathrm{Dir}},  \  \det \nabla \varphi \geq 0 \text{ a.e.\ in } \Omega, \\
& \int_\Omega \det (\nabla \varphi) \dd x \leq \mathrm{vol}(y(\Omega))  \}.
\end{aligned}
\end{equation}
Since $\calF_{\mathrm{nsi}}$ is weakly closed in
$W^{1,\qphi}(\Omega;\R^d)$ if $\qphi>d$, it is possible to define our
reduced energy by minimizing out from the stored energy the
deformations in $\calF_{\mathrm{nsi}} $. Then, all our results carry
over to this case.\medskip
 
In what follows, we will focus more specifically on the extension to
time-dependent Dirichlet loadings.
%
 To replace the time-independent  Dirichlet condition
$\varphi(t,x) = \varphi_{\mathrm{Dir}}(x)$  for $(t,x) \in [0,T] \times \Gamma_{\mathrm{Dir}}$  by 
\begin{equation}
\label{gdir-cond}
\varphi(t,x) = \gdir (t,x) \quad (t,x) \in [0,T] \times \Gamma_{\mathrm{Dir}},
\end{equation}
with $\gdir :[0,T] \times \Gamma_\dir \to \R^d$ given, we  follow the ideas from
\cite[Sec.\ 5]{FraMie06ERCR}. 

For this,  we will suppose that $\gdir$ can be extended to $[0,T]\times \R^d$ and search for the deformation $\varphi :[0,T]\times \Omega \to \R^d$ in the form of 
  the  composition
 \begin{equation}
 \label{multiplic-split}
 \varphi(t,x) = \gdir (t, y(t,x)) \quad \text{with } y(t,\cdot) \in \calfid:= \{ y \in W^{1,\qphi}(\Omega;\R^d)\, :  \ y =\mathrm{Id} \text{ on } \Gamma_\dir \}.
 \end{equation}
In fact, underlying \eqref{multiplic-split} is the implicit idea  that $\gdir (t,\cdot) : \R^d \to \R^d$ is a diffeomorphism. Therefore, as in 
 \cite[eqn.\ (5.7)]{FraMie06ERCR} we  require that $\gdir$  satisfies 
 \begin{equation}
 \label{conds-g}
 \begin{aligned}
 &
 \gdir \in \rmC^1 ([0,T]\times \R^d;\R^d), \quad \nabla \gdir \in \mathrm{BC}^1  ([0,T]\times \R^d;\R^{d\times d} ), \\ 
 &
 \exists\, C_9>0 \ \  \forall\, (t,x) \in [0,T]\times \R^d\,:  \ \ 
 \ |\nabla \gdir (t,x)^{-1}| \leq C_9,
 \end{aligned}
\end{equation} 
where $\mathrm{BC}$ stands for \emph{bounded continuous}. 
 
We now rewrite the stored and the reduced energy functionals in terms
of the variable $y$, taking into account that the composition
\eqref{multiplic-split} leads to the following multiplicative split
for the deformation gradient
\begin{equation}
 \label{split}
 \nabla \varphi(t,x) = \nabla \gdir (t,y(t,x)) \nabla y(t,x) \quad
 \text{for every } (t,x) \in [0,T] \times \Omega. 
\end{equation}
With a slight abuse of notation, we will continue to use the symbols
$\calI$ and $\calE$ for
\begin{align}
 &\label{stored-y}
 \begin{aligned}
 & \calI(t,y,P):= \en_1(P) + \calI_2(t,\varphi, P) \quad \text{with }
 \\
 & \quad 
  \calI_2(t,\varphi, P):  = \left\{
  \begin{array}{lll}
  \begin{array}{@{}ll}
  \int_{\Omega} W(x,\nabla \gdir(t, y)\nabla y  \inve{P})   \dd x  
  \\
  \qquad \quad -
  \pairing{}{W^{1,\qphi}\OOd}{\ell(t)}{\gdir(t,y)}
  \end{array}
   & \text{if $(y,P) \in \calfid \times \calP$,}
  \\
  \infty  & \text{otherwise,}
  \end{array} \right.
 \end{aligned}
\\
& \label{reduced-y}
\begin{aligned}
&
\ene tP:= \inf\{ \calI(t,y,P)\, : \ y \in \calfid\} = \en_1(P) +
\en_2(t,P) 
\\
&
\quad \quad \text{with } \en_2(t,P):=  \inf\{ \calI_2(t,y,P)\, :
\ y \in \calfid\}. 
\end{aligned}
\end{align} 
Clearly, the marginal subdifferential of $\en$ is now given by
\begin{equation}
 \label{diff-dir}
  \diff tP :=  -\Delta_{\qg} P +
 \rmD K(P) +\left\{ \Backx{\cdot}{\nabla \gdir(t,y)\nabla y}P \, : \ y
   \in \mins tP\right\}, 
\end{equation}
where we still use the notation $\mins tP$ for $\mathrm{Argmin}_{y
  \in\calfid} \calI (t,y,P)$.
 
In what follows, we will discuss the conditions the gradient system
$\grasys$, with $X = L^p(\Omega;\R^{d\times d})$, needed to obtain the
abstract conditions \eqref{basic-1}--\eqref{eq:468}.  The lower
semicontinuity and coercivity \eqref{basic-1} and \eqref{eq:17-bis}
can be checked by suitably adapting the proofs of Lemmas
\ref{l:prelim1} and \ref{l:prelim2}, cf.\ also the arguments in the
proofs of \cite[Lemma 5.5]{FraMie06ERCR}.  We now closely examine the
time dependence of $\en$ and the properties of the power function
$\Ptname$.

\paragraph*{The power of the external loadings}
In order to check conditions \eqref{eq:diffclass_a} and
\eqref{e:ass-p-a}, we will resort to the explicit calculation of the
power $\partial_t \calI(t,y,P)$ from \cite[Lemma 5.5]{FraMie06ERCR}.
Therein, it was proved that, in the case of zero volume and surface
loadings (i.e.\ for $\ell =0$)
\begin{equation}
 \label{power-formula-maimie}
 \partial_t \calI (t,y,P) = \partial_t \calI_2(t,y,P) = \int_\Omega
 \inpow{x}{\gdir}{y(x)}{ P(x)^{-1}} : V(t,y(x)) \dd x, 
\end{equation}
where $\Kirchname$ denotes the (multiplicative) \emph{Kirchhoff stress
  tensor}
\begin{equation}
 \label{Kirch}
 \Kirchx {x}F := \partial_F W(x,F) \transp{F}
\end{equation}
and we have used the short-hand notation $V(t,y);= \nabla \dot{g}_\dir (t,y) (\nabla \gdir
(t,y))^{-1} $.  Formula
\eqref{power-formula-maimie} was established under stress control
conditions for $\Kirchname$ of the same type as the ones in
\eqref{w3}, namely
\begin{equation}
\label{w3-K}
\tag{6.$\mathrm{W}_{3}$}
\begin{aligned}
& \exists\, \delta>0 \    \ \exists\, \widetilde{C}_4,\, \widetilde{C}_5>0    \ \  \forall\,
(x,F) \in \dom(W) \ \forall\, N \in \mathcal{N}_\delta\,:
\\
&   \text{(i)} \  \text{$W(x,\cdot) : \mathrm{GL}^+(d) \to \R$ is differentiable,} 
\\
& \text{(ii)} \  |  \partial_F W(x,F) \transp{F}  | \leq  \widetilde{C}_4 (W(x,F) +1),
\\
& \text{(iii)} \ | \rmD_F W(x,F)\transp{F} -  \rmD_F
W(x,N F) \transp{(NF)} | \leq  \widetilde{C}_5 |N -\mathbf{1}| (W(x,F) +1).
\end{aligned}
\end{equation}

Therefore, we have the following explicit formula and properties for
$\partial_t \calI$.

\begin{lemma}
  Assume \eqref{loading}, that $W$ is frame-indifferent and that it
  fulfills \eqref{w1}--\eqref{w3} as well as \eqref{w3-K}, and let
  $\gdir$ comply with \eqref{conds-g}. Then, for every $(y,P) \in
  \calfid \times \domain$ we have 
 \begin{equation}
 \label{formula-power}
 \begin{aligned}
   \partial_t \calI (t,y,P) & = \int_\Omega
   \inpow{x}{\gdir}{y(x)}{P(x)^{-1}} : V(t,y(x)) \dd x
   \\
   & \quad - \pairing{}{W^{1,\qphi}\OOd}{\dot{\ell}(t)}{\gdir (t,y)} -
   \pairing{}{W^{1,\qphi}\OOd}{\ell (t)}{\dot{g}_\dir (t,y)}
  \end{aligned}
\end{equation}
and there exist $c_9,\, c_{10}>0$ and a modulus of continuity $\omega$
such that for every $s,\, t \in [0,T]$
 \begin{align}
 & 
 \label{est-mod-der}
 |\partial_t \calI(t,y,P)| \leq c_9 (\calI (t,y,P) + 1),
 \\
 & 
 \label{est-diff-der}
 |\partial_t \calI(t,y,P) - \partial_t \calI(s,y,P)  | \leq \omega
 (|t{-}s|) (\calI (t,y,P) + c_{10}). 
 \end{align}
Hence,  for every $s, t \in [0,T] $ and every $ (y,P)
 \in \calfid \times \domain$ there holds 
\begin{equation}
 \label{simple-Gronwall-argument}
 \calI(t,y,P) + 1 \leq \exp( c_9 |t{-}s|) \left( \calI(s,y,P) + 1 \right)\,.  
\end{equation}
\end{lemma}
\begin{proof} In the case $\ell =0$, i.e.\ when $\partial_t \calI$ is given
by \eqref{power-formula-maimie}, formulae \eqref{est-mod-der} and
\eqref{est-diff-der} were proved in \cite[Thm.\ 5.3]{MaiMie09GERI}.
In view of conditions \eqref{loading}, \eqref{w3-K}, and \eqref{conds-g}, It is easy
to check that they extend to the case when \eqref{formula-power}
holds. Hence, \eqref{simple-Gronwall-argument} follows via a simple
Gronwall argument.  
\end{proof}

Then, the time-dependence estimate \eqref{eq:diffclass_a} is an
immediate consequence of \eqref{simple-Gronwall-argument}, whereas
\eqref{e:ass-p-a} can be checked straightforwardly, also resorting to
\eqref{w3-K}.

\begin{corollary}
\label{cor:time-dep-Dir}
Assume \eqref{loading}, that $W$ is frame-indifferent and fulfills
\eqref{w1}--\eqref{w3}  and let $\gdir$ comply with
\eqref{conds-g}. Then, $\en$ from \eqref{reduced-y} complies with
\eqref{eq:diffclass_a} and \eqref{e:ass-p-a}.
\end{corollary}

To establish the closedness condition \eqref{eq:468}, we exploit
\eqref{conds-g} to check that for $\diffname$ from \eqref{diff-dir}
the analog of Lemma \ref{l:prelim4} holds.  Furthermore, a combination
of the arguments from the proof of Lemma \ref{l:prelim5} with the
techniques from \cite[Prop.\ 5.1, Thm.\ 5.2]{MaiMie09GERI} allows us
to prove that $(\en,\diffname)$ comply with the closedness property in
\eqref{eq:468} (whence the variational sum rule \eqref{eq:42-bis}).

Hence, it remains to check the upper semicontinuity of the
functional $\Ptname: \mathrm{graph}(\diffname) \to \R$.
Preliminarily, we examine the continuity properties of $\partial_t
\calI$. The following result is a consequence of \cite[Prop.\
4.4]{MaiMie09GERI} (see also \cite{FraMie06ERCR}), combined with
\eqref{est-diff-der}.
\begin{lemma}
\label{cont-power}
Assume 
 \eqref{loading}, let $W$ be frame-indifferent and fulfill 
 \eqref{w1}--\eqref{w3} and \eqref{w3-K},  and  let $\gdir$ comply with \eqref{conds-g}. 
 Then, we have 
 \begin{equation}
 \label{power-cont}
 \left.
 \begin{array}{l}
 t_n\to t, \ y_n \to y \text{ in } W^{1,\qphi}(\Omega;\R^d), \\
  P_n \weakto P \text{ in } W^{1,\qg} (\Omega;\R^{d\times d}),
 \\
 \calI(t_n,y_n,P_n) \to  \calI(t,y,P)<\infty 
 \end{array}
 \right\} \ \Longrightarrow \ \partial_t \calI(t_n,y_n,P_n) \to  \partial_t \calI(t,y,P).  
 \end{equation}
\end{lemma}

Therefore, let $( t_n )_n \subset [0,T]$, $(P_n)_n\subset
X = L^{p}(\Omega;\R^{d\times d})$, and $ (\Xi_n)_n \subset
X^* = L^{p'}(\Omega;\R^{d\times d}) $ with $\Xi_n \in \diff{t_n}{P_n}$ for
all $n \in \N$ converge to $t,\, P,\, \Xi$ as in
\eqref{convs-4-closedness}. Thus $\Pt {t_n}{P_n}{\Xi_n} = \partial_t
\calI (t_n,\tilde{y}_n,P_n)$ for some $\tilde{y}_n \in
\Rt{t_n}{P_n}{\Xi_n}$.  The very same arguments as in Step $5$ in the
proof of Lemma \ref{l:prelim5} yield that, up to a further (not
relabeled) subsequence $\tilde{y}_n$ weakly converges in
$W^{\qphi}(\Omega;\R^d)$ to some $\tilde{y} \in \Rt tP{\Xi}$.  Thanks
to Lemma \ref{cont-power} we have
\[
\lim_{n \to \infty} \Pt {t_n}{P_n}{\Xi_n} = \lim_{n \to
  \infty} \partial_t \calI (t_n,y_n,P_n) = \partial_t \calI (t,y,P)
\leq \Pt tP{\Xi}, 
\]
whence the upper semicontinuity of $\Ptname$.  This concludes the
proof of the closedness \eqref{eq:468}.

Finally, combining the arguments in the proof of Lemma \ref{l:prelim6}
with property \eqref{est-diff-der} it can be checked that the CRI
\eqref{eq:48strong} holds for the regularized energy $\en_\rpam$ also
in the case of time-dependent Dirichlet boundary conditions. Thus, the
existence of EDI and EDB solutions follows as in Section \ref{s:5}.

\footnotesize

\newcommand{\etalchar}[1]{$^{#1}$}
\def\cprime{$'$}


\end{document}